\documentclass[bibliography=totoc]{scrartcl}
\usepackage[utf8]{inputenc}
\usepackage[T1]{fontenc}

\usepackage[style=alphabetic]{biblatex}
\AtEveryBibitem{\clearfield{issn}}

\usepackage{graphicx}
\usepackage{amsmath,amssymb,amsthm}
\usepackage{stmaryrd,wasysym}
\usepackage{bm}
\usepackage[shortlabels]{enumitem}
\usepackage{color}
\usepackage{comment}
\usepackage[bookmarks=true, bookmarksnumbered=true, bookmarksopen=true,
    unicode=true, colorlinks=true, linktoc=all,
    linkcolor=blue, citecolor=blue, filecolor=black, urlcolor=blue,
    pdfstartview=FitH]{hyperref}
\usepackage[normalem]{ulem} 
\usepackage{caption}
\DeclareCaptionLabelFormat{sc-parens}{(\textsc{#2})}
\usepackage[labelformat=parens]{subcaption}
\usepackage{stmaryrd} %to use llbracket et rrbracket
\setlist{nosep}

\newtheorem{theorem}{Theorem}[section]
\newtheorem{proposition}[theorem]{Proposition}
\newtheorem{lemma}[theorem]{Lemma}
\newtheorem{corollary}[theorem]{Corollary}

\theoremstyle{definition}
\newtheorem{definition}[theorem]{Definition}

\theoremstyle{remark}
\newtheorem{remark}[theorem]{Remark}

\newcommand{\yy}{\scalebox{.6}{\RIGHTcircle}}

\newcommand{\pd}{\ensuremath{\vec d}}
\newcommand{\dB}{\ensuremath{\vec B}}

\newcommand{\rg}{\ensuremath{\mathrm{g}}}

\ifdefined\rm
  \renewcommand{\rm}{\ensuremath{\mathrm{m}}}
\else
  \newcommand{\rm}{\ensuremath{\mathrm{m}}}
\fi

\newcommand{\rs}{\ensuremath{\mathrm{s}}}
\newcommand{\rc}{\ensuremath{\mathrm{c}}}

\newcommand{\inc}[1]{\ensuremath{\text{inc}(#1)}}
\renewcommand{\v}{\ensuremath{\mathrm{v}}}
\newcommand{\bols}{\ensuremath{\bar w}}

\newcommand{\rV}{\ensuremath{\mathrm{V}}}
\newcommand{\rE}{\ensuremath{\mathrm{E}}}
\newcommand{\rF}{\ensuremath{\mathrm{F}}}
\newcommand{\rFn}{\ensuremath{\mathrm{F}^{\bullet}}}
\newcommand{\rFb}{\ensuremath{\mathrm{F}^\circ}}
\newcommand{\rC}{\ensuremath{\mathrm{C}}}

\newcommand{\cA}{\mathcal{A}}
\newcommand{\cB}{\mathcal{B}}

\newcommand{\blar}{\mbox{\tiny \ensuremath{\circ,\!\geq}}}
\newcommand{\bstr}{\mbox{\tiny \ensuremath{\circ,\!>}}}
\newcommand{\nlar}{\mbox{\tiny \ensuremath{\bullet,\!\geq}}}
\newcommand{\nstr}{\mbox{\tiny \ensuremath{\bullet,\!>}}}
\newcommand{\dmaxb}{\ensuremath{\Delta^{\!\circ}}}
\newcommand{\dmaxn}{\ensuremath{\Delta^{\!\bullet}}}

\newcommand{\Wb}{\ensuremath{W^{\circ}}}
\newcommand{\Wn}{\ensuremath{W^{\bullet}}}

\newcommand{\Z}{\mathbb{Z}}

\newcommand{\Ring}{\mathcal{R}}
\newcommand{\C}{\mathbb{C}}

\title{The slice decomposition of planar hypermaps}
\author{
  Marie Albenque\thanks{\, Université Paris Cité, CNRS, IRIF, F-75013 Paris, France \protect\\
    Email: \texttt{malbenque@irif.fr} } \and
  Jérémie Bouttier\thanks{\, Sorbonne Université and Université Paris Cité, CNRS, IMJ-PRG, F-75005 Paris, France \protect\\
    On leave from: \protect\\
    Universit\'e Paris-Saclay, CNRS, CEA, Institut de physique th\'eorique, 91191, Gif-sur-Yvette, France \protect\\
    Email: \texttt{jeremie.bouttier@imj-prg.fr}}}
\date{\today}

\begin{document}
\maketitle

\begin{abstract}
  The slice decomposition is a bijective method for enumerating planar maps (graphs embedded in the sphere) with control over face degrees. In this paper, we extend the slice decomposition to the richer setting of \emph{hypermaps}, naturally interpreted as properly face-bicolored maps, where the degrees of faces of each color can be controlled separately. This setting is closely related with the two-matrix model and the Ising model on random maps, which have been intensively studied in theoretical physics, leading to several enumerative formulas for hypermaps that were still awaiting bijective proofs.

  Generally speaking, the slice decomposition consists in cutting along geode\-sics. A key feature of hypermaps is that the geodesics along which we cut are \emph{directed}, following the canonical orientation of edges imposed by the coloring. This orientation requires us to introduce an adapted notion of slices, which admit a recursive decomposition that we describe.

Using these slices as fundamental building blocks, we obtain new bijective decompositions of several families of hypermaps: disks (pointed or not) with a monochromatic boundary, cylinders with monochromatic boundaries (starting with trumpets or cornets having one geodesic boundary), and disks with a “Dobrushin” boundary condition. In each case, the decomposition ultimately expresses these objects as sequences of slices whose increments correspond to downward-skip free (\L{}ukasiewicz-type) walks subject to natural constraints.

Our approach yields bijective proofs of several explicit expressions for hypermap generating functions. In particular, we provide a combinatorial explanation of the algebraicity and of the existence of rational parametrizations for these generating functions when face degrees are bounded.
\end{abstract}

\newpage
\tableofcontents

\newpage  

\section{Introduction}

\subsection{Context, motivations, and contributions of this paper}

The enumeration of maps is a classical topic in combinatorics, which
was initiated by Tutte in the sixties, and which has seen many
developments up to now. See for instance the review by
Schaeffer~\cite{Schaeffer15} and the references cited therein. Let us
also mention the book by Lando and Zvonkin~\cite{Lando2004} which
provides an overview of some fascinating connections between map
enumeration and other branches of mathematics, involving avatars of
maps such as \emph{dessins d'enfants} and combinatorial models of the
moduli space of complex curves.

A paradigmatic problem in map enumeration is to count planar maps while controlling face degrees. This question was first considered
by Tutte~\cite{Tutte62} who obtained a remarkably simple expression
when all degrees are even. In the general case, he obtained a
functional equation for the corresponding generating
function~\cite{Tutte68}, whose solution was given by Bender and
Canfield~\cite{BeCa94}, see also~\cite{BoJe06}. Independently, the
same problem was solved in theoretical physics in connection with the
so-called one-matrix model, see e.g.~\cite{DGZ95,BouOHRMT11} and
references therein.

While these approaches lead to an explicit answer, in the form of
generating functions amenable to an asymptotic analysis, it is natural
to look for a \emph{bijective} approach. Such an approach was pioneered
by Schaeffer who rederived Tutte's result in the case of even degrees~\cite{Schaeffer97}, and later extended to arbitrary degrees by Bouttier, Di Francesco and Guitter~\cite{census}. Both papers rely on correspondences between
planar maps and certain decorated trees, often called \emph{blossom(ing)
  trees}. Another bijection, involving another family of trees known as \emph{mobiles}, was given in~\cite{mobiles}.

Trees are easier to enumerate thanks to their recursive structure:
for rooted trees, the operation of removing the root vertex translates
into equations---that are typically algebraic---determining their
generating functions. This naturally raises the question of whether such recursive decompositions can be performed directly at the level of maps. A construction of this type was proposed in~\cite[Appendix~A]{hankel}, and
is now known as the \emph{slice decomposition}, see
also~\cite[Chapter~2]{BouttierHDR} and~\cite{Bouttier2024}. It
provides an elementary yet elegant bijective solution to the problem
of counting maps with controlled face degrees, and moreover extends naturally to maps with girth constraints~\cite{irredmaps,Bouttier2024a} and to scaling
limits~\cite{BeMi17}.

In this paper, we extend the slice decomposition to the more general
context of planar \emph{hypermaps}. Recall that a hypergraph generalizes a graph by allowing \emph{hyperedges}, which may be incident to an arbitrary number of vertices rather than just two. A hypermap is the generalization of a
map in the same manner. In practice, we represent hyperedges as
faces of a second color, so that hypermaps are identified with properly
face-bicolored maps, dual to bipartite maps~\cite{Walsh75}.

The corresponding generalization of the above enumerative problem is
to count planar hypermaps with a control on the degrees of both faces
and hyperedges\footnote{We may of course ask to control the degrees of
  faces, hyperedges and vertices all at the same time. This is a
  significantly harder problem, even for maps (where edge degrees are fixed to be two); see
  \cite{DFIt93,KaStWy96,Kazakov2022,BenDali2024} for some results in
  this direction.}  or, in the representation chosen in this paper, the
degrees of faces of each color separately. This indeed generalizes the
previous problem, since maps correspond to hypermaps in which all
hyperedges have degree two. The question seems to have been first
addressed in the physics literature in connection with the
\emph{two-matrix
  model}~\cite{Itzykson1980,Mehta1981,Douglas91,Staudacher1993,Daul1993}. A
key motivation for physicists was the study of \emph{two-dimensional
  quantum gravity}~\cite{DGZ95}, as it was realized that the
two-matrix model exhibits a much richer critical behavior than the
one-matrix model. In other words, generating functions for planar
hypermaps with controlled degrees display certain types of
singularities that cannot be obtained with planar maps only. In
particular, a specialization of the two-matrix model is the celebrated
Ising model on (dynamical) random planar maps, first solved by
Kazakov~\cite{Kazakov86}, which remarkably displays a phase transition
and a critical point~\cite{BoKa87}. The deep algebraic and integrable structures underpinning the two-matrix model were further
investigated, in particular by Bertola, Eynard and their
collaborators, see e.g.~\cite{Bertola2003,Eynard2003,EO2007} and
references therein. We refer to the eighth chapter of the
book~\cite{EynardBook} for a combinatorially-oriented account of some
of these results\footnote{This reference deals with maps endowed with
  a bicoloring of the faces which is not necessarily proper. Such maps
  can be transformed into hypermaps by adding faces of degree two, and
  the enumeration problems are equivalent. See for instance the
  discussion in~\cite[Appendix~B]{BC2021}.}. Let us mention in passing the recent works~\cite{Duits_Hayford_Lee24,Hayford24} on a Riemann-Hilbert approach to the two-matrix model and its application to the Ising model coupled to 2D gravity.

On the bijective side, the first approach to the problem of counting
degree-controlled hypermaps was made by Bousquet-Mélou and
Schaeffer~\cite{BoSc02} via blossoming trees (and in the dual setting
of bipartite maps), thereby confirming rigorously the prediction of Boulatov and Kazakov for the genus 0 partition function of the Ising model on random tetravalent planar maps. Another bijective approach via mobiles was given
in~\cite[Section~3]{mobiles}\footnote{We shall mention that this
  reference discusses more the bijection itself than its enumerative
  consequences. On the latter aspect, it provides an infinite system
  of recursive equations characterizing the generating function of
  mobiles and ``half-mobiles''. It also claims that, in the ``limit of
  large labels'' (which amounts to lifting the positivity condition on
  labels), one recovers the equations from~\cite{BoSc02}: this is
  actually not quite true, as we will see in
  Remark~\ref{rem:mobiles_connection} below. Let us mention that the
  full infinite system of recursive equations has recently been
  related to the theory of integrable systems
  in~\cite{BergereEtAl}.}. These two approaches give systems of
recursive equations which are equivalent to those obtained for slices
in the present paper, see Remarks~\ref{rem:BoSc_connection} and
\ref{rem:mobiles_connection} below. Their limitation, however, was that the bijections imposed somewhat ``unnatural'' constraints on the trees or maps. Namely, the main result of
Bousquet-Mélou and Schaeffer~\cite[Theorem~3]{BoSc02} is about maps
rooted at a black vertex of degree $2$, and it was unclear whether
this restriction could be lifted. And the construction given
in~\cite[Section~3.3]{mobiles} produces mobiles whose labels are
constrained to be positive, which complicates their enumeration considerably.
Such issues were typical in the early days of the bijective approach
to map enumeration, and were later addressed by lifting the balance condition in blossoming trees (using $\alpha$-orientations), or by considering pointed rooted maps to relax positivity in mobiles, etc. Modern treatments of tree bijections for hypermaps
can be found in~\cite{Bernardi2020} and~\cite{AMT2025}: the first
paper discusses a very general bijection for hypermaps with girth
constraints, and recovers~\cite{BoSc02,mobiles} as special cases; the
second paper introduces a family of orientations characterizing
bipartite maps (dual to hypermaps) which allows to recover and extend
the blossoming bijection of~\cite{BoSc02}. Still, to the best of our knowledge,
most of the nice enumerative results from~\cite[Chapter~8]{EynardBook}
have not yet received a bijective proof, and this is the gap that we
intend to address with this paper.

We will show that the slice framework is particularly well-suited for this purpose: slices are certain (hyper)maps that admit recursive
decompositions equivalent to those of blossoming trees and mobiles, while also serving as building blocks to construct bijectively
several ``natural'' families of hypermaps. Let us list, by increasing
order of complexity, which families we treat in the paper: pointed
disks with one monochromatic boundary, cylinders with two
monochromatic boundaries (starting with the case of so-called trumpets
and cornets, i.e. cylinders with one geodesic boundary)\footnote{We believe that the
  decomposition of a cylinder into a trumpet/cornet pair is very
  similar to the bijection used in the proof
  of~\cite[Lemma~44]{Bernardi2020}. The idea of cutting a cylinder (or
  annular map) along an extremal minimal separating cycle appears in
  previous papers, and might be the first bijective argument in
  map enumeration not easily reformulated in terms of trees.}, disks with one monochromatic
boundary, disks with a ``Dobrushin'' boundary condition.  The case of
more general ``mixed'' boundary conditions, discussed
in~\cite[Section~8.4]{EynardBook}, is not treated in the present
paper, but will be adressed in future works~\cite{BEL,Lejeune}. 

Another contribution of the present paper is a combinatorial
explanation for the existence of a ``rational parametrization'' for
the generating functions of hypermaps with monochromatic boundaries,
in terms of the parameters controlling the boundary lengths. This
phenomenon was previously observed in the approach via loop equations,
see e.g.~\cite[Section~8.3]{EynardBook}, but its origin remained
mysterious to us. One of our insights is that hypermaps with
monochromatic boundaries can be ultimately decomposed as sequences of
slices whose increments form downward-skip free (\L{}ukasiewicz-type)
walks. Generating functions for such sequences/walks can all be
expressed in terms of a single ``master'' series, which is the
generating function of \emph{excursions}. See Appendix~\ref{sec:dsf}
below for a self-contained discussion. In our application to hypermaps
we actually have two master series, one per color, and the variable
controlling the length of a boundary of a given color turns out to be a
Laurent polynomial in the corresponding master series.

Let us mention that one of our main motivations is the aforementioned
Ising model on random maps, which is expected to display at the
critical point large-scale geometric properties that are totally
different from those of generic maps. A probabilistic approach to this model was recently developed in the
papers~\cite{Albenque2021,Albenque2022,Chen2020,Chen2023,Turunen2020},
which rely heavily on enumerative results such as those given
in~\cite[Section~12]{Bernardi2011}\footnote{This paper treats the
  Ising model as the $q=2$ specialization of the more general
  $q$-state Potts model. Let us mention that the rational
  parametrization involved in~\cite[Theorems~21 and~23]{Bernardi2011}
  is for the variable associated to the ``volume'' of the maps, while
  the rational parametrization discussed in the present paper is for
  the variables associated with the lengths of the boundaries.}. We
believe that more combinatorial results are needed to go further, and
we hope that the present paper might help to make progress in that
direction.

Finally, let us note that some of the ideas developed here have already been applied in~\cite{Bonzom2024} to the enumeration of certain planar constellations, corresponding to the spectral curve of Orlov–Scherbin $2$-Toda tau functions.

The remainder of this introduction delves into the matter, first by
giving some precise definitions, then by giving an overview of the
enumerative results which will follow from our bijective approach, and
finally by presenting the outline of the rest of the paper.

\subsection{Maps and hypermaps: basic notions and definitions}\label{sub:DefIntro}

Let us now introduce precisely the objects which we intend to
enumerate in this paper. We start with some basic notions on maps,
and refer to~\cite{Schaeffer15} for more details.

A \emph{planar map}, hereafter called \emph{map} for short, is a proper embedding of a connected finite multigraph into the sphere, considered up to orientation-preserving homeomorphism. By multigraph, we mean a graph where loops and multiple edges are allowed. A map consists of \emph{vertices}, \emph{edges}, \emph{faces}, and their incidence relations. The embedding fixes the cyclical order of edges around each vertex, which defines readily a \emph{corner} as a couple of consecutive edges around a vertex. The set of vertices, edges, faces and corners of a planar map $\rm$ are respectively denoted by $\rV(\rm)$, $\rE(\rm)$, $\rF(\rm)$ and $\rC(\rm)$. For any $\kappa\in \rC(\rm)$, the vertex incident to $\kappa$ is denoted by $\v(\kappa)$. A corner is also incident to a unique face.

The \emph{degree} of a vertex or a face is defined as the number of its incident corners. In other words, it counts incident edges, with multiplicity 2 for each loop (in the case of vertex degree) or for each bridge (in the case of face degree).  The \emph{contour} of a face is the cycle formed by its incident edges and vertices.

A \emph{boundary} is a distinguished face. Unless otherwise stated, each boundary is \emph{rooted}, i.e.\ it is endowed with a marked incident corner. Faces which are not boundaries are called \emph{inner faces}. Planar maps with one and two boundaries are colloquially called \emph{disks} and \emph{cylinders}, respectively. It is customary to draw a planar map in the plane, upon performing a stereographic projection of the sphere. This requires choosing a point at infinity, usually within a face which is therefore promoted to the role of \emph{outer} (unbounded) face. In a map with boundaries, the outer face is chosen among the boundaries. Note that the other boundaries are not called inner faces in our present terminology.

\begin{figure}[t]
\centering
\subcaptionbox{An hypermap with one monochromatic boundary,\label{subfig:hypermap}}{\;\includegraphics[page=1,width=0.4\linewidth]{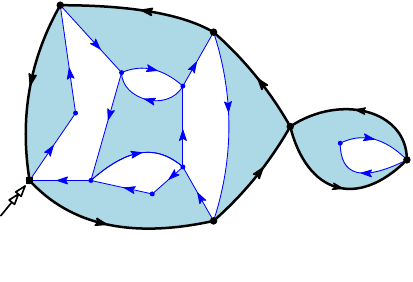}}\qquad\qquad
  \subcaptionbox{a hypermap with two monochromatic boundaries (one white and one black),\label{subfig:hypermapMonochromatic}}{\;\includegraphics[page=1,width=0.45\linewidth]{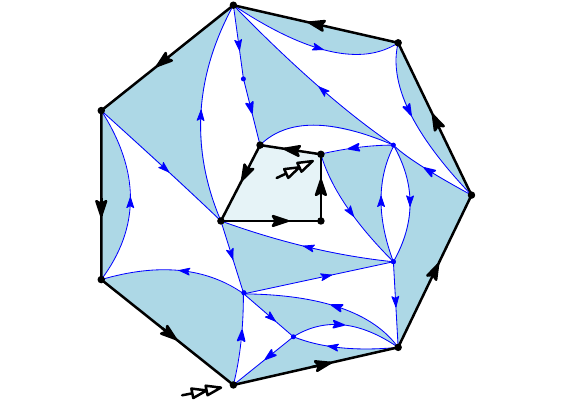}}\qquad\qquad
  \subcaptionbox{a hypermap with one non-monochromatic boundary.\label{subfig:hypermapBoundary}}{\;\includegraphics[page=4,width=0.4\linewidth]{images/hypermap.pdf}}\qquad\qquad
  \subcaptionbox{and, a hypermap with a $(5,2)$-Dobrushin boundary condition.\label{subfig:hypermapDobrushin}}{\;\includegraphics[page=3,width=0.4\linewidth]{images/hypermap.pdf}}
 \caption{Illustration of the different families of hypermaps introduced. Root corners are indicated by double black arrows and the contour of boundaries by thicker black edges (for readability, throughout this paper, black faces are colored in blue).}
  \label{fig:firstExamples}
\end{figure}

In this paper, we consider \emph{hypermaps}, that are planar maps endowed with a proper bicoloration (say in black and white) of their faces. Each edge of a hypermap can be canonically oriented, by requiring that the face on its right is white and that the face on its left is black. This gives a bijection between hypermaps and maps endowed with an orientation of their edges such that the contour of each face is an oriented cycle. See Figure~\ref{subfig:hypermap} for an illustration. 

There is a slight subtlety when considering hypermaps with boundaries. The simplest notion is that of \emph{hypermaps with monochromatic boundaries}, defined as hypermaps with some distinguished faces, colored in the same way as the inner faces, see Figures~\ref{subfig:hypermap} and~\ref{subfig:hypermapMonochromatic}. However, it is also interesting to consider non-monochromatic boundaries, which are more conveniently defined via orientations: a \emph{hypermap with non-monochromatic boundaries} is defined as a map with boundaries endowed with an orientation of its edges such that the contour of each \emph{inner} face is a directed cycle. In other words, we relax the condition that the contours of the boundaries are directed cycles, see Figure~\ref{subfig:hypermapBoundary}. Note that inner faces may still be naturally colored black and white, but the boundaries themselves do not necessarily have a defined color. Let us emphasize that from the coloring of the inner faces, only the orientation of the edges incident to at least one inner face can be retrieved, but not the orientation of edges only incident to boundaries (e.g. bridges).

In particular, a \emph{hypermap with a Dobrushin boundary of type
  $(p,q)$} is defined as a hypermap with one non-monochromatic
boundary such that the contour of the boundary face consists of two
directed paths: one of length $p$ formed by edges having the boundary
face on their right, the other of length $q$ formed by edges having
the boundary face on their left. By convention the root corner is
placed at the origin of these two paths, see Figure~\ref{subfig:hypermapDobrushin}. Note that, for $p=0$ or $q=0$, we recover a rooted monochromatic boundary.

\subsection{Generating functions of planar hypermaps: an overview}
\label{sec:genover}

We now introduce the main protagonists of this paper: the generating functions of certain classes of planar hypermaps. For
$\rm$ a hypermap with boundaries (monochromatic or not), we generally
define its weight $w(\rm)$ as
\begin{equation}\label{eq:Boltzmannweight}
w(\rm) :=t^{\left\lvert\rV(\rm)\right\rvert}\prod_{f \in \rFb(\rm)}t^\circ_{\deg(f)}\prod_{f \in \rFn(\rm)}t^\bullet_{\deg(f)},
\end{equation}
where 
$t, t^\circ_1,t^\circ_2,\ldots,t^\bullet_1,t_2^\bullet,\ldots$ are
formal variables, $\rV(\rm),\rFb(\rm),\rFn(\rm)$ denote respectively the set of vertices, white inner faces, black inner faces of $\rm$, and $\deg(f)$ denotes the degree of a face $f$. Throughout this paper, we denote the ring of formal power series in these variables with rational coefficients by $\Ring:=\mathbb{Q}[[t, t^\circ_1,t^\circ_2,\ldots,t^\bullet_1,t_2^\bullet,\ldots]]$.

A fundamental role is played by the \emph{disk generating functions},
namely the generating functions of (planar) hypermaps with one
monochromatic boundary. For $p\geq 1$, we set

\begin{equation}\label{eq:defFp}
F_p^\circ := \sum_{\substack{\rm \text{ hypermap with}\\ \text{a white rooted boundary}\\ \text{ of degree }p}} w(\rm)\qquad \text{and}\qquad F_p^\bullet := \sum_{\substack{\rm \text{ hypermap with}\\ \text{a black rooted boundary}\\ \text{ of degree }p}} w(\rm),
\end{equation}
which are both series belonging to $\Ring$. By convention, we set $F_0^\circ=F_0^\bullet=t$, which can be
interpreted as the weight of the map reduced to a single vertex (and
no edge). We gather the $F_p^\circ$ and $F_p^\bullet$ into the grand
generating functions
\begin{equation}\label{eq:defWDiskintro}
\Wb(x):= \sum_{p\geq 0}\frac{F_p^\circ}{x^{p+1}}\qquad \text{and} \qquad \Wn(y) := \sum_{p\geq 0}\frac{F_p^\bullet}{y^{p+1}}
\end{equation}
which lie in $\Ring[[x^{-1}]]$ and of $\Ring[[y^{-1}]]$, respectively; that is, they are formal power series in $x^{-1}$ (resp. $y^{-1}$) with coefficients in $\Ring$. Here, we
follow the usual convention of the physics literature to
work with inverse variables: one can intuitively think of \eqref{eq:defWDiskintro} as the analytic expansions at infinity of the ``resolvents'' $\Wb(x)$ and $\Wn(y)$. This choice, as well as the shift of the exponent by $+1$, will turn out to give nicer enumerative formulas. 

We will see that the disk generating functions can be expressed in
terms of two families of auxiliary series $(a_k)_{k \geq -1}$ and
$(b_k)_{k \geq 0}$ which are determined through a system of
``tree-like'' recursive equations. A compact way of writing this
system is to introduce the Laurent series
\begin{equation}\label{eq:defxyintro}
  x(z):=\sum_{k \geq -1} a_kz^{-k} \quad\text{ and } \quad y(z):=z^{-1} + \sum_{k\geq 0} b_kz^{k},
\end{equation}
then the system consists in imposing the conditions
\begin{equation}
  \label{eq:absys}
  \begin{split}
    [z^{-k}] \Big( x(z) - t z -\sum_{d \geq 1} t^\bullet_d
    y(z)^{d-1}\Big) = 0 &\qquad \text{for all $k \geq -1$}, \\
    [z^k] \Big( y(z) - \sum_{d \geq 1} t^\circ_d x(z)^{d-1}\Big) =
    0 &\qquad \text{for all $k \geq 0$}.
  \end{split}
\end{equation}
Here, $[z^k]A(z)$ denotes the coefficient of $z^k$ in $A(z)$. As we
will discuss in more detail in Section~\ref{sec:slices}, the
system~\eqref{eq:absys} uniquely determines the $a_k$ and $b_k$ as
elements of $\Ring$. In the case of \emph{bounded face
  degrees}, that is when we set $t^\bullet_k=0$ for all $k>\dmaxn$ and
$t^\circ_k=0$ for all $k>\dmaxb$, with $\dmaxn$ and $\dmaxb$ being arbitrary integers chosen a priori, then the system~\eqref{eq:absys}
implies that $a_k=0$ for $k \geq \dmaxn$ and $b_k=0$ for
$k \geq \dmaxb$. Hence, we obtain an \emph{algebraic} system of equations
for the series $a_{-1},\ldots,a_{\dmaxn-1},b_0,\ldots,b_{\dmaxb-1}$.
Essentially the same system appears
in~\cite{BoSc02} and \cite{mobiles}, where it was given a combinatorial
interpretation in terms of blossomming trees and mobiles, respectively. In
this paper, we give yet another combinatorial interpretation, this time in
terms of certain hypermaps called \emph{slices}, which will be defined
in Section~\ref{sec:slices}.

Slices have the advantage of permitting manipulations that are harder to visualize on blossoming trees or mobiles. This will allow us to make the connection with the disk
generating functions.

Let us first consider the first derivatives of these series with
respect to the variables
$t, t^\circ_1,t^\circ_2,\ldots,\allowbreak
t^\bullet_1,t_2^\bullet,\ldots$: as is often the case in the
enumeration of planar maps, these quantities admit simpler
expressions. The derivatives with respect to the vertex weight $t$
correspond to generating functions of \emph{pointed disks}, i.e.\
disks with an extra marked vertex. Corollary~\ref{cor:pointed} below
states that
\begin{equation}
  \label{eq:Fpt}
  \frac{\partial F_p^\circ}{\partial t} = [z^0] x(z)^p, \qquad
  \frac{\partial F_p^\bullet}{\partial t}  = [z^0] y(z)^p.
\end{equation}
At this stage, the combinatorially-inclined reader might have noticed
that the Laurent polynomials $x(z)^p$ and $y(z)^p$ have a natural
interpretation as the generating function of certain weighted walks on
$\Z$, which we call \emph{downward skip-free (DSF) walks} (see
Definition~\ref{def:DSF} and Remark~\ref{rem:DSFxy} below).  Then,
extracting the coefficient of $z^0$ means that we consider those walks
ending at their starting point, which are often called
\emph{bridges}. Proving~\eqref{eq:Fpt} boils down to exhibiting a
bijection between pointed disks and bridges whose steps are decorated
with slices, which we will do in Section~\ref{sub:pointed}. We will
furthermore show in Section~\ref{sec:apprelim} that the grand
generating functions of pointed disks read
\begin{equation}
  \label{eq:Wt}
   \frac{\partial W^\circ(x)}{\partial t} = \frac{d}{dx} \ln z^\circ(x), \qquad
  \frac{\partial W^\bullet(y)}{\partial t}  = - \frac{d}{dy} \ln z^\bullet(y),
\end{equation}
where $z^\circ(x)$ and $z^\bullet(y)$ are ``compositional inverses''
of respectively $x(z)$ and $y(z)$. Precisely, the series
$\tilde{z}^\circ(x)=z^\circ(x)^{-1}$ and $z^\bullet(y)$ are the unique
formal power series in respectively $x^{-1}$ and $y^{-1}$ such that
\begin{equation}
  \tilde{z}^\circ(x)  = \sum_{k\geq -1} \frac{a_k}{x} \tilde{z}^\circ(x)^{k+1}, \qquad
  z^\bullet(y) = \sum_{k\geq -1} \frac{b_k}{y} z^\bullet(y)^{k+1},
\end{equation}
which indeed amount to $x = x\left(z^\circ(x)\right)$ and
$y = y(z^\bullet(y))$. Combinatorially, $\tilde{z}^\circ(x)$ and
$z^\bullet(y)$ can be interpreted as generating functions of so-called DSF
\emph{excursions}.

As for the derivatives with respect to the face weights
$t^\circ_1,t^\circ_2,\ldots,\allowbreak
t^\bullet_1,t_2^\bullet,\ldots$, they correspond to generating
functions of cylinders with monochromatic boundaries. Following the
convention of having the boundaries rooted, we consider the generating
functions
\begin{equation}
  \label{eq:Fpq}
  F^{\circ\circ}_{p,q} = q \frac{\partial F^\circ_p}{\partial t^\circ_q} =
  p \frac{\partial F^\circ_q}{\partial t^\circ_p}, \qquad
  F^{\bullet\bullet}_{p,q} = q \frac{\partial F^\bullet_p}{\partial t^\bullet_q} =
  p \frac{\partial F^\bullet_q}{\partial t^\bullet_p}, \qquad
  F^{\circ\bullet}_{p,q} = q \frac{\partial F^\circ_p}{\partial t^\bullet_q} =
  p \frac{\partial F^\bullet_q}{\partial t^\circ_p}.
\end{equation}
They correspond to cylinders whose two boundaries are respectively
both white, both black, white and black, and have prescribed degrees
$p$ and $q$. Corollary~\ref{cor:cylindersenum} below gives the
following expressions:
\begin{equation}
  \label{eq:cylintro}
  \begin{split}
    F^{\circ\circ}_{p,q} &= \sum_{h \geq 1} h \left([z^h] x(z)^p \right) \left( [z^{-h}] x(z)^{q} \right),\\
    F^{\bullet\bullet}_{p,q} &= \sum_{h \geq 1} h \left([z^h] y(z)^p \right) \left( [z^{-h}] y(z)^{q} \right),\\
    F^{\circ\bullet}_{p,q} &= \sum_{h \geq 1} h \left([z^h] x(z)^p \right) \left( [z^{-h}] y(z)^{q} \right).\\
  \end{split}
\end{equation}
The associated grand generating functions admit then the expressions
\begin{equation}
  \label{eq:Wcyl}
  \begin{split}
    W^{\circ\circ}(x_1,x_2) &= \sum_{p,q \geq 1} \frac{F^{\circ\circ}_{p,q}}{x_1^{p+1} x_2^{q+1}}
    = \frac{\partial^2}{\partial x_1 \partial x_2} \ln \left(\frac{z^\circ(x_1)-z^\circ(x_2)}{x_1-x_2}\right),\\
    W^{\bullet\bullet}(y_1,y_2) &= \sum_{p,q \geq 1} \frac{F^{\bullet\bullet}_{p,q}}{y_1^{p+1} y_2^{q+1}} = \frac{\partial^2}{\partial y_1 \partial y_2} \ln \left(\frac{z^\bullet(y_1)-z^\bullet(y_2)}{y_1-y_2}\right),\\
    W^{\circ\bullet}(x,y) &= \sum_{p,q \geq 1} \frac{F^{\circ\bullet}_{p,q}}{x^{p+1} y^{q+1}} = -
    \frac{\partial^2}{\partial x \partial y} \ln \left( 1 - \frac{z^\bullet(y)}{z^\circ(x)}\right)
\end{split}
\end{equation}
still in terms of the series $z^\circ(x)$ and $z^\bullet(y)$
introduced above, see Proposition~\ref{prop:ggfpointcyl} below. We
note that essentially the same expressions appear in
\cite[Section~9]{Daul1993}. An important idea in our bijective
derivation of~\eqref{eq:cylintro} is that the quantities
$[z^{\pm h}] x(z)^p$, $[z^{\pm h}] y(z)^p$ can be interpreted as the
generating functions of certain types of cylinders which we call
\emph{trumpets} and \emph{cornets}: these are cylinders in which one
of the boundaries is assumed to be ``geodesic'' or ``tight'', see
Section~\ref{sub:trumpets}. Also, the attentive reader might have
noticed that a fourth generating function seems to be missing
in~\eqref{eq:cylintro}. We leave this as a teaser for
Section~\ref{sub:cyl}.

Let us now consider the undifferentiated disk generating functions.
Several expressions will be obtained combinatorially in
Section~\ref{sub:mono}, let us just quote here the most compact ones:
\begin{equation}
  \label{eq:Fpcompactintro}
  F^\circ_p = \frac{1}{p+1}\sum_{h \in \Z} h b_{-h} [z^h]x(z)^{p+1}, \qquad
  F^\bullet_p = \frac{1}{p+1}\sum_{h \in \Z} h a_{-h} [z^{-h}]y(z)^{p+1}.
\end{equation}
Proposition~\ref{prop:Wsubsbis} states that the associated grand
generating functions read
\begin{equation}
  \label{eq:Wsubs}
  W^\circ(x) = y(z^\circ(x)) - \sum_{d \geq 1} t^\circ_d x^{d-1}, \qquad
  W^\bullet(y) = x(z^\bullet(y)) - \sum_{d \geq 1} t^\bullet_d y^{d-1}.
\end{equation}
As discussed in Section~\ref{sec:apmono}, this result is essentially
equivalent to the rational parametrization of the spectral curve
given in~\cite[Theorem~8.3.1]{EynardBook}. 
Finally, we consider the generating function $F^{\yy}_{p,q}$ of
hypermaps with a Dobrushin boundary of type $(p,q)$. By convention we
have $F^{\yy}_{p,0}=F^\circ_p$, $F^{\yy}_{0,q}=F^\bullet_q$, and
$F^{\yy}_{0,0}=t$. Proposition~\ref{prop:Wdobr} below states that
\begin{equation}
  \label{eq:Wdobrintro}
                   \sum_{p,q \geq 0}  \frac{F^{\yy}_{p,q}}{x^{p+1} y^{q+1}}
  =\exp \left(\sum_{h \in \Z} h \left([z^h] \ln \left(1-\frac{x(z)}x\right) \right) \left( [z^{-h}] \ln \left(1-\frac{y(z)}y\right) \right)\right)-1.
\end{equation}
The connection with the expression given
at~\cite[Equation~(8.4.6)]{EynardBook} is discussed in
Section~\ref{sec:apdobru}.

\subsection{Outline, acknowledgments}
\label{sec:outline}

\paragraph{Outline.} The remainder of this paper is organized as
follows. Section~\ref{sec:slices} introduces slices, the fundamental
objects of our combinatorial approach: we start with the basic
definitions in Section~\ref{sub:defSlices}, before turning to the
recursive decomposition of slices, and its translation in terms of
generating functions, in
Section~\ref{sec:slicesgf}.

Section~\ref{sec:wrapping} discusses the operation of ``wrapping''
slices, by gluing their left and right boundaries together, to produce
other families of maps. We start in
Section~\ref{sub:pointedRootedMaps} with the easiest case of pointed
rooted hypermaps, then proceed in Section~\ref{sub:pointed} to the
case of pointed disks obtained by wrapping slices of zero increment. Finally, in Section~\ref{sub:trumpets}, we treat the wrapping of slices of nonzero increment, producing the so-called trumpets and cornets.

Section~\ref{sec:cyldisk} explains how cylinders with monochromatic
boundaries may be obtained by gluing together a trumpet and a cornet.
We treat the case of general cylinders in Section~\ref{sub:cyl}, and
observe in particular that cylinders with boundaries of opposite
colors come in two kinds, which we dub ``two-way'' and ``one-way''
cylinders. Using this distinction, we show in Section~\ref{sub:mono}
that disks with a monochromatic boundary are in bijection with one-way
cylinders of a special kind, which allows us to enumerate them. Then,
in Section~\ref{sub:Dobru}, we show that disks with a Dobrushin
boundary condition are closely related to one-way cylinders, and we
introduce the so-called ``blob decomposition'' to derive an exact counting formula. 

Section~\ref{sec:allperim} is devoted to the all-perimeter ``grand''
generating functions. Using enumerative results on downward skip-free
walks, we obtain in Section~\ref{sec:apprelim} the grand generating
functions of pointed disks and cylinders with monochromatic
boundaries. Section~\ref{sec:apmono} deals with disks with
monochromatic boundaries, and shows that, under the assumption of
bounded face degrees, their generating function admits a rational
parametrization. Section~\ref{sec:apdobru} explains how the generating
function of disks with a Dobrushin boundary condition is related to the resultant of the spectral curve.

Some concluding remarks are gathered in Section~\ref{sec:conc}, and
supplementary material may be found in the appendices:
Appendix~\ref{sec:dsf} gives a self-contained discussion of the
enumeration of downward skip-free walks, and
Appendix~\ref{sec:altmono} gives an alternative approach to the
enumeration of monochromatic disks, via pointed disks.

  \paragraph{Acknowledgments.} We thank Bertrand Eynard, Emmanuel
Guitter and Thomas Lejeune for useful discussions. JB acknowledges the
hospitality of Laboratoire de physique, ENS de Lyon, where part of
this project was completed. This work was supported by the CNRS PEPS
Carma and the Agence Nationale de la Recherche via the grants
ANR-18-CE40-0033 ``Dimers'', ANR-19-CE48-0011 ``Combiné'',
ANR-21-CE48-0007 ``IsOMa'' and ANR-23-CE48-0018 ``CartesEtPlus''.

\section{Slices}\label{sec:slices}

In this section, we introduce \emph{slices}, the fundamental objects of our combinatorial approach. In a nutshell, slices are hypermaps with certain directed geodesic boundaries. We define them in detail in Section~\ref{sub:defSlices}: preliminaries on directed distances and geodesics are given in Section~\ref{sub:odgprelim}, followed by the introduction of general slices in Section~\ref{sub:generalSlices} and elementary slices in Section~\ref{sub:elem}.

We then turn, in Section~\ref{sec:slicesgf}, to the enumeration of slices via recursive decomposition. In Section~\ref{sub:decompositionElementary}, we explain how a general slice can be decomposed into a sequence of elementary slices; Section~\ref{sub:enumElemSlices} shows how an elementary slice, in turn, yields a general slice by removing the base edge. Altogether this yields a system of equations characterizing the generating functions of elementary slices, equivalent to that obtained in~\cite{BoSc02} for blossomming trees. Finally, Section~\ref{sub:decAlternative} discusses an alternative decomposition for certain elementary slices, which is useful to make the connection with the mobiles of~\cite{mobiles}.

\subsection{Basic definitions and properties}\label{sub:defSlices}

\subsubsection{Directed distance and geodesics}\label{sub:odgprelim}

Fix $\rm$ a hypermap endowed with its canonical orientation. Then, for
vertices $u,v\in \rV(\rm)$, the \emph{directed distance} from $u$ to
$v$, which we denote by $\pd(u,v)$ or by
$\pd_{\rm}(u,v)$ if we need to make the dependency on $\rm$ explicit,
is defined as the minimal length (number of edges) of a directed 
path from $u$ to $v$. If no such path exists, we set $\pd(u,v)=\infty$ by convention.
We have $\pd(u,v)=0$ if and only if $u=v$, and $\pd$ satisfies the
oriented version of the triangle inequality:
\begin{equation}\label{eq:triangleInequality}
\pd(u,w)\leq \pd(u,v)+\pd(v,w), \quad \text{for any }u,v,w \in \rV(\rm).
\end{equation}
Note however that $\pd$ is not symmetric, i.e. $\pd(u,v)$ and $\pd(v,u)$ are not equal in general.

A directed path from $u$ to $v$ of length $\pd(u,v)$ is called a \emph{directed geodesic}, or \emph{geodesic} for short, from $u$ to
$v$. If, instead of a vertex $u$, we consider a corner $\kappa$, then by a 
geodesic from $\kappa$ to $v$ we mean a geodesic from
$\v(\kappa)$, the vertex incident to $\kappa$, to $v$. Similarly, $v$ may be replaced by a corner, if needed.
 
There can \emph{a priori} exist multiple geodesics between two given vertices. To single out one canonical geodesic, previous works on slice decompositions~\cite{hankel,irredmaps} rely on the so-called leftmost geodesic starting from a corner. Here, we use the same notion with the necessary adjustments to handle directed distance instead of the usual graph distance. Precisely, given a corner $\kappa$ and a vertex $v$, the \emph{leftmost geodesic} $\gamma$ from $\kappa$ to $v$ is defined recursively as follows. If $v$ is equal to $\v(\kappa)$, then $\gamma$ is the path of length zero starting and ending at $v$. Otherwise, among all the edges incident to $\v(\kappa)$ that belong to a geodesic from $\kappa$ to $v$, choose $e$ to be the first one encountered when turning clockwise around $\v(\kappa)$ starting at $\kappa$. Let $w$ be the endpoint of $e$, and let $\kappa'$ be the corner immediately following $e$ when turning clockwise around $w$. Then, we define $\gamma$ to be the path obtained by prefixing $e$ to the leftmost geodesic from $\kappa'$ to $v$.

\begin{remark} \label{rem:access} Let us discuss in which conditions
  the directed distance can be infinite. When a hypermap $\rm$ has
  only monochromatic boundaries, its canonical orientation is
  \emph{Eulerian}, meaning that the numbers of incoming and outgoing
  edges at each vertex are equal. Hence, there exists an Eulerian circuit,
  which implies that $\rm$ is strongly connected, i.e.\ 
  $\pd(u,v)<\infty$ for any pair of vertices $u$ and $v$.
  However, a hypermap with non monochromatic boundaries is not
  necessarily strongly connected, see e.g.\ 
  Figure~\ref{subfig:hypermapBoundary} (the bridge separates two
  strongly connected components). Still, it is straightforward to
  check that, given any vertex $v$, there exists a vertex $w$ incident
  to one of the boundaries such that $\pd(v,w)<\infty$.
\end{remark}

\subsubsection{General slices}\label{sub:generalSlices}

Armed with the above definitions, we may now introduce the main concept of this article:
\begin{definition}\label{def:slice}
  A \emph{slice} is a planar hypermap with one (not necessarily
  monochromatic) boundary, chosen as the outer face, having three
  incident distinguished corners -- hereafter denoted $o$, $l$ and $r$
  -- which are not necessarily distinct. The corners appear in counterclockwise
  order around the map and divide the contour of the outer face into three parts subject to the following conditions:
  \begin{itemize}
  \item The part between $o$ and $l$, called the \emph{left boundary}, is a directed geodesic from
    $l$ to $o$,
  \item The part between $r$ and $o$, called the \emph{right boundary}, is
    the unique directed geodesic from $r$ to $o$. It intersects the left boundary only at the vertex incident to $o$, referred to as the \emph{apex}.
  \item The part between $l$ and $r$, called the \emph{base}, is a
    directed path either from $l$ to $r$ or from $r$ to $l$. The slice is said to be of type $\cA$ in the first case and of type $\cB$
    in the second case.
  \end{itemize}
\end{definition}

\begin{figure}[t]
  \centering
  \includegraphics[scale=0.8]{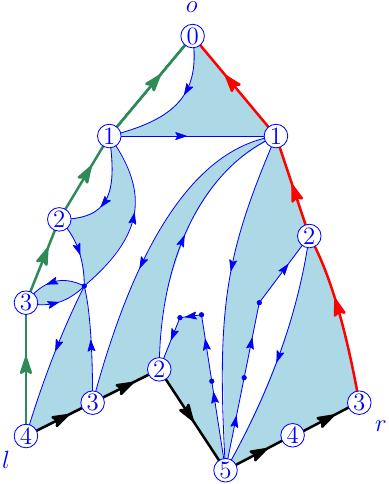}
  \caption{\label{fig:FirstExamplesSlices} A slice of type $\cA$ with
    its distinguished corners $o,l,r$. The left boundary, the right
    boundary, and the base are shown in respectively green, red, and
    black. The blue numbers on boundary vertices indicate their
    canonical labeling, i.e.\ their directed distance towards the apex
    (vertex incident to $o$). Inner vertices also receive labels but
    we do not display them for readability. The increment of the slice
    is the difference between the label of $r$ and that of $l$, here
    it is equal to $3-4=-1$.  }
\end{figure}

See Figure~\ref{fig:FirstExamplesSlices} for an illustration. Note
that we do not require the boundary of a slice to be simple: while the
left and right boundaries must be simple paths, as they are
geodesics, the base may visit the same vertex several times, or
share vertices or edges with the left or right boundaries.
In the papers~\cite{hankel,irredmaps}, the base of a slice was required
to be reduced to one edge. This more restrictive notion of slices
corresponds to what we call \emph{elementary slices} below.

\begin{remark}\label{rem:leftmost}
Note that the left boundary is clearly the leftmost geodesic from $l$ to $o$. Moreover, since the right boundary is the only geodesic from $r$ to $o$, it is in particular also the leftmost geodesic from $r$ to $o$. 
\end{remark}

The \emph{canonical labeling} of a slice $\rm$ is the function
$\ell:V(\rm)\rightarrow \mathbb{Z}_{\geq 0}\,;\,u \mapsto \pd(u,\v(o))$. Note that $\ell$ takes only finite values, since $\v(o)$ is clearly accessible from any vertex incident to the boundary, and hence from any vertex by Remark~\ref{rem:access}. 
We extend the definition of $\ell$ to $\rC(\rm)$ by setting  $\ell(\kappa)=\ell(\mathrm{v}(\kappa))$, for any $\kappa \in \rC(\rm)$. Then, the \emph{increment} of $\rm$ -- denoted by $\inc{\rm}$ -- is defined as:
\begin{itemize}
\item $\ell(r)-\ell(l)$ for a slice of type $\cA$ (i.e.\ with a base directed from $l$ to $r$),
\item $\ell(l)-\ell(r)$ for a slice of type $\cB$ (i.e.\ with a base directed from $r$ to $l$).
\end{itemize}

\begin{remark} \label{rem:incBound} The oriented triangle
  inequality~\eqref{eq:triangleInequality} entails that the increment
  of the slice $\rm$ is at least equal to the opposite of the length
  of the base. Indeed, for type $\cA$ we have
  $\inc{\rm}=\ell(r)-\ell(l)=\pd(\v(r),\v(o))-\pd(\v(l),\v(o)) \geq -
  \pd(\v(l),\v(r))$, and the base is directed from $l$ to $r$ so its
  length is at least $\pd(\v(l),\v(r))$. For type $\cB$ the argument
  is the same up to exchanging $l$ and $r$.
\end{remark}

\medskip

In the following, we will adopt the convention that edges from the left boundary are represented in green (and called the \emph{green edges}), whereas edges from the right boundary are represented in red (and called the \emph{red edges}). This choice of color stems from the fact that it is allowed to have a geodesic from $l$ to $o$ different from the left boundary, whereas it is forbidden to have a geodesic from $r$ to $o$ different from the right boundary.

\subsubsection{Elementary slices}\label{sub:elem}

\begin{figure}
  \centering
  \includegraphics{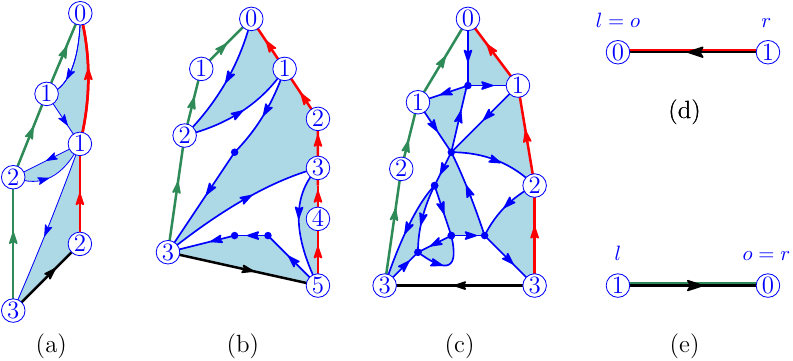}
  \caption{Some examples of elementary slices: (a) type $\cA_{-1}$,
    (b) type $\cA_{2}$, (c) type $\cB_{0}$, (d) the trivial slice,
    which is the unique slice of type $\cB_{-1}$, (e) the empty slice,
    which has type $\cA_{-1}$ but which is not the unique slice of this type.}
  \label{fig:hyperSlices}
\end{figure}

\begin{definition}\label{def:typeAk}
  An \emph{elementary slice} is a slice in which the base is reduced to a single edge, see Figure~\ref{fig:hyperSlices}.
  An elementary slice with increment $k$ is called:
\begin{itemize}
 \item an elementary slice of type $\cA_k$ (or $\cA_k$-slice) if it has type $\cA$,
\item an elementary slice of type $\cB_k$ (or $\cB_k$-slice) if it has type $\cB$.
 \end{itemize} 
\end{definition}

By Remark~\ref{rem:incBound}, the increment of an elementary slice is
at least $-1$. There exist two special elementary slices which
consist of a single directed edge: the \emph{empty} slice and the
\emph{trivial} slice, which are displayed on
Figure~\ref{fig:FirstExamplesSlices}, and are of types $\cA_{-1}$ and
$\cB_{-1}$ respectively.  We then make the following observation:

\begin{lemma}\label{lem:increment}
  In an elementary slice different from the trivial slice and the empty
  slice, the base edge is incident to an inner face, which is black
  for type $\cA$ and white for type $\cB$, and the increment is
  strictly smaller than the degree of this face. The trivial slice is
  the unique elementary slice of type $\cB_{-1}$.
\end{lemma}
\begin{proof}
  Consider an elementary slice $\rm$ endowed with its three
  distinguished corners $o,l,r$ and its canonical labeling $\ell$.

  If the base edge of $\rm$ is a bridge, then it must also belong to the left or the right boundary. Since these boundaries are geodesics that intersect only at the apex, it follows that $\rm$ must be reduced to a
  single edge, and be equal either to the trivial or to the empty
  slice.

  By contraposition, we deduce that, for any other elementary slice, the
  base edge is not a bridge, and hence is incident to an inner face
  which we denote by $f$. The color of $f$ is determined by the
  orientation of the base edge: black for type $\cA$, white for type
  $\cB$.  Let us denote by $d$ the degree of $f$. If $\rm$ is of type
  $\cA$, then the contour of $f$, with the base edge removed, forms an
  directed path of length $d-1$ from $\v(r)$ to $\v(l)$. Concatenating
  this path with the left boundary yields a directed path of length
  $d-1+\ell(l)$ from $\v(r)$ to $\v(o)$. This path must have length at
  least $\ell(r)$ hence $\inc{\rm}=\ell(r)-\ell(l) \leq d-1$ as
  wanted. If $\rm$ is of type $\cB$, the argument is the same up to
  exchanging $l$ and $r$.

  If $\rm$ is of type $\cB_{-1}$, then the concatenation of the base
  and of the left boundary forms a geodesic from $r$ to $o$. It must
  necessarily coincide with the right boundary, by its uniqueness
  property, so $\rm$ cannot have inner faces. In particular, the base
  edge is a bridge, and by the previous discussion, $\rm$ must therefore be the trivial slice.
\end{proof}

\subsection{Enumeration of slices via recursive decomposition}
\label{sec:slicesgf}

We now describe the recursive decomposition of slices. As it will
translate into equations determining their generating functions, let
us first discuss how these are defined.
By analogy with~\eqref{eq:Boltzmannweight}, we define the weight
$\bols(\rm)$ of a slice $\rm$ as
\begin{equation}\label{eq:defWeightslice}
\bols(\rm):=t^{|\rV'(\rm)|}\prod_{f\in \rFb(\rm)}t^\circ_{\deg(f)}\prod_{f\in \rFn(\rm)}t^\bullet_{\deg(f)},
\end{equation}
where $\rV'(\rm)$ denotes the set of vertices of $\rm$ \emph{not
  incident to its right boundary} (a detail that will prove convenient later). For any integer $k \geq -1$, let
$a_k$ and $b_k$ be the corresponding generating functions of
elementary slices of type $\cA_k$ and $\cB_k$, namely
\begin{equation}\label{eq:defAB}
a_k := \sum_{\rm\text{ of type }\cA_k}\bols(\rm)
\qquad
\text{and}
\qquad
b_k := \sum_{\rm\text{ of type }\cB_k}\bols(\rm).
\end{equation}
These formal power series belong to the ring
$\Ring:=\mathbb{Q}[[t,
t^\circ_1,t^\circ_2,\ldots,t^\bullet_1,t_2^\bullet,\ldots]]$. By the
last statement in Lemma \ref{lem:increment}, we have $b_{-1}=1$. It
will prove convenient to gather the $a_k$ and $b_k$ into the Laurent
series
\begin{equation}\label{eq:defxy}
x(z):=\sum_{k\geq -1} a_kz^{-k}\quad\text{ and }\quad y(z):=\sum_{k \geq -1}b_kz^{k}=z^{-1}+\sum_{k\geq 0} b_kz^{k}
\end{equation}
which belong respectively to the rings $\Ring((z^{-1}))$ and
$\Ring((z))$.
These series are actually Laurent polynomials if a bound is imposed on the 
face degrees:

\begin{lemma}[Bounded face degrees]
  \label{lem:boundxypol}
  Fix two positive integers $\dmaxn,\dmaxb$. If we set to zero the weights for black and white faces of degree larger than $\dmaxn$
  and $\dmaxb$, respectively (i.e.\ $t^\bullet_k=0$ for
  all $k>\dmaxn$ and $t^\circ_k=0$ for all $k>\dmaxb$), then $a_k$ and
  $b_k$ vanish for $k \geq \dmaxn$ and $k \geq \dmaxb$, respectively. Consequently, $x(z)$ and $y(z)$ are Laurent polynomials. 
\end{lemma}

\begin{proof}
  By Lemma~\ref{lem:increment}, for any $k \geq 0$, an elementary
  slice of type $\cA_k$ (resp.\ $\cB_k$) has its base edge incident to
  a black (resp.\ white) inner face of degree larger than $k$. Thus,
  its weight vanishes for $k \geq \dmaxn$ (resp.\ $k \geq \dmaxb$).
\end{proof}

\subsubsection{Decomposing a general slice into a sequence of elementary slices}\label{sub:decompositionElementary}
We present in this section the canonical decomposition of general slices into sequences of elementary slices.

\begin{proposition}\label{prop:decComposite}
For any $k\in \mathbb{Z}$ and any $p\in \mathbb{Z}_{>0}$, there exists a weight-preserving bijection between: 
\begin{itemize}
\item Slices of type $\cA$ with increment $k$ and a base of length $p$,
\item and $p$-tuples of elementary slices of type $\cA$, such that the sum of their increments is equal to $k$ (where the weight of a $p$-tuple is defined as the product of the weights of its components).  
\end{itemize}
A similar statement holds for slices of type $\cB$.
\end{proposition}

Before proving this proposition, let us first state its enumerative consequence:

\begin{corollary}\label{cor:genSerComp}
  For any $p\in \mathbb{Z}_{>0}$, we have
\begin{equation}\label{eq:enumIncA}
x(z)^p = \sum_{\substack{\rm \text{ slice of type }\cA\\ \text{with base of length }p}} \bols(\rm)z^{-\inc \rm},
\end{equation}
and 
\begin{equation}\label{eq:enumIncB}
y(z)^p = \sum_{\substack{\rm \text{ slice of type }\cB\\ \text{with base of length }p}} \bols(\rm)z^{\inc \rm}.
\end{equation}
Equivalently,
\begin{equation}\label{eq:pathIncBis}
  P^\circ_{p,h}:=[z^{-h}]x(z)^p \qquad \text{and} \qquad P^\bullet_{p,h}:=[z^{h}]y(z)^p
\end{equation}
are the generating functions of slices with a base of length $p$, increment $h$, and of type $\cA$ or type $\cB$, respectively.
\end{corollary}

This corollary follows from the observation that, for any
$p\in \mathbb{Z}_{\geq 0}$, the Laurent series $x(z)^p$ and $y(z)^p$
are the generating functions of $p$-tuples of elementary slices, with $z$
recording the sum of their increments. Note that, for slices of type
$\cA$, the exponent of $z$ is in fact equal to minus the increment:
this convention will be useful later. Moreover, since $x(z)$ (resp.\ $y(z)$) contains no powers of $z$ larger than $1$
(resp.\ smaller than $-1$), it follows that $P_{p,h}^\circ$ (resp.\
$P_{p,h}^\bullet$) vanishes for $h<-p$. This is consistent with
Remark~\ref{rem:incBound}.

To prove Proposition~\ref{prop:decComposite}, we actually establish a refined bijective
result. To state it, some additional terminology is needed.

\begin{definition}\label{def:DSF}
  A \emph{downward skip-free walk}, or \emph{DSF walk} for short, is a
  finite sequence of integers ${\bm \pi}=(\pi_0,\pi_1,\ldots,\pi_p)$ of
  arbitrary length such that, at each \emph{step} $i=1,\ldots,p$, the
  \emph{increment} $\pi_i-\pi_{i-1}$ is greater than or equal to
  $-1$. The integers $\pi_0,\pi_1,\ldots,\pi_p$ are called the
  \emph{positions} of the walk, with $\pi_0$ the starting position and
  $\pi_p$ the final position.
\end{definition}

\begin{remark}\label{rem:DSFxy}
  Observe that, for any $p\in \mathbb{Z}_{\geq 0}$ and any
  $h\in \mathbb{Z}$, the quantity $P^\circ_{p,h}=[z^{-h}]x(z)^p$
  (resp.\ $P^\bullet_{p,h}=[z^{h}]y(z)^p$) is equal to the generating
  function of DSF walks with $p$ steps, starting position $0$ and
  final position $h$, where each step of increment $k$ (with $k\geq -1$) is assigned a weight $a_k$ (resp.\ $b_k$).
\end{remark}

Since along any directed edge $(u_1,u_2)$ of a slice, we have
$\ell(u_2)\geq \ell(u_1)-1$, it follows that the sequence of labels
read along any directed path of a slice forms a DSF walk. This holds in particular for the base of the slice, which motivates the following
definition.

Fix a DSF walk ${\bm \pi}$ with $p$ steps and starting position $\pi_0=0$. Consider a slice $\rm$ of type $\cA$ with base length $p$ and canonical labeling $\ell$. Let $(\kappa_0=l,\kappa_1,\ldots,\kappa_p=r)$ denote the sequence of corners incident both the outer face and to the base, listed from $l$ to $r$. We say that $\rm$ has \emph{base condition given by $\bm \pi$} if $\pi_i=\ell(\kappa_i)-\ell(l)$ for all $i=0,\ldots,p$.
We define similarly slices of type $\cB$ with base condition given by $\bm \pi$, except that the corners along the base are listed from $r$ to $l$ (rather than from $l$ to $r$), and the labels are shifted by $-\ell(r)$ (instead of $-\ell(l)$).
Proposition~\ref{prop:decComposite} is then a direct consequence of the following result:
\begin{proposition}\label{prop:decCompositePaths}
For any DSF walk ${\bm \pi}=(\pi_0,\pi_1,\ldots,\pi_p)$, there exists a weight-preserving bijection between: 
\begin{itemize}
\item Slices of type $\cA$, with base condition given by ${\bm \pi}$,
\item and $p$-tuples of elementary slices of type $\cA$, such that the increment of the $i$-th slice is equal to $\pi_i-\pi_{i-1}$, for any $i=1,\ldots,p$.
\end{itemize}
A analogous statement holds for slices of type $\cB$.

\end{proposition}
\begin{figure}
\centering
\includegraphics[width=0.9\linewidth]{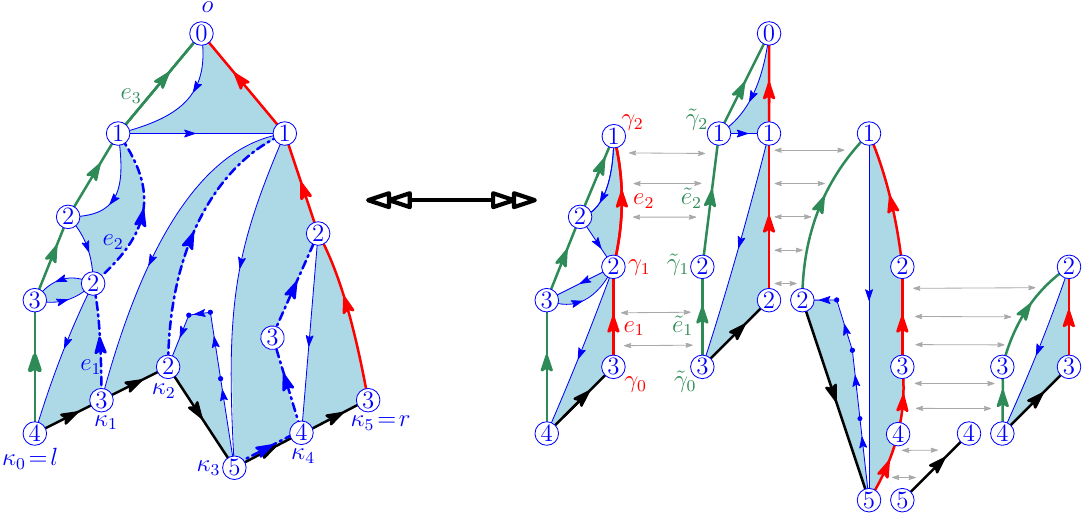}
\caption{\label{fig:decComposite} Decomposition of a slice of type
  $\cA$ (left) with base condition $\pi =(0,-1,-2,1,0,-1)$, and
  $\ell(l)=4$. The leftmost geodesics emanating from the corners
  $\kappa_1,\ldots,\kappa_4$ incident to the base are represented in
  dashed edges.  Cutting along them produces a $5$-tuple of
  elementary slices (right). The canonical labeling of the original
  slice yields a labeling of each elementary slice, which coincides with its own canonical labeling up to a shift.}
\end{figure}

\begin{proof}
  Let us first describe the bijection informally: fix $\rm$ a slice of type $\cA$ with
  base condition given by $\bm \pi$. If $p=1$, $\rm$ is already an
  elementary slice and there is nothing to prove. For $p \geq 2$,
  write $l$, $r$ and $o$ for the left, right and apex corners of
  $\rm$. List the corners incident to the base from $l$ to $r$ as $(\kappa_0=l,\kappa_1,\ldots,\kappa_p=r)$. For each
  $i\in\{1,\ldots,p-1\}$, consider the \emph{leftmost geodesic}
  $\gamma^{(i)}$ from $\kappa_i$ to $o$. As illustrated in
  Figure~\ref{fig:decComposite} for the case $p=5$, cutting along these geodesics yields the desired $p$-tuple of elementary slices.

  Let us provide a detailed description of the decomposition. We proceed inductively on $p$ by cutting along $\gamma^{(1)}$, then
  $\gamma^{(2)}$, etc. Write respectively $(\gamma_0,\ldots,\gamma_k)$
  and $(e_1,\ldots,e_k)$ for the vertices and edges of $\gamma^{(1)}$
  from $\kappa_1$ to $o$.  Let $\gamma_j$ be the first vertex of
  $\gamma^{(1)}$ that belongs to either the left or the right boundary of
  $\rm$ (such a vertex exists since $\gamma_k=\v(o)$ is the apex).  As
  these boundaries are themselves leftmost geodesics, $\gamma^{(1)}$
  coalesces with one of them after $\gamma_j$, and we only need to cut
  along the portion of $\gamma^{(1)}$ from $\gamma_0$ to
  $\gamma_j$. The cutting splits each vertex $\gamma_i$,
  $i=0,\cdots,j$ (resp.\ each edge $e_i$, $i=1,\cdots,j$) into two
  copies: the one on the left of $\gamma^{(1)}$ is still denoted
  $\gamma_i$ (resp.\ $e_i$) while the one on the right is denoted
  $\tilde{\gamma}_i$ (resp.\ $\tilde{e}_i$). See again
  Figure~\ref{fig:decComposite}, where we have $j=2$ and $k=3$.

  The cutting operation separates $\rm$ into two connected components:
  we denote by $\rc_1$ the left one, containing $l=\kappa_0$, and by
  $\tilde \rm$ the right one, containing $r=\kappa_p$. Using
  Remark~\ref{rem:leftmost} and the fact that we have cut along a leftmost
  geodesic, we can see that both $\rc_1$ and $\tilde \rm$ are slices:
  \begin{itemize}
  \item $\rc_1$ is an elementary slice of type $\cA$ with increment
    equal to $\pi_1-\pi_0$, its left corner being $l=\kappa_0$, its
    right corner being the new corner incident to $\gamma_0$ , and its
    apex being either the new corner incident to $\tilde \gamma_j$ if
    $\gamma_j$ is on the left boundary of $\rm$, or $o$ otherwise,
  \item $\tilde \rm$ is a slice with a base of length $p-1$ and base
    condition given by $(0,\pi_2-\pi_1,\ldots,\pi_p-\pi_1)$, its right
    corner being $r=\kappa_p$, its left corner being the new corner
    incident to $\tilde{\gamma}_0$, and its apex being either $o$ if
    $\gamma_j$ is on the right boundary of $\rm$, or the new corner
    incident to $\gamma_j$ otherwise.
  \end{itemize}
For any $k\geq 1$, let us write respectively $f^\circ_k(\cdot)$ and $f^\bullet_k(\cdot)$, for the function giving the number of white and black inner faces. Then we clearly have:
\[
f^\circ_k(\rm)=f^\circ_k(\rc_1)+f^\circ_k(\tilde \rm)\quad \text{ and } \quad f^\bullet_k(\rm)=f^\bullet_k(\rc_1)+f^\bullet_k(\tilde \rm).
\]
Moreover, recall from~\eqref{eq:defWeightslice} that vertices incident to the right boundary of a slice do not contribute to its weight. This ensures that the vertices along $\gamma^{(1)}$ in $\rm$ contribute only to the weight of $\tilde \rm$, and not to that of $\rc_1$. Hence we have: 
\[
  \bar w(\rm)=\bar w(\rc_1) \times \bar w(\tilde \rm).
\]
We then iterate the above decomposition on $\tilde\rm$. Assuming by
induction that $\tilde\rm$ is decomposed into a $(p-1)$-tuple of
slices $(\rc_2,\ldots,\rc_p)$ with
$\bar w(\tilde \rm)=\bar w(\rc_2) \times \cdots \times \bar w(\rc_p)$,
the decomposition of $\rm$ yields the $p$-tuple
$(\rc_1,\rc_2,\ldots,\rc_p)$ satisfying
\[
  \bar w(\rm)=\bar w(\rc_1) \times \bar w(\rc_2) \times \cdots \times \bar w(\rc_p)
\]
as wanted.

Let us now describe the reverse gluing operation: fix a collection
$(\rc_i)_{1\leq i \leq p}$ of elementary slices of type $\cA$ such
that $\inc{\rc_i}=\pi_i-\pi_{i-1}$ for all $i=1,\ldots,p$. As illustrated in Figure~\ref{fig:decComposite}, the operation consists of identifying the right corner of $\rc_i$ with the left corner of
$\rc_{i+1}$, for all $i=1,\ldots,p-1$, and then gluing the right
boundaries to the left boundaries together. For a more detailed
description, let us again proceed by induction on $p$ and assume that
we have already glued together the elementary slices
$\rc_2,\ldots,\rc_p$ into a slice $\tilde\rm$ with a base of length
$p-1$, base condition $(0,\pi_2-\pi_1,\ldots,\pi_p-\pi_1)$ and weight
$\bar w(\tilde \rm)=\bar w(\rc_2) \times \cdots \times \bar w(\rc_p)$:
we now explain how to construct a slice $\rm$ by gluing $\rc_1$ and $\tilde \rm$ together.

Write $o_1,l_1,r_1$ (resp.\ $\tilde o,\tilde l,\tilde r$) for the
apex, left corner, and right corner of $\rc_1$ (resp.\
$\tilde \rm$). We then identify $r_1$ to $\tilde l$, and glue the right boundary of $\rc_1$ to the left boundary of
$\tilde \rm$. More precisely, let us denote by $e_1,\ldots,e_k$ the
edges of the right boundary of $\rc_1$ read from $r_1$ to $o_1$, and
by $\tilde e_1,\ldots,\tilde e_{\tilde k}$ the edges of the left
boundary of $\tilde \rm$ read from $\tilde l$ to $\tilde o$. The
gluing consists in matching the right side of $e_i$ to the left side of
$\tilde e_i$ for all $i=1,\ldots,\min(k,\tilde k)$, see again
Figure~\ref{fig:decComposite}. Note that, if $k > \tilde k$ (resp.\ if
$k < \tilde k$), the edges $e_{\tilde k+1},\ldots,e_k$ (resp.\
$\tilde e_{k+1},\ldots,\tilde e_{\tilde k}$) remain unmatched. We
claim that the resulting map $\rm$ is a slice:
\begin{itemize}
\item its left corner is $l_1$,
\item its right corner is $\tilde r$,
\item its apex is $o_1$ if $k> \tilde k$, $\tilde o$ if
  $k < \tilde k$, or the corner resulting from the merging of $o_1$
  and $\tilde o$ if $k=\tilde k$,
\item its left boundary is the left boundary of $\rc_1$,
  possibly extended by $\tilde e_{k+1},\ldots,\tilde e_{\tilde k}$ if $k < \tilde k$,
\item its right boundary is the right boundary of
  $\tilde \rm$, possibly extended by $e_{\tilde k+1},\ldots,e_k$ if $k>\tilde k$.
\end{itemize}
It may be checked that the left and right boundaries are indeed leftmost
geodesics, as required for a slice. Furthermore, by construction, the
base condition of $\rm$ is ${\bm \pi}=(\pi_0,\pi_1,\ldots,\pi_p)$, and
its weight is
$\bar w(\rm)=\bar w(\rc_1) \times \bar w(\tilde \rm)= \bar w(\rc_1)
\times \bar w(\rc_2) \times \cdots \times \bar w(\rc_p)$ as wanted.

We conclude the proof by noting that the cutting and gluing operations
are indeed inverses of one another. It is clear that cutting followed
by gluing recovers the original slice $\rm$. The converse fact, that
gluing followed by cutting recovers the original $p$-tuple of
elementary slices $(\rc_1,\ldots,\rc_p)$, is slightly more subtle. This may
be checked by induction: in the above description of the gluing process, the path obtained by merging the right boundary of $\rc_1$
with the left boundary of $\tilde \rm$ is precisely the leftmost
geodesic along which we performed the cut.
\end{proof}

\subsubsection{Removing the base edge of an elementary slice}\label{sub:enumElemSlices}

We now describe how to further decompose an elementary slice. In view
of Lemma~\ref{lem:increment}, it is natural to remove its base edge.
The following proposition asserts that the resulting map, which inherits three distinguished outer corners from the original slice, is itself a slice:

\begin{proposition} \label{prop:baseremoval}
For any integers $d \ge 1$ and $k \ge 0$ (resp.\ $k \ge -1$), removing the base edge defines a bijection between:

\begin{enumerate}
    \item elementary slices of type $\cB_k$ (resp.\ $\cA_k$) whose base edge is incident to an inner face of degree $d$, and
    \item general slices of type $\cA$ (resp.\ $\cB$) whose base consists of $d-1$ edges and whose increment is $-k$.
\end{enumerate}
\end{proposition}

\begin{proof}
  \begin{figure}
    \centering
    \includegraphics{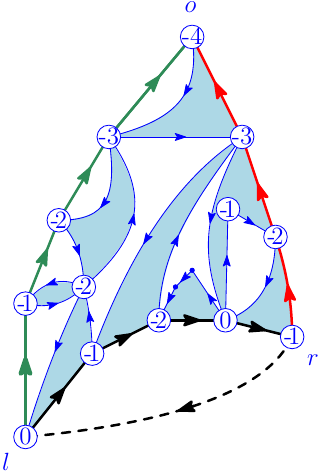}
    \caption{Upon removing the base edge (dashed) of an elementary
      slice of type $\cB_1$, we obtain a slice of type $\cA$ with
      increment $-1$.}
    \label{fig:decSliceBredux}
  \end{figure}
  As can be seen in Figure~\ref{fig:decSliceBredux}, when removing the
  base edge of an elementary slice different from the trivial and the
  empty slice, we obtain a general slice of the opposite type and
  increment, and the length of the new base is equal to the degree of
  the ``merged'' inner face minus one. (Removing the base edge does
  not affect the fact that the left and right boundaries are
  geodesics.)

  It is clear that the operation is injective, since there is a unique
  way to reconnect the endpoints of the base leaving the three
  distinguished corners incident to the outer face.

  Checking that the mapping is surjective is a subtler point. Given
  integers $d \geq 1$ and $k \geq 0$, let $\bar \rm$ be a general
  slice with type $\cA$, base of length $d-1$, and increment
  $-k$. Denote as usual $o,l,r$ its distinguished corners. Write $\rm$
  for the map obtained by adding a directed edge $e$ from $r$ to $l$
  clockwise around the map. Since
  $k=\pd_{\bar \rm}(\v(l),\v(o))-\pd_{\bar \rm}(\v(r),\v(o))$ is
  nonnegative, $e$ does not belong to any directed geodesic towards
  $\v(o)$ in $\rm$, and we have
  $\pd_{\bar \rm}(u,\v(o))=\pd_{\rm}(u,\v(o))$ for any
  $u \in \rV(\bar \rm)=\rV(\rm)$. We deduce that $\rm$ is a slice,
  with the same left and right boundaries as $\bar \rm$, and its base
  consists of the single edge $e$: $\rm$ is an elementary slice of
  type $\cB$. By construction its increment is $k$, and $e$ has an
  inner white face of degree $d$ on its right. So, $\rm$ is the wanted
  preimage of $\bar \rm$.

  The argument is similar, but slightly different, for $\bar \rm$ a
  general slice with type $\cB$, base of length $d-1$, and increment
  $-k \leq 1$. Write $\rm$ for the map obtained by adding a directed
  edge $e$ from $l$ to $r$ counterclockwise around the map. Now, we
  have
  $k=\pd_{\bar \rm}(\v(l),\v(o))-\pd_{\bar \rm}(\v(r),\v(o)) \geq -1$,
  so the key property still holds: the directed distances to $\v(o)$ are the same in $\bar \rm$ and $\rm$. However, $e$ \emph{may} now belong to a geodesic from $\v(l)$ to $\v(o)$ in $\rm$, which is not an issue since there is no uniqueness requirement for the left boundary. All in all, $\rm$ is again the wanted preimage of
  $\bar \rm$.
\end{proof}

Combining the above decomposition with that of
Section~\ref{sub:decompositionElementary} yields a recursive
decomposition of elementary slices, providing a characterization of their generating functions:
\begin{proposition}\label{prop:decslices}
Recall the definitions of $(a_k)$, $(b_k)$, $x(z)$, $y(z)$, $P^\circ_{p,k}$ and $P^\bullet_{p,k}$ given in \eqref{eq:defAB}, ~\eqref{eq:defxy} and~\eqref{eq:pathIncBis}. Then, we have $b_{-1}=1$ and for $k\in \mathbb{Z}_{\geq 0}$:
\begin{equation}\label{eq:decBslices}
b_k=\sum_{d\geq 1}t^\circ_d P^\circ_{d-1,-k} \qquad \text{or equivalently} \qquad [z^k] \Big( y(z) - \sum_{d\geq 1}t^\circ_d x(z)^{d-1}\Big) = 0.
\end{equation}
Similarly, for $k\in \mathbb{Z}_{\geq -1}$, we have
\begin{equation}\label{eq:decAslices}
a_k=t\delta_{k,-1}+\sum_{d\geq1}t^\bullet_d P^\bullet_{d-1,-k} \quad \text{or equivalently} \quad [z^{-k}] \Big( x(z) - t z -\sum_{d\geq1}t^\bullet_d y(z)^{d-1}\Big) = 0.
\end{equation}
These equations form an infinite system of 
equations which determines uniquely the $(a_k)$ and $(b_k)$ as elements of $\Ring:=\mathbb{Q}[[t, t^\circ_1,t^\circ_2,\ldots,t^\bullet_1,t_2^\bullet,\ldots]]$.
\end{proposition}

\begin{proof}
  We have already seen that $b_{-1}=1$ since the unique elementary
  slice of type $\cB_{-1}$ is the trivial slice. Now, fix $k \geq 0$. By Lemma~\ref{lem:increment}, all elementary slices of type $\cB_k$
  have their base edge incident to a white inner face. Hence, we may
  write $b_k=\sum_{d \geq 1} b_k^{(d)}$ where $b_k^{(d)}$ is the
  generating function of those slices for which this white inner face
  has degree $d$. Using Proposition~\ref{prop:baseremoval} and
  Corollary~\ref{cor:genSerComp}, and noting that the removal of the
  base edge modifies the slice weight by a factor $t^\circ_d$, we get
  $b_k^{(d)} = t^\circ_d P^\circ_{d-1,-k}$, which
  yields~\eqref{eq:decBslices} by summing over $d$. The proof
  of~\eqref{eq:decAslices} is entirely similar, except that for $k=-1$
  we must add the contribution of the empty slice, which has weight
  $t$.

  That the equations \eqref{eq:decBslices} and \eqref{eq:decAslices}
  determine all $(a_k)$ and $(b_k)$ follows from the recursive nature
  of our decomposition. Alternatively, this can be seen directly from the
  first forms of the equations by standard arguments on formal power
  series, noting that the formal variables $t$, $t_k^\circ$ and
  $t_k^\bullet$ appear in factor in the right-hand sides.
\end{proof}

\begin{remark}\label{rem:degBorne}
  As we have seen in Lemma~\ref{lem:boundxypol}, under the assumption
  of bounded face degrees only finitely many $a_k$ and $b_k$ are
  nonzero. They are determined by the equations~\eqref{eq:decBslices}
  and~\eqref{eq:decAslices}, which form an algebraic system of
  equations. Indeed, in these equations the sums over $d$ have finite
  support ($t_d^\circ$ and $t_d^\bullet$ vanish for $d$ large enough),
  and $P^\circ_{d-1,-k}$ and $P^\bullet_{d-1,-k}$ are polynomials of
  degree $d-1$ in the $(a_i)$ and $(b_i)$ respectively.
\end{remark}

\begin{remark}[Connection with Bousquet-Mélou--Schaeffer blossomming trees]
  \label{rem:BoSc_connection}
  The recursive equations~\eqref{eq:decBslices}
  and~\eqref{eq:decAslices} satisfied by the generating functions of
  elementary slices are equivalent (upon setting the vertex weight $t$
  to $1$) to \cite[Equations~(2) and (3)]{BoSc02} which characterize
  the generating functions of bicolored blossoming trees. More precisely,
  we have the following correspondence between the notations of the
  present paper and those of \cite{BoSc02}:
  \begin{equation}
    \begin{array}{rclcrcl}
      t_k^\circ & \leftrightarrow & x_k & \qquad & t_k^\bullet & \leftrightarrow & y_k \\
      a_k & \leftrightarrow & \delta_{k,1} + B_{-k}  & \qquad & b_k & \leftrightarrow & \delta_{k,-1} + W_k \\
      x(z) & \leftrightarrow & z + B(z) & \qquad & y(z) & \leftrightarrow & \frac1z + W(z). \\
    \end{array}
  \end{equation}
  This equivalence means that there exists a weight-preserving
  bijection between elementary slices and bicolored blossoming trees. Such a bijection can be defined recursively, or alternatively
  through a variant of the Bousquet-Mélou--Schaeffer opening/closure
  procedures. We will not go into details here, but refer
  instead to~\cite{Bouttier2024a} which discusses, in a different
  context, a bijection between slices and decorated trees that we believe could be adapted to the context of hypermaps.
\end{remark}

\begin{remark}[Connection with Eynard's notations]
  \label{rem:Eynard_connection}
  The generating functions $(a_k)$ and $(b_k)$ of elementary slices
  are closely related with the coefficients $\gamma,\alpha_k,\beta_k$
  appearing in~\cite[Section~8.3.1]{EynardBook}. To make the
  connection precise, we have to perform the change of variables
  $z \to \gamma z$ in Eynard's $x(z)$ and $y(z)$, so that we have the
  identifications
  \begin{equation}
    \label{eq:Eynard_connection}
    a_{-1} = \gamma^2, \qquad a_k = \frac{\alpha_k}{\gamma^k}, \qquad
    b_k = \beta_k \gamma^k, \qquad k \geq 0.
  \end{equation}
  Then, the system of equations of \cite[Theorem~8.3.1]{EynardBook}
  is, upon taking $c=-1$, $a=b=0$ (so that $c_{+-}=1$),
  $V'_1(x) = - \sum_{d \geq 1} t_d^\circ x^{d-1}$, and
  $V'_2(y) = - \sum_{d \geq 1} t_d^\bullet y^{d-1}$, equivalent to our
  Propositions~\ref{prop:decslices} and~\ref{prop:alternativeDec}
  below. In all rigor, Eynard only considers the case where all faces
  have degree at least three, but this restriction is inessential.
\end{remark}

\begin{remark}\label{rem:nonbicolored}
The case of non-bicolored maps can be recovered by setting $t_2^{\bullet}=1$ and $t_k^{\bullet}=0$ for $k\neq 2$ (indeed, by collapsing all the bivalent black faces, we obtain a map with only white faces). By Lemma~\ref{lem:increment}, the increment of an elementary slice of type $\cA$ belongs to $\{-1,0,1\}$ and by~\eqref{eq:decAslices}, we have:
\[
  a_{-1} = t + b_1,\quad a_0=b_0\quad \text{and} \quad a_{1}=b_{-1}=1.
\]
Then, performing the change of variables $a_{-1}=R$ and $a_0=S$, so that $x(z)=z^{-1} + S + Rz$, Equations~\eqref{eq:decBslices} for $k=0,1$ match \cite[Equation~(2.5)]{census,hankel} with $t_k^\circ=g_k$.
\end{remark}

\subsubsection{An alternative decomposition for elementary slices of type $\cA_{-1}$}\label{sub:decAlternative}

The decomposition of elementary slices presented in the previous section allows to decompose elementary slices of type $\cA$ into a sequence of elementary slices of type $\cB$ and vice-versa. It turns out that for elementary slices of type $\cA_{-1}$, there exists an alternative decomposition which involves only elementary slices of type $\cA$.
It implies the following enumerative result:

\begin{proposition}\label{prop:alternativeDec}
We have the relation:
\begin{equation}\label{eq:decAlternative}
a_{-1}=t+a_{-1}\sum_{d\geq 1}t^\circ_d P^\circ_{d-1,1}  \quad \text{or equivalently} \quad [z^{-1}] \Big( y(z) - \frac{t}{a_{-1}z} - \sum_{d\geq1}t^\circ_d x(z)^{d-1}\Big) = 0.
\end{equation}
\end{proposition}

\begin{proof}
  \begin{figure}[t]
    \centering
    \includegraphics[page=1]{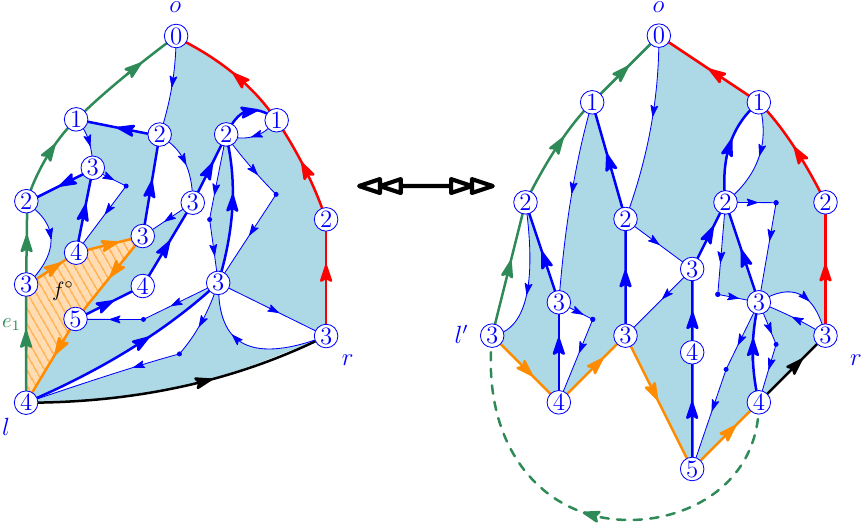}
    \caption{\label{fig:decAlternative} The alternative decomposition
      of a non-empty elementary slice $\rm$ of type $\cA_{-1}$ (left),
      obtained by deleting the first edge of its left boundary. As
      explained in the text, the resulting slice $\bar \rm$ (right)
      has the same canonical labeling.}
  \end{figure}
  Let $\rm $ be an elementary slice of type $\cA_{-1}$, with
  distinguished corners $o,l,r$ and canonical labeling $\ell$. By
  definition of the increment, we have $\ell(l)=\ell(r)+1 > 0$ hence
  the left boundary has positive length.

  If $\rm$ is the empty slice, its weight is $t$, which
  accounts for the $t$ in~\eqref{eq:decAlternative}. Otherwise, the first
  edge $e_1$ of the left boundary has a white inner face on its
  right\footnote{The edge $e_1$ cannot have the outer face on its
    right since the left and right boundaries only meet at the apex,
    and since $e_1$ cannot be the base edge either by
    Lemma~\ref{lem:increment}.}: we denote it by $f^\circ$, and its
  degree by $d$.

  Let $\bar \rm$ be the map obtained from $\rm$ by deleting $e_1$. Let
  $l'$ be the corner of $\bar \rm$ corresponding to the endpoint of
  $e_1$, see Figure~\ref{fig:decAlternative}. Then, $\bar \rm$ endowed
  with the distinguished corners $o,l',r$ is a slice of type $\cA$
  such that:
\begin{itemize}
\item it has the same right boundary as $\rm$, 
\item it has the same left boundary, save from $e_1$,
\item its base is the concatenation of the contour of $f^\circ$ from the head of $e_1$ to its tail and of the base of $\rm$. 
\end{itemize}
Moreover, $\bar \rm$ and $\rm$ admit the same canonical labeling of their vertices. Indeed, since $\rm$ is of type $\cA_{-1}$, the path obtained by concatenating its base with its right boundary is a geodesic from the root vertex to the apex, which implies that deleting $e_1$ does not change the label of $\v(l)$. Therefore, the boundary path condition of $\bar \rm$ is the concatenation of a DSF walk of length $d-1$ and increment $1$ (corresponding to the contour of $f^\circ \setminus \{e_1\}$), followed by a downstep (corresponding to the base of $\rm$). 

The weights of the two slices differ only by a factor $t^\circ_d$, corresponding to the weight of $f^\circ$. Since the construction is easily seen to be invertible, by adding an edge and rerooting, the result follows from Corollary~\ref{cor:genSerComp}.
 \end{proof}

\begin{remark}[Connection with mobiles]\label{rem:mobiles_connection}
  Equation~\eqref{eq:decAlternative} may be rewritten as
  \begin{equation}
    a_{-1} = \frac{t}{1-\sum_{d=1}^{\dmaxb}t^\circ_d P^\circ_{d-1,1}} =
    \frac{t}{1- [z^{-1}] \Big( \sum_{d=1}^{\dmaxb}t^\circ_d x(z)^{d-1}\Big)},
  \end{equation}
  which corresponds combinatorially to an iteration of the
  decomposition used in the proof of
  Proposition~\ref{prop:alternativeDec}. Together with the recursive
  decomposition of slices of types $\cA_k$ and $\cB_k$ for $k \geq 0$
  considered in Section~\ref{sub:enumElemSlices}, we obtain an
  alternate recursive decomposition of slices which is equivalent to
  that of the generalized mobiles considered
  in~\cite[Section~3.3]{mobiles}.
\end{remark}

\begin{remark}
Even in the case of non-bicolored slices, this alternative decomposition is new. Following the notations of Remark~\ref{rem:nonbicolored},~\eqref{eq:decAlternative} becomes: 
\[
  R= t+R\sum g_{k}[z^{-1}](z^{-1}+S+Rz)^{k-1},
\]
or equivalently: 
\begin{equation}
 R = \frac{t}{1-\sum g_{k}[z^{-1}](z^{-1}+S+Rz)^{k-1}}.
\end{equation}
An equation of this form was previously derived via mobiles, see~\cite[Section~4.2]{mobiles}. In contrast, from Proposition~\ref{prop:decslices}, we obtain the equation: 
\begin{equation}
  R= t+\sum g_{k}[z](z^{-1}+S+Rz)^{k-1},
\end{equation}
which is equivalent to the former one, but was originally derived via blossoming trees \cite{census}.
\end{remark}

\section{Wrapping slices}
\label{sec:wrapping}
In the previous section, we described how a general slice can be decomposed as a sequence of elementary slices. In this section, we describe a second fundamental construction for general slices, which corresponds to gluing their left and right boundaries together. This allows us to derive bijectively enumerative formulas for pointed rooted hypermaps, pointed disks and so-called ``trumpets'' and ``cornets'' with monochromatic boundary conditions.

\subsection{Pointed rooted hypermaps}\label{sub:pointedRootedMaps}

A \emph{pointed rooted} hypermap is a hypermap which has no boundary,
but which has instead both a marked vertex (the
\emph{pointed vertex}) and a marked edge (the \emph{root
  edge}). We will show that such a map is in bijection with a pair of
slices with the same increment but with opposite types. This
generalizes to the case of hypermaps a construction described
in~\cite[Appendix A]{hankel}.

Given a pointed rooted hypermap $\rm$, let $v^*$ denote its
pointed vertex. The \emph{canonical labeling} of $\rm$ is the
function
$\ell:V(\rm)\rightarrow \mathbb{Z}_{\geq 0}\,;\,u \mapsto
\pd(u,v^*)$. For an edge $e$, we define its \emph{increment} $\mathrm{inc}(e)$ as the difference between the label of its endpoint and that of its origin, with respect to the canonical orientation of $e$. The \emph{increment} of $\rm$ is then the increment of its root edge. The weight of $\rm$ is defined as in~\eqref{eq:Boltzmannweight} except that the pointed vertex receives
no weight $t$.

\begin{proposition}\label{prop:akbk}
  For any $k \geq -1$, there is a weight-preserving bijection between the set of pointed rooted hypermaps with increment $k$ and the set of pairs $(\mathrm{a},\mathrm{b})$ such that $\mathrm{a}$ is a slice of type $\cA_k$, $\mathrm{b}$ is a slice of type $\cB_k$, and at least one of them has an inner face. As a consequence, the generating series (with no weight for the pointed vertex) of pointed rooted hypermaps with increment $k$ is equal to $a_kb_k-t\delta_{k,-1}$. 
\end{proposition}
\begin{figure}
\centering
\includegraphics[width=0.8\linewidth,page=4]{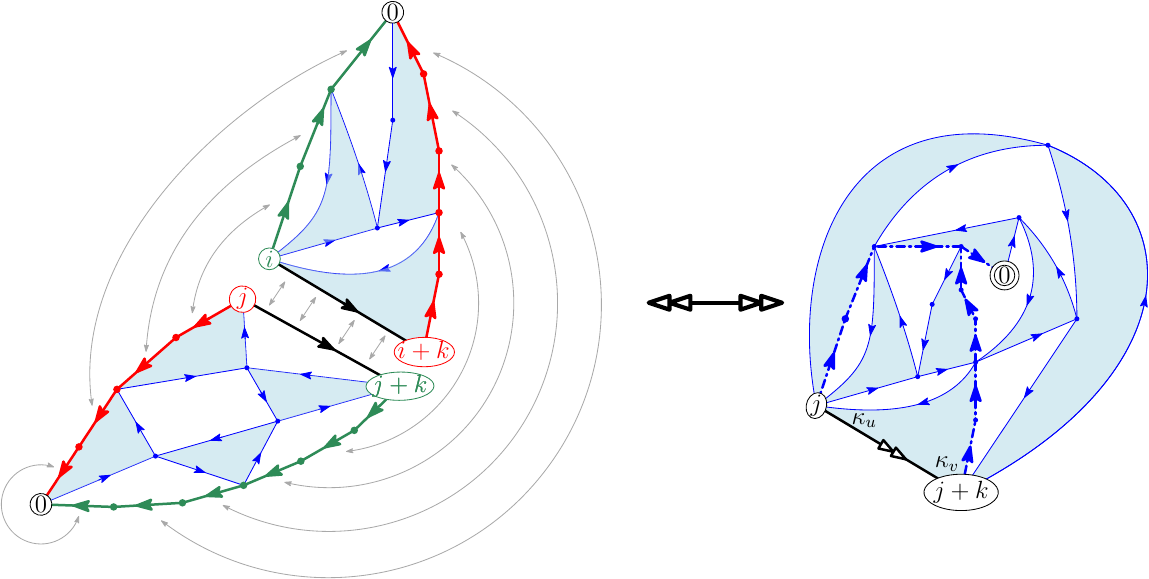}
\caption{Gluing of a slice of a type $\cA_k$ and of a slice of type $\cB_k$ (left), to obtain a pointed rooted hypermap with increment $k$ (right). Here, we have $k=2$. In this example, using the notation of the proof of Proposition~\ref{prop:akbk}, we have $i=3$ and $j=4$. Since $j\geq i$, some boundary edges of the slice of type $\cB_k$ are matched one to another.\label{fig:akbk}}
\end{figure}
\begin{proof}
Fix $k\geq -1$, and let $\mathrm{a}$ and $\mathrm{b}$ be two slices of respective types $\cA_k$ and $\cB_k$, at least one of them having an inner face.
Let $i,j\geq 0$ be such that the base edges of $\mathrm{a}$ and $\mathrm{b}$ are respectively labeled $(i,i+k)$ and $(j,j+k)$.

Then, as illustrated in Figure~\ref{fig:akbk}, we glue the base edge of $\mathrm{a}$ to that of $\mathrm{b}$, and we sequentially glue the green (resp. red) edges of $\mathrm{a}$ with the red (resp. green) edges of $\mathrm{b}$. If $j \geq i$ (resp.\ $i \geq j$), then exactly $j-i$ (resp.\ $i-j$) green edges and red edges of $\mathrm{b}$ (resp.\ $\mathrm{a}$) remain unmatched, and we glue them together to finish the construction. As there is at least one face, the resulting object is a planar map, which we denote by $\rm$. It is a pointed rooted map, the pointed vertex being the apex of $\mathrm{b}$ (resp.\ $\mathrm{a}$), and the root edge being the edge formed by the gluing of the base edges of $\mathrm{a}$ and $\mathrm{b}$.

Since all the gluings performed to construct $\rm$ respect the canonical orientation of the edges, $\rm$ is a hypermap. Its white and black faces consist of the reunion of the white and black inner faces of $\mathrm{a}$ and $\mathrm{b}$. Recall that $\rV'(.)$ stands for the set of vertices not incident to the right boundary of a slice. Then, we observe that $\rV'(\mathrm{a}) \cup \rV'(\mathrm{b})$ is in bijection with $\rV(\rm)$ minus the pointed vertex. Hence, the definition of the weight of a slice ensures that
\begin{equation}
\bols(\mathrm{a})\bols(\mathrm{b})=
t^{|V(\rm)|-1}\prod_{f\in \rFb(\rm)}t^\circ_{\deg(f)}\prod_{f\in \rFn(\rm)}t^\bullet_{\deg(f)}
\end{equation}
which is the wanted weight for $\rm$.

Reciprocally, to recover a pair of slices from a pointed rooted hypermap $\rm$, we proceed as follows. Write $(u,v)$ for the root edge of $\rm$, let $\kappa_u$ and $\kappa_v$ be the corners incident to $u$ and $v$ respectively, and situated just before (respectively after) the root edge in clockwise  order. To decompose $\rm$, we cut along the leftmost geodesics from $\kappa_u$ to $v^\star$ and from $\kappa_v$ to $v^\star$. The fact that it produces two slices follows by the same argument as the proof of Proposition~\ref{prop:decslices}. Since a pointed rooted hypermap has at least one face, at least one of the resulting slices must contain an inner face.

These two operations are clearly inverse to each another, and hence define a bijections between pointed rooted hypermaps and pairs of slices with the same increment and opposite types. The restriction that at least one of the two slices has an inner face excludes the pair consisting of the empty slice and the trivial slice, which accounts for the subtraction of $-t\delta_{k,-1}$ in the proposition.
\end{proof}

As a direct consequence of this proposition, we obtain the following enumerative results:
\begin{corollary}\label{cor:ahbh}
The generating function of pointed rooted hypermaps is equal to 
\begin{equation}
\sum_{k\geq -1}a_kb_k-t
\end{equation}
where we recall that the pointed vertex receives no weight $t$. Moreover, we have
\begin{equation}\label{eq:kakbk}
\sum_{k\geq -1}ka_kb_k=t
\end{equation}
\end{corollary}
\begin{proof}
The first formula follows directly from Proposition~\ref{prop:akbk} by summing over $k$. 
Still by the same proposition, the second formula amounts to the identity
\begin{equation}
  \sum_{\substack{\rm \text{ pointed}\\ \text{rooted hypermap}}} \inc{\rm} w(\rm) = 0
\end{equation}
where $\inc{\rm}$ and $w(\rm)$ denote respectively the increment and the weight of $\rm$, as defined above. We claim that we actually have the stronger identity
$\sum_{\rm \in \bar{\rm}} \inc{\rm} w(\rm) = 0$
which holds for any pointed hypermap $\bar{\rm}$, where the
notation $\rm \in \bar{\rm}$ means that $\rm$ is a pointed rooted
hypermap obtained by selecting a root edge among the edges of
$\bar\rm$. Note that $w(\rm)$ does not depend on the choice of the
root edge, and recall that $\inc\rm$ is, by definition, the difference
between the labels of the endpoint and the origin of the root edge. Thus, the
stronger identity boils down to
\begin{equation} \label{eq:incsum}
  \sum_{e \in \rE(\bar \rm)} \inc{e} = 0, 
\end{equation}
where we recall that $\inc{e}$ is the difference of labels between the two endpoints of $e$. 
In the sum above, let us group
together all edges incident to a same black face. W
claim that their increments sum to zero, since these edges form a
closed counterclockwise directed cycle, along which the labels form a
DSF walk starting and ending at the same position. We
get~\eqref{eq:incsum} by summing over all black faces of $\bar \rm$, which completes the proof.
\end{proof}

\begin{remark}
  The equality~\eqref{eq:kakbk} was already proved in Proposition 2.5 of~\cite{BergereEtAl} by a similar argument, involving mobiles instead of slices.
\end{remark}

\subsection{Pointed disks}
\label{sub:pointed}
We now turn to the case of pointed disks, which may be obtained by
wrapping slices of zero increment:

\begin{proposition}\label{prop:bijPointedRooted}
There exists an explicit weight-preserving bijection between: 
\begin{itemize}
\item Pointed hypermaps with a monochromatic white (respectively black) boundary of degree $p$, in which the pointed vertex carries no weight; and
\item Slices of type $\cA$ (respectively $\cB$) with increment $0$ and base length $p$.
\end{itemize}
\end{proposition}

Combining this proposition with Corollary~\ref{cor:genSerComp}, we obtain:
\begin{corollary}\label{cor:pointed}
  For any $p\in \Z_{\geq 0}$, the generating functions
  $\frac{\partial F_p^{\circ}}{\partial t}$ and
  $\frac{\partial F_p^{\bullet}}{\partial t}$ of pointed hypermaps with
  a monochromatic boundary of degree $p$, respectively white or black,
  are given by:
\begin{equation}\label{eq:pointedEnumF}
\frac{\partial F_p^{\circ}}{\partial t}= P^\circ_{p,0}=[z^0]x(z)^p \qquad \text{and}\qquad \frac{\partial F_p^{\bullet}}{\partial t} = P^\bullet_{p,0}=[z^0]y(z)^p.
\end{equation}
\end{corollary}

To prove Proposition~\ref{prop:bijPointedRooted}, we proceed as for Proposition~\ref{prop:decComposite} by establishing a refined bijective correspondence. 
\begin{definition}\label{def:boundaryLabelCondition}
Fix a DSF walk ${\bm \pi}=(\pi_0,\pi_1,\ldots,\pi_p)$ of length $p$ starting and ending at $0$. 
We say that a pointed hypermap $(\rm,v^\star)$ with a monochromatic boundary has \emph{label boundary condition given by $\bm \pi$} if
\[
  \bm \pi = \Big(\pd(\rho_i,v^\star)-\pd(\rho,v^\star)\Big)_{0\leq i \leq p},
\]
where $(\rho_i)$ is the sequence of corners incident to the boundary, started at the marked corner $\rho$ and listed in the order given by the orientation of the edges, see Figure~\ref{fig:pointedRooted}. 
\end{definition}

Since the label boundary condition of any hypermap is a DSF walk, Proposition~\ref{prop:bijPointedRooted} is a direct corollary of the following result:
\begin{proposition}\label{prop:decPointedRooted}
For $p\in \mathbb{Z}_{\geq 0}$, fix a DSF walk ${\bm \pi}$ of length $p$ starting and ending at $0$. Then, there exists a weight-preserving bijection between: 
\begin{itemize}
\item Pointed hypermaps with a monochromatic white (resp. black) boundary, label boundary condition given by $\bm \pi$, and no weight assigned to the pointed vertex; and,
\item Slices of type $\cA$ (resp. $\cB$), and base condition given by $\bm \pi$. 
\end{itemize}
\end{proposition}

\begin{figure}[t!]
\centering
\includegraphics[width=0.9\linewidth]{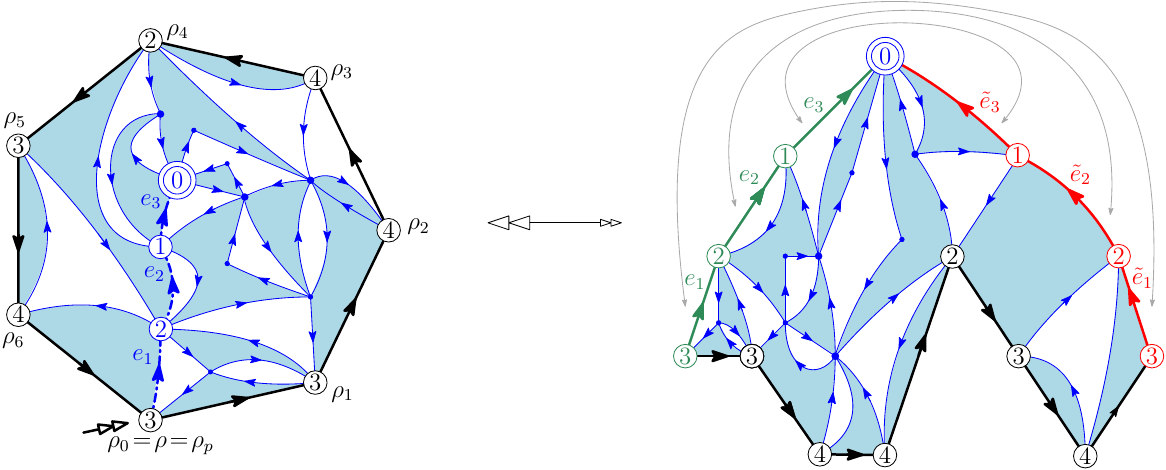}
\caption{\label{fig:pointedRooted}Illustration of the opening of a pointed and rooted hypermap into a slice of increment 0. The label boundary condition of this map is given by $\bm\pi = (0,0,1,1,-1,0,1,0)$.\\
The pointed vertex is encircled and the leftmost geodesic from the root corner to the pointed vertex is represented in bold dashed edges.}
\end{figure}
\begin{proof}
Start from a pointed and rooted hypermap $(\rm,v^\star)$ equipped with its canonical labeling. Write $(e_1,\ldots,e_k)$ for the sequence of edges in the the leftmost geodesic from the root corner $\rho$ to the pointed vertex $v^\star$ and let $(\gamma_0=\v(\rho),\ldots,\gamma_k=v^\star)$ be the corresponding sequence of vertices, see Figure~\ref{fig:pointedRooted}.

To obtain a slice, we proceed as in the proof of Proposition~\ref{prop:decCompositePaths} and cut along this geodesic.  
 The map obtained is clearly a slice, with base condition given by $\bm \pi$. Moreover, it has the same weight as $\rm$ (recall from~\eqref{eq:defWeightslice} that the vertices incident to the right boundary of a slice do not contribute to the weight).
\medskip

Reciprocally, given a slice with $0$ increment, we identify its left and right boundaries (which have the same length since the increment is $0$). The corners $l$ and $r$ of the slice are identified and we root the resulting hypermap at this corner. The sequence of identified edges forms the leftmost geodesic from the root corner to the pointed vertex. Hence, this is the inverse operation of the opening procedure described above, which concludes the proof.
\end{proof}

\begin{remark}
  In a hypermap, the number of edges is equal to the number of corners
  incident to white faces. Therefore, pointed rooted hypermaps (as
  defined in Section~\ref{sub:pointedRootedMaps}) are in bijection
  with pointed hypermaps with a marked corner incident to a white
  face, i.e.\ pointed hypermaps with a white boundary of arbitrary
  degree.  Of course the previous assertions remain true if we replace
  white by black. Thus, the generating function of pointed rooted hypermaps
  is given by the three expressions
  \begin{equation}
    \sum_{k\geq -1}a_kb_k-t = \sum_{p \geq 1} t_p^\circ \frac{\partial F_p^{\circ}}{\partial t} = \sum_{p \geq 1} t_p^\bullet \frac{\partial F_p^{\bullet}}{\partial t}.
  \end{equation}
  where the first is Corollary~\ref{cor:ahbh}, and in the second and
  third we need to incorporate the weight $t_p^\circ$ and
  $t_p^\bullet$ respectively of the boundary face.
  Corollary~\ref{cor:pointed} thus provides another way to enumerate
  pointed rooted hypermaps.
\end{remark}

\subsection{Trumpets and cornets}\label{sub:trumpets}
We investigate in this section the wrapping of general slices, with positive or negative increment. By doing so, we will obtain hypermaps with two boundaries satisfying some additional constraints, which we now define.

\begin{definition}\label{def:tight}
Given a planar hypermap with two boundaries, a \emph{separating cycle} is a directed cycle of edges that separates the hypermap into two regions, each containing one of the boundaries.
A monochromatic boundary is said to be \emph{tight} if its contour is a separating cycle of minimal length among all separating cycles oriented in the same direction. Moreover, it is said to be \emph{strictly tight} if its contour is the unique such cycle.
\end{definition}

\begin{definition}\label{def:trumpetCornet}
A \emph{trumpet} is a hypermap with two monochromatic boundaries, the first one being rooted, and the second one being tight, unrooted and \emph{black}. A \emph{cornet} is a hypermap with two monochromatic boundaries, the first one being rooted, and the second one being strictly tight, unrooted and \emph{white}.

The color of the rooted boundary is a priori unspecified: if it is black (respectively white) the trumpet or the cornet is said to be \emph{black} (respectively \emph{white}).
We define the \emph{perimeter} of a cornet or a trumpet as the degree of its rooted boundary, and its \emph{girth} as the degree of its unrooted (tight or strictly tight) boundary. 
\end{definition}

\begin{figure}[h!]
\centering
\includegraphics[width=0.8\linewidth]{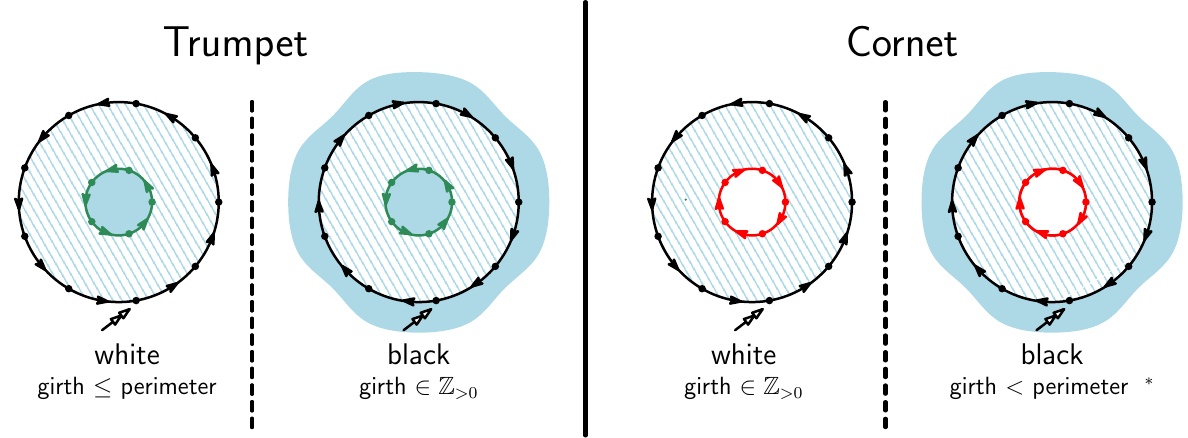}
\caption{Schematic representation of trumpets and cornets. Tight boundaries are represented in green and strictly tight boundaries in red.\\
\emph{(${}^*$ In a black cornet, the perimeter is equal to the girth if and only if the cornet is reduced to a cycle of edges.
\label{fig:TrumpetCornet}})}
\end{figure}

Note that, for a white trumpet, the girth is not larger than the
perimeter, since the contours of the two boundaries are oriented in
the same direction. The same property holds for a black cornet, which
in the case of equality is necessarily reduced to a single directed cycle. For
black trumpets and white corners, the girth may be larger than the
perimeter, since the contours of the two boundaries are oriented in opposite directions.

The main result of this section is the following bijective correspondence, illustrated on Figure~\ref{fig:trumpets}: 
\begin{proposition}[Bijection between trumpets/cornets and
  slices]\label{prop:trumpets}
For any $p,k\in \mathbb{Z}_{> 0}$, there exists a weight-preserving bijection between: 
\begin{itemize}
\item White (respectively black) trumpets with perimeter $p$ and girth $k$, and
\item Slices of type $\cA$ (respectively $\cB$), increment $-k$ (respectively $k$), and with a base of length $p$.
\end{itemize}
There is also a weight-preserving bijection between:
\begin{itemize}
\item White (respectively black) cornets, with perimeter $p$ and girth $k$, and where no weight is given to the vertices incident to the strictly tight boundary; and,
\item Slices of type $\cA$ (respectively $\cB)$, increment $k$ (respectively $-k$),  and with a base of length $p$.
\end{itemize}
\end{proposition}

\begin{remark}
By considering a pointed vertex as a (degenerate) strictly tight boundary of degree 0, this proposition also makes sense for $k=0$, and we retrieve the statement of Proposition~\ref{prop:bijPointedRooted} for pointed disks.
\end{remark}

Proposition~\ref{prop:trumpets} has the following enumerative corollary:

\begin{corollary}
  \label{cor:trumpenum}
  Given $p,k\in \mathbb{Z}_{> 0}$,
  the generating function of white (respectively black) trumpets with perimeter $p$ and girth $k$ is equal to $P^\circ_{p,-k}=[z^k]x(z)^p$
  (respectively $P^\bullet_{p,k}=[z^k]y(z)^p$). 

  Similarly, the generating
  function of white (respectively black) cornets, with perimeter $p$ and girth $k$, where no weight is given to the vertices incident to the strictly tight boundary, is equal to
  $P^\circ_{p,k}=[z^{-k}]x(z)^p$ (respectively
  $P^\bullet_{p,-k}=[z^{-k}]y(z)^p$).
\end{corollary}

\begin{remark}
  In the particular case $p=1$, where $P^\circ_{1,k}=a_k$ and $P^\bullet_{1,k}=b_k$ for $k \geq -1$, Corollaries~\ref{cor:trumpenum} and \ref{cor:pointed} give a combinatorial interpretation of $a_k$ and $b_k$ in terms of hypermaps with monochromatic boundaries. Namely:
  \begin{itemize}
  \item for $k \geq 1$, $a_k$ corresponds to white cornets with perimeter $1$ and girth $k$,
  \item for $k \geq 1$, $b_k$ corresponds to black trumpets with perimeter $1$ and girth $k$,
  \item $a_{-1}$ corresponds to white trumpets with perimeter $1$ and girth $1$,
  \item $b_{-1}=1$ corresponds to the unique black cornet with perimeter $1$ and girth $1$, which is reduced to a directed loop.
  \item and $a_0$ and $b_0$ correspond to pointed disks whose boundaries
    have length $1$ and are white and black, respectively.
  \end{itemize}
\end{remark}

The remainder of this section is devoted to the proof of
Proposition~\ref{prop:trumpets}, which is inspired
from~\cite[Sections~7 and 9.3]{irredmaps}
and~\cite[Section~4.4]{Bouttier2024}.

\subsubsection{From slices to trumpets/cornets}

We start by proving the easier direction of Proposition~\ref{prop:trumpets}, which is going from slices to trumpets or cornets. Consider a slice $\rs$ of type $\cA$ with a base of length $p$. Write $k$ for its increment, and suppose first that $k>0$. This implies that its right boundary has $k$ more edges than its left boundary, or equivalently that $\rs$ has $k$ more red edges than green edges. The \emph{wrapping} of $\rs$ is defined as the hypermap obtained by gluing sequentially green edges to red edges starting from $l$ and $r$ respectively. Note that, if we denote by $\ell$ the canonical labeling of $\rs$, a vertex $u$ of the left boundary is identified with the vertex $v$ from the right boundary such that $\ell(v)=\ell(u)+k$. 
The $k$ additional red edges are left unglued, see Figure~\ref{fig:trumpets}. The hypermap obtained has two white monochromatic boundaries: the first one, of degree $p$, inherits the rooting of the original slice whereas the second one, of degree $k$, is unrooted.

To prove that this hypermap is a white cornet as wanted, we need to
check that the unrooted boundary is strictly tight. Any separating cycle oriented in the same direction as its contour, corresponds in $\rs$ to a directed path $\gamma$ from a vertex -- say $v$ -- of the right boundary of $\rs$ to a vertex -- say $u$ -- of its left boundary, such that $u$ and $v$ are identified in the wrapping operation, and so we have $\ell(v)=\ell(u)+k$. The length of $\gamma$ is at least $k$: indeed the concatenation of $\gamma$ with the portion of the left boundary from $u$ to the apex $o$ is a directed path from $v$ to $o$, hence its length $|\gamma|+\ell(u)-\ell(o)$ is at least $\ell(v)-\ell(o)$, which is the length of the right boundary from $v$ to $o$. Moreover, since the right boundary is the unique geodesic from $v$ to $o$ and since the left and right boundaries only meet at $o$, we see that if $|\gamma|=k$, then necessarily $u$ is the apex and $\gamma$ is the unmatched portion of the right boundary, which becomes the contour of the unrooted boundary after wrapping. Hence, the unrooted boundary is strictly tight as claimed.

The case of a slice of type $\cB$ with a negative increment is completely similar, except that we now get a black cornet since the rooted boundary of degree $p$ is black. In the remaining two cases—namely, a slice of type $\cA$ with negative increment or a slice of type $\cB$ with positive increment—the left boundary has $k$ more edges than the right boundary. Consequently, after the wrapping operation, $k$ green edges remain unmatched, forming the contour of an unrooted black monochromatic boundary. The argument that this boundary is tight proceeds along the same lines as above; however, multiple geodesics may exist between vertices on the left boundary. As a result, the unrooted boundary is tight but not necessarily strictly tight, yielding a trumpet rather than a cornet.

\begin{figure}
\centering
\includegraphics[width=0.9\linewidth,page=1]{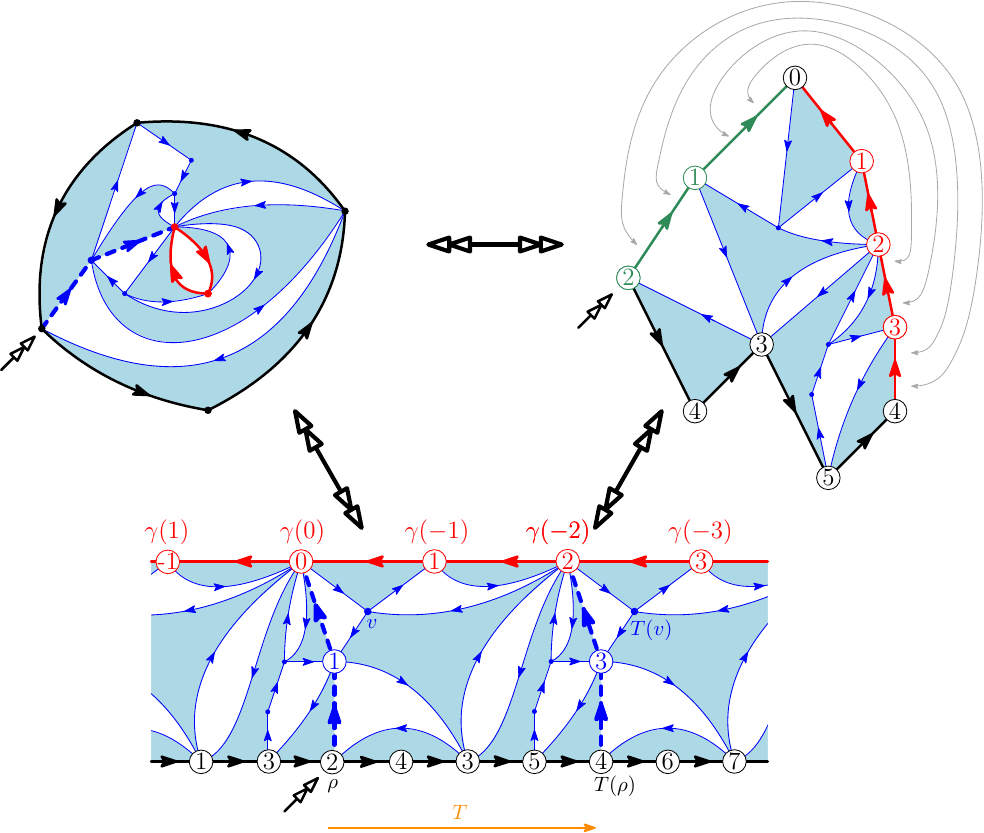}
\caption{\label{fig:trumpets}Illustration of the bijective correspondence between a white cornet with perimeter $p=4$ and girth $k=2$ (top left) and a slice of type $\cA$ with increment $k=2$ and base of length $p=4$ (top right). We pass from the slice to the cornet by gluing green and red edges as displayed, giving rise to the path shown in blue dashed edges in the cornet.\\
  The reverse construction passes through the universal cover of the cornet (bottom). The labels on vertices correspond to the value of the Busemann function $\dB_\gamma$ associated with the bi-infinite directed geodesic $\gamma$, shown in red, which is the preimage of the tight boundary of the cornet. The translation $T$ is an automorphism of the universal cover, which increases the value of $\dB_\gamma$ by $k$. The paths shown in blue dashed edges correspond to the initial segments of the leftmost infinite geodesics towards $\gamma$ that start from an arbitrary preimage of the root corner, and from its translate by $T$. By cutting along them, we get a fundamental domain which is the slice we are looking for.}
\end{figure}

\subsubsection{Infinite directed geodesics and directed Busemann functions}
\label{sec:dirBuse}

To prove the other direction of Proposition~\ref{prop:trumpets} from trumpets to slices, we
first introduce in this section some additional material and concepts
related to infinite directed graphs. Indeed, the slice decomposition
of annular maps, originally introduced in~\cite{irredmaps}, requires
considering their universal cover which is an infinite map. It was
realized in~\cite{Bouttier2024} that the decomposition is most
accurately described using the concept of Busemann function. In the context of the present work, we consider hypermaps whose edges are directed, which necessitates extending the concept of the Busemann function to the directed setting in order to accurately describe the corresponding slice decomposition. 

Throughout this section, we consider an infinite directed graph $\rg$,
planarity playing no role in our discussion. We denote by $\rV(\rg)$ the vertex set of $\rg$, and by $\pd$ the directed graph distance in $\rg$. That is, given $u, v \in \rV(\rg)$, $\pd(u,v)$ is the minimal number of edges in a directed path from $u$ to $v$, with the convention that $\pd(u,v) = +\infty$ if no such path exists. All facts
mentioned in the first paragraph of Section~\ref{sub:odgprelim}, and
in particular the oriented triangle inequality
\eqref{eq:triangleInequality}, apply here mutatis mutandis.
To simplify the presentation, we assume that $\rg$ is simple, so that a directed path can be coded as a sequence of vertices. The extension to multigraphs is straightforward.

\begin{definition}\label{def:infiniteGeo}
  An \emph{infinite directed geodesic} is an infinite directed path
  $\gamma=(\gamma(s))_{s \in \mathbb{Z}_{\geq 0}}$, such that for any
  $s\leq t$, we have $\pd(\gamma(s),\gamma(t))=t-s$. Moreover,
  $\gamma$ is called an \emph{infinite directed strict geodesic} if,
  for any $s\leq t$, $\gamma|_{[s,t]}$ is the unique directed geodesic
  from $\gamma(s)$ to $\gamma(t)$.

\emph{Bi-infinite directed (strict) geodesics} are defined similarly upon replacing $\mathbb{Z}_{\geq 0}$ by $\mathbb{Z}$ in the above definitions. 
\end{definition}

\begin{figure}
  \centering
  \includegraphics[width=.7\textwidth]{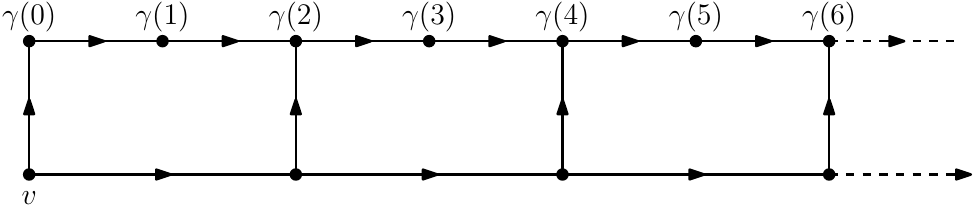}
  \caption{An infinite directed graph with an infinite directed
    geodesic $\gamma$ such that the directed Busemann function
    $\dB_\gamma$ takes negative infinite values: this is the case for
    instance for the vertex $v$ on the bottom left, which is such that
    $\pd(v,\gamma(2s))=s+1$ for all $s$, and hence
    $\pd(v,\gamma(2s))-2s \to -\infty=\dB_\gamma(v)$.}
  \label{fig:InfBusemann}
\end{figure}

Let $\gamma$ be an infinite or bi-infinite directed geodesic, and let
$v\in \rV(\rg)$. The oriented triangle inequality asserts
that
  \begin{equation}\label{eq:BusemannDecreasing}
\pd(v,\gamma(t))\leq \pd(v,\gamma(s))+\pd(\gamma(s),\gamma(t))= \pd(v,\gamma(s))+t-s
\end{equation}
for any $s \leq t$. In other words, the function
$s\mapsto\pd(v,\gamma(s))-s$ is nonincreasing, and thus admits a limit
for $s \to +\infty$. Note that the function takes values in
$\Z \cup \{+\infty\}$, with the natural order, using the convention
that $+\infty-s=+\infty$ for all $s \in \Z$. The limit itself belongs
to $\Z \cup \{+\infty,-\infty\}$: it is equal to $+\infty$ if and only
if $\pd(v,\gamma(s))=+\infty$ for all $s$, that is no vertex on
$\gamma$ is accessible from $v$ by a directed path;
Figure~\ref{fig:InfBusemann} displays an example where the limit is
$-\infty$. This discussion leads to the following:
\begin{definition}
  \label{def:Busemann}
  Given an infinite or bi-infinite directed geodesic $\gamma$, the
  \emph{directed Busemann function} associated with $\gamma$ is the
  function $\dB_\gamma: V(\rg) \to \Z \cup \{\pm \infty\}$ defined by
  \begin{equation}\label{eq:directedBusemann}
    \dB_\gamma(v):= \lim_{s\rightarrow+\infty}\pd(v,\gamma(s))-s, \quad\text{for all }v\in V(\rg).
\end{equation}
\end{definition}

Extra assumptions ensure that the directed Busemann function takes
only finite values:

\begin{lemma}
  \label{lem:busefin}
  Given an infinite or bi-infinite directed geodesic $\gamma$ and
  a vertex $v$ of $\rg$:
  \begin{enumerate}
  \item if there exists $s$ such that $\pd(v,\gamma(s))<+\infty$, then $\dB_\gamma(v)<+\infty$,
  \item if there exists $s$ such that $\pd(\gamma(s),v)<+\infty$, then $\dB_\gamma(v)>-\infty$.
  \end{enumerate}
  (Note that the above assumptions are always satisfied if $\rg$ is strongly connected.)
\end{lemma}

\begin{proof}
  The first point has already been discussed above, we only restate it to
  emphasize the symmetry with the second point, which we now establish. Let $s$ be such that $\pd(\gamma(s),v)<+\infty$, then for any
  $t \geq s$, we have
  \begin{equation}
    t-s = \pd(\gamma(s),\gamma(t)) \leq \pd(\gamma(s),v) + \pd(v,\gamma(t)).
  \end{equation}
  Hence $\pd(v,\gamma(t))-t$ is bounded below by
  $-\pd(\gamma(s),v)-s$. Taking $t \to \infty$ we conclude that
  $\dB_\gamma(v)>-\infty$, as wanted.
\end{proof}

We conclude this section by recording some properties that will be useful later:
\begin{lemma}\label{lem:directedBusemann2}
  Let $\gamma$ be an infinite or bi-infinite directed geodesic.
  \begin{enumerate}
  \item For any $u,v \in \rV(\rg)$ such that $(u,v)$ is a directed edge, we have $\dB_\gamma(u)\leq \dB_\gamma(v)+1$, with
    the convention $\pm \infty+1=\pm \infty$.
  \item For any $u \in \rV(\rg)$ such that $\dB_\gamma(u)$ is finite,
    there exists $v  \in \rV(\rg)$ such that $(u,v)$ is a directed edge with
    $\dB_\gamma(u) = \dB_\gamma(v)+1$.
  \item Any directed path $\mathrm p:=(p(0),p(1),\ldots)$ along which
    $\dB_\gamma$ is finite and decreases exactly by $1$ along each
    edge is an infinite geodesic.
  \end{enumerate}
\end{lemma}
\begin{proof}
  \begin{enumerate}
  \item By the oriented triangle inequality we have
    $\pd(u,\gamma(s))\leq 1+\pd(v,\gamma(s))$ for all $s$, and the
    result follows from Definition~\ref{def:Busemann}.
  \item Since the function $s \mapsto \pd(u,\gamma(s))-s$ is
    nonincreasing, takes integer values and tends to the finite limit
    $\dB_\gamma(u)$, it is eventually constant. Hence, there exists $s_0$ such that
    $\dB_\gamma(u)=\pd(u,\gamma(s_0))-s_0$. Consider
    a directed geodesic from $u$ to $\gamma(s_0)$. The first
    edge of this geodesic leads to a vertex $v$ such that
    $\pd(u,\gamma(s_0))=1+\pd(v,\gamma(s_0))$, and hence
    $\dB_\gamma(u)=\pd(u,\gamma(s_0))-s_0=1+\pd(v,\gamma(s_0))-s_0
    \geq 1+\dB_\gamma(v)$. The latter inequality is in fact an
    equality by the previous point.
  \item Let $\mathrm p$ be a path satisfying the assumptions, and fix
    $0\leq s < t$. Let $\sigma$ be large
    enough for both
    \begin{equation}
      \dB_{\gamma}(p(s))=\pd(p(s),\gamma(\sigma))-\sigma \quad \text{and}
      \quad \dB_{\gamma}(p(t))=\pd(p(t),\gamma(\sigma))-\sigma
    \end{equation}
    to hold. Since $\dB_\gamma$ decreases exactly by $1$ along each
    edge of $\mathrm p$, we have
    $t-s=\dB_{\gamma}(p(s))-\dB_{\gamma}(p(t))=\pd(p(s),\gamma(\sigma))-\pd(p(t),\gamma(\sigma))$. From
    the oriented triangle inequality
    $\pd(p(s),\gamma(\sigma)) \leq \pd(p(s),p(t)) +
    \pd(p(t),\gamma(\sigma))$ we conclude that
    $\pd(p(s),p(t)) \geq t-s$. This must be in fact an equality since
    $\mathrm p$ is a directed path. Thus, $\mathrm p$ is a directed
    geodesic as wanted. \qedhere
  \end{enumerate}
\end{proof}

\subsubsection{From trumpets/cornets to slices}

We have now the necessary tools to describe the reverse construction in Proposition~\ref{prop:trumpets}: we consider a trumpet or cornet $\rm$, and explain how to cut it into a slice. Throughout this section, $p$ and $k$ denote the perimeter and girth of $\rm$ respectively.

Let us define the \emph{universal cover} $\tilde \rm$ of $\rm$ as its preimage through the exponential mapping $z\mapsto \exp(2i\pi z)$, assuming that $\rm$ is embedded in the complex plane in such a way that the origin of the plane lies within its unrooted boundary and that its rooted face is the outer face, see Figure~\ref{fig:trumpets}. 

The universal cover $\tilde \rm$ is an infinite planar map\footnote{We
  define an infinite planar map by replacing the word ``finite'' by
  ``infinite'' in the definition of a planar map given in
  Section~\ref{sub:DefIntro}.}, in which every vertex has finite
degree. The translation of the plane $z\mapsto z+1$ induces a natural
automorphism $T$ of $\tilde \rm$, and the exponential mapping
$z\mapsto \exp(2i\pi z)$ induces a \emph{projection} from $\tilde \rm$
to $\rm$. Two elements (vertices, edges, faces or corners) $x_1,x_2$
of $\tilde \rm$ project to the same element of $\rm$ if and only if
there exists $\ell \in \Z$ such that $T^\ell(x_1)=x_2$. Each edge of
$\tilde \rm$ naturally inherits the orientation of its projection in
$\rm$. Each inner face $f$ of $\rm$ is lifted in $\tilde \rm$ to an
infinite number of copies with the same degree and color as $f$.  The
preimage of the unrooted boundary (respectively of the rooted
boundary) is a single face of infinite degree called the \emph{upper
  boundary} (respectively \emph{lower boundary}) of $\tilde \rm$. The
contour of the upper boundary is an infinite directed path, which we
denote by $\gamma=(\gamma(s))_{s \in \Z}$, choosing an arbitrary
parametrization\footnote{As in Section~\ref{sec:dirBuse} we assume
  that $\tilde \rm$ is simple so that a path can be coded as a
  sequence of vertices; one can straightforwardly extend our
  discussion to allow multiple edges, at the price of slightly heavier
  notation.}. Note that $\gamma$ is directed from left to right (resp.\
from right to left) if $\rm$ is a trumpet (resp.\ a cornet), so we
have
\begin{equation}
  \label{eq:gammaT}
  T(\gamma(s)) =
  \begin{cases}
    \gamma(s+k) & \text{if $\rm$ is a trumpet,}\\
    \gamma(s-k) & \text{if $\rm$ is a cornet.}
  \end{cases}
\end{equation}
for all $s \in \Z$. The contour of the lower boundary is also an infinite directed path, going from left to right if $\rm$ is white and from right to left is $\rm$ is black. We root $\tilde\rm$ at an arbitrary preimage $\rho$ of the root
corner of $\rm$: note that $\rho$ and $T(\rho)$ are both incident to the lower boundary, with $p$ edges between them.
For the rest of this section, $\pd$ denotes the directed distance on
$\tilde \rm$. 
\medskip

Let us now exploit the fact that $\rm$ has a tight boundary. We start with the following claim, which is a directed version of the wrapping lemma from~\cite[p.957]{irredmaps}. Since its proof follows exactly the same lines as in the undirected case, we omit it.

\begin{lemma}\label{lemma:wrappingLemma}
If $\rm$ is a trumpet, then for any $v\in \rV(\tilde \rm)$ and $d \in \Z_{>0}$, we have 
\begin{equation}\label{eq:Translation}
  \pd(v,T^d(v)) \geq kd.
\end{equation}
If $\tilde \rm$ is a cornet, then for any $v\in \rV(\tilde \rm)$ and $d \in \Z_{>0}$, we have: 
\begin{equation}\label{eq:strictTranslation}
 \pd(v,T^{-d}(v)) \geq kd,
\end{equation}
where equality holds if and only if $v$ belongs to $\gamma$.
\end{lemma}

This lemma implies almost directly the following one (details of the proof are postponed to the end of this section):
\begin{lemma}\label{lem:geodesicContour}
  If $\rm$ is a trumpet (resp.\ cornet), then $\gamma$ a bi-infinite directed geodesic (resp.\ strict geodesic). 
\end{lemma}

We may thus consider the associated directed Busemann function
$\dB_\gamma$. Note that it depends on the parametrization chosen for
$\gamma=(\gamma(s))_{s \in \Z}$, but the only possible changes of
parametrization are translations $s \mapsto s+s_0$, which just modify
$\dB_\gamma$ by a global additive constant that does not affect the
following construction. By~\eqref{eq:directedBusemann} and
\eqref{eq:gammaT}, it is straightforward to check that, for any vertex
$v$ of $\tilde \rm$, we have
\begin{equation}
  \label{eq:dBT}
  \dB_\gamma(T(v)) =
  \begin{cases}
    \dB_\gamma(v) - k & \text{if $\rm$ is a trumpet,}\\
    \dB_\gamma(v) + k & \text{if $\rm$ is a cornet.}
  \end{cases}
\end{equation}

We claim that $\dB_\gamma$ takes only finite
values. Indeed, let $v$ be a vertex of $\tilde \rm$ and let $v'$ be
its projection in $\rm$. By the strong connectivity of $\rm$ (see
again Remark~\ref{rem:access}) there exists a directed path connecting
$v'$ to the tight boundary, and another in the opposite direction. Such paths lift to
directed paths from $v$ to $\gamma$ and back, and it follows from Lemma~\ref{lem:busefin} that $-\infty<\dB_\gamma(v)<+\infty$. Note that $\tilde \rm$ is not necessarily
strongly connected: for instance if $\rm$ is reduced to a cycle then
$\tilde \rm$ just consists of the infinite directed path $\gamma$. (In
fact, one can show that the universal cover of a hypermap with two monochromatic
boundaries is strongly connected if and only if the two boundaries of
are not adjacent to each other. This fact is closely related with
Lemma~\ref{lem:DobruCyl} below.)

For any corner $c$ of $\tilde \rm$, we define the \emph{leftmost infinite directed geodesic $g_c$ from $c$ towards $\gamma$} as follows. Let $v=:g_c(0)$ be the vertex incident to $c$, and consider the set of vertices $w$ such that $(v,w)$ is a directed edge and $\dB_\gamma(w)=\dB_\gamma(v)-1$. Lemma~\ref{lem:directedBusemann2} ensures that this set is non-empty and set $g_c(1)$ to be its element first encountered when turning clockwise around $v$ starting at $c$. Then, define $(g_c(i))_{i>1}$ recursively as the leftmost infinite directed geodesic towards $\gamma$ from the corner following the directed edge $(g_c(0),g_c(1))$ clockwise around $g_c(1)$.
The fact that $g_c$ is a directed geodesic is an immediate consequence of Lemma~\ref{lem:directedBusemann2}. Moreover, we have: 
\begin{lemma}\label{lemma:coalescence}
For any corner $c$ incident to the lower face of $\tilde \rm$, define 
$\sigma_c:=\min \{s\,,\, g_c(s) \in \gamma\}$. Then, $\sigma_c<\infty$ and, for any $t\geq \sigma_c$, $g_c(t)\in \gamma$.

In other words, for any corner $c$ incident to the lower face of $\tilde \rm$, $g_c$ reaches $\gamma$ and coalesces with it after the first time it touches it.
\end{lemma}

\begin{proof}
This is the only part of the proof that differs substantially between the trumpet and the cornet cases. We start with the case of a cornet.
  
Since $g_c$ is an infinite self-avoiding path in $\tilde \rm$, it visits at least two different pre-images of the same vertex of $\rm$. In other words, there exist $v$ in $\tilde \rm$, $0\leq s<t$ and $d\in \mathbb{Z} \setminus \{0\}$ such that $g_c(s)=v$ and $g_c(t)=T^{-d}(v)$. Using in turn the property that $g_c$ is a directed geodesic, that $\dB_\gamma$ decreases by $1$ along each of its edges, and~\eqref{eq:dBT}, we get the equalities
  \begin{equation}
    \pd(v,T^{-d}(v)) = t-s = \dB_\gamma(v)-\dB_\gamma(T^{-d}(v)) = kd.
  \end{equation}
  We deduce that $d>0$ and can therefore apply Lemma~\ref{lemma:wrappingLemma}, which entails that $v$ is on $\gamma$, and $\sigma_c\leq s<\infty$.

To prove that $g_c$ coalesces with $\gamma$ after $\sigma_c$, we proceed by contradiction. Suppose that there exists $t> \sigma_c$ such that $g_c(t)\notin \gamma$, and take the minimum value of $t$ for which this holds. By considering the leftmost geodesic started at $(g_c(t-1),g_c(t))$, the same reasoning as above gives the existence of $u>t$ such that $g_c(u)\in \gamma$. The subpath $(g_c(t-1),\ldots,g_c(u))$ of $g$ is a geodesic path between two points of $\gamma$ which contains at least one vertex (namely $g_c(t)$) not in $\gamma$. Since $\gamma$ is a bi-infinite strict geodesic, this is a contradiction, which concludes the proof of this first case. 
\medskip

We now consider the case where $\tilde \rm$ is the universal cover of a trumpet. We are going to prove that $\sigma_c= s^\star := \inf\{s\in \mathbb{Z}\,,\,\dB_\gamma(c)=\pd(c,\gamma(s))-s\}$, see Figure~\ref{fig:coalescence}. Consider a directed geodesic $\mathrm p$ from $c$ to $\gamma(s^\star)$. For any $s$, $\dB_\gamma(\gamma(s))=-s$, so that $\dB_\gamma(c)-\dB_\gamma(\gamma(s^\star))=\pd(c,\gamma(s^\star))$. Hence, the value of $\dB_\gamma$ decreases exactly by one along each edge of $\mathrm p$ by Lemma~\ref{lem:directedBusemann2}. 

Next, the definition of $s^\star$ ensures that $\mathrm p$ does not intersect $\gamma$ before $\gamma(s^{\star})$, so that cutting $\tilde \rm$ along the path $\mathrm p$ decomposes $\tilde \rm$ in only two infinite connected components: the left one $\mathcal L$ and the right one $\mathcal R$. By the triangle inequality, for every vertex $v\in \mathcal L\backslash\{ \gamma(s^\star)\}$, we have $\dB_\gamma(v)>\dB_\gamma(\gamma(s^\star))=-s^\star$.

Next, consider the leftmost geodesic $g_c$ started at $c$. Either $g_c$ coincides with $\mathrm p$ up to $\gamma(s^\star)$, and this proves the first part of the claim, or $g_c$ contains at least one vertex from $\mathcal L$. In the latter case, $g_c$ eventually intersects $\mathrm p$ to reach $\mathcal R$. Let $(u,v)$ be the last edge of $g_c$ where $u\in \mathcal{L}$ and $v\in \mathrm p$. Since $g_c$ is a leftmost geodesic, it coalesces with $\mathrm p$ after $v$ and reaches $\gamma(s^\star)$. After $\gamma(s^\star)$, $g_c$ lies in $\mathcal{R}$ and hence follows the contour of the upper boundary face of $\tilde \rm$.
\end{proof}
\begin{figure}[h]
\centering
\includegraphics[scale=1.1]{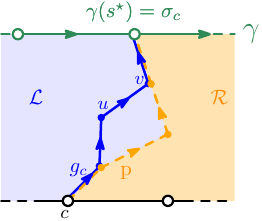}
\caption{Illustration of the proof of Lemma~\ref{lemma:coalescence}. The paths $\mathrm{p}$ and $g_c$ are respectively represented by orange dashed edges and by blue plain edges.\label{fig:coalescence}} 
\end{figure}

We can finally associate to $\rm$ a slice via $\tilde \rm$. We start again with the case where $\rm$ is a cornet. Consider the infinite leftmost geodesics $g_\rho$ and $g_{T(\rho)}$ in $\tilde \rm$. By translational invariance, we have $\sigma_\rho=\sigma_{T(\rho)}+k$, and in particular, $\sigma_\rho>\sigma_{T(\rho)}$. Combined with the preceding claim, this yields that $g_\rho$ and $g_{T(\rho)}$ intersect for the first time at $\sigma_\rho$ and then coalesce. 

By cutting $\tilde \rm$ along both these paths, we decompose it into three connected components: two infinite connected components, and one finite connected component, which we denote by $\mathrm s$. The boundary of $\rs$ consists into three parts: 
\begin{itemize}
\item a \emph{left boundary} which is the initial segment of $g_\rho$ going from $\rho$ to $\gamma(\sigma_\rho)$,
\item a \emph{right boundary} which is the portion of $g_{T(\rho)}$ going from $T(\rho)$ to $\gamma(\sigma_\rho)$: precisely it is the concatenation of the initial segment of $g_{T(\rho)}$ going from $T(\rho)$ to $\gamma(\sigma_{T(\rho)})=T(\gamma(\sigma_\rho))$, and of the portion of the upper boundary $\gamma$ of $\tilde \rm$ from $T(\gamma(\sigma_\rho))$ to $\gamma(\sigma_\rho)$, which has length $k$,
\item and a \emph{base} which is the the portion of the lower boundary
  of $\tilde \rm$ between $\rho$ and $T(\rho)$: its length is equal to the perimeter of $\rm$; and it is directed from $\rho$ to $T(\rho)$ if $\rm$ is white, and from $T(\rho)$ to $\rho$ if $\rm$ is black.
\end{itemize}
We check that $\rs$ is a slice by noting that the left and right boundaries are leftmost geodesics (being portions of infinite leftmost geodesics), which only meet at their common endpoint $\gamma(\sigma_\rho)$. Observe that, if $\rm$ is white (resp.\ black), then $\rs$ is of type $\cA$ (resp.\ $\cB$), and its increment is equal to $k$ (resp.\ $-k$) as the right boundary has $k$ more edges than the left boundary.

The construction of $\rs$ in the case of the universal cover of a trumpet is completely similar and is left to the reader. 
\medskip

It follows from our construction that by performing the wrapping operation on $\rs$, we retrieve $\rm$. Hence, to conclude the proof of Proposition~\ref{prop:trumpets}, it only remains to prove that a slice with non-zero increment can be recovered as the slice decomposition of its gluing. Let $\rs$ be a slice, without loss of generality, assume that $\rs$ has a base directed from $l$ to $r$ and has increment $k>0$, write $L$ for the length of its left boundary, so that its right boundary has length $L+k$.

Consider infinitely many copies $(\rs_i)_{i\in \mathbb{Z}}$ of $\rs$, and write respectively $l_i$ and $r_i$ for the left and right corners of $\rs_i$, for any $i\in \mathbb{Z}$. By gluing together the $L$ first edges of the right boundary of $\rs_i$ to the $L$ edges of the left boundary of $\rs_{i+1}$ for any $i\in \mathbb{Z}$, the infinite map we obtain is clearly equal to the universal cover $\tilde \rm$ of $\rm$ associated to $\rs$ by the wrapping operation. 

We can then check that the leftmost infinite geodesic from $l_0$ in $\tilde \rm$ follows the left boundary of $\rs_0$ and then coalesces with the boundary of the upper face of $\tilde \rm$. This implies that the slicing of $\tilde \rm$ along the leftmost geodesic started at $l_0$ and at $r_0=l_1$ gives back precisely $\rs$, which concludes the proof. 

The three remaining cases can be treated exactly along the same lines. 

\begin{proof}[Proof of Lemma~\ref{lem:geodesicContour}]
Let us first treat the case where $\rm$ is a cornet, and denote by $k$ its girth. Recall that $\gamma=(\gamma(s))_{s\in \mathbb{Z}}$ stands for the bi-infinite path forming the contour of the upper face of $\tilde \rm$.

Let $s\in \mathbb{Z}$ and let $r \in \{0,\ldots,k-1\}$. Let $\mathrm p$ be a directed path from $\gamma(s)$ to $\gamma(s+r)$ in $\tilde \rm$. The projection of $\mathrm p$ to $\rm$ is a directed path connecting two vertices incident to the tight face of $\rm$, hence its length is either equal to $r$, and in this case it follows the contour of its tight boundary, or, strictly greater than $r$. Since the length of $\mathrm p$ is equal to the length of its projection, it implies that the unique geodesic from $\gamma(s)$ to $\gamma(s+r)$ in $\tilde \rm$ is the corresponding subpath of $\gamma$.

Now, let $t>s$, and denote respectively by $q_{s,t}$ and $r_{s,t}$ the quotient and the remainder in the Euclidean division of $t-s$ by $k$. If $q_{s,t}=0$ or $r_{s,t}=0$, the result follows either by the previous paragraph or by the directed wrapping lemma stated in Lemma~\ref{lemma:wrappingLemma}. Otherwise, let $\mathrm p$ be a directed path from $\gamma(s)$ to $\gamma(t)$ in $\tilde \rm$, and let $\mathrm p':=T^{-q_{s,t}}(
\mathrm p)$. Note that $\mathrm p'$ is a directed path from $\gamma(s+q_{s,t}k)$ to $\gamma(t+q_{s,t}k)$, see Figure~\ref{fig:noShortcut}. 
\begin{figure}
\centering
\includegraphics[scale=1.1]{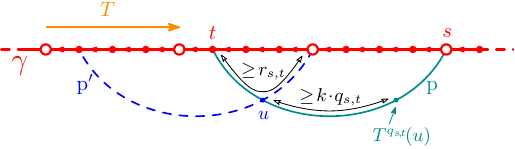}
\caption{Illustration of the proof of Lemma~\ref{lem:geodesicContour}. The paths $\mathrm{p}$ and $\mathrm{p}'$ are respectively represented by blue plain edges and by teal dashed edges.\label{fig:noShortcut}} 
\end{figure}

The two paths $\mathrm{p}$ and $\mathrm{p}'$ necessarily intersect, and we denote by $u$ their first point of intersection, starting from $\gamma(s)$. By translational invariance, $T^{q_{s,t}}(u)$ also belongs to $\mathrm p$, and by Lemma~\ref{lemma:wrappingLemma} the subpath of $\mathrm p$ from $T^{q_{s,t}}(u)$ to $u$ has length at least $q_{s,t}k$ (with equality if and only if the path follows the boundary of the upper face). Moreover, the rest of $\mathrm p$ is made of the subpath from $u$ to $\gamma(t)$ and of the subpath from $\gamma (s)$ to $T^{q_{s,t}}(u)$ (which has the same length as the subpath of $\mathrm p'$ from $\gamma(s+q_{s,t}k)$ to $u$). Since, any path from $\gamma(s+q_{s,t}k)$ to $\gamma(t)$ has length at least $r_{s,t}$ by the preceding paragraph (with equality if and only if the path follows the boundary of the upper face), the path $\mathrm p$ has length at least $t-s$ with equality if and only if its follows the boundary of the upper face, which concludes the proof.

The case of the universal cover of a trumpet can be handled in a similar way, and is in fact slightly simpler, since it suffices to show that it is geodesic rather than strictly geodesic.
\end{proof}

\section{Cylinders and disks via trumpets/cornets}
\label{sec:cyldisk}

In the previous section, we introduced trumpets and cornets as
cylinders with monochromatic boundaries, one of which is tight, and we
enumerated them through a bijection with slices. Building on these results, we now turn to the enumeration of general cylinders with monochromatic boundaries (Subsection~\ref{sub:cyl}). With additional combinatorial trick, we will also derive enumeration formulas for disks with either monochromatic (Subsection~\ref{sub:mono}) or Dobrushin (Subsection~\ref{sub:Dobru}) boundary conditions.

\subsection{Cylinders}\label{sub:cyl}

The key idea of this subsection is that hypermaps with two
monochromatic boundaries can be decomposed into a pair made of a
cornet and a trumpet, by cutting them along a minimal separating cycle.

More precisely, consider a hypermap $\rm$ with two boundaries, which
we draw in the plane with one of the boundaries chosen as the outer
face. This face will be called the \emph{outer boundary}, whereas the
other is called the \emph{central boundary} (we remind that the
adjective ``inner'' refers to faces which are not boundaries).

We define the \emph{ccw-girth} of $\rm$ as the minimal length of a
separating cycle oriented in the counterclockwise direction (recall
Definition~\ref{def:tight} of a separating cycle). If no such cycle
exists, the ccw-girth is conventionally set to infinity. This case may
only occur if the outer boundary is black and the central boundary is
white: indeed, for any other choice of colors for the boundaries, the
contour of at least one boundary forms a counterclockwise separating
cycle.

\begin{proposition}[Trumpet/cornet decomposition of cylinders]\label{prop:cylinders}
For any positive integers $p,q,h$, there exists a weight-preserving $h$-to-$1$ correspondence between:
\begin{itemize}
\item the set of hypermaps with two monochromatic rooted boundaries, the outer one of degree $p$ and the central one of degree $q$, and with ccw-girth $h$, 
\item and the set of ordered pairs $(\rm_1,\rm_2)$, where $\rm_1$ is a trumpet with perimeter $p$ and girth $h$ and $\rm_2$ is a cornet with perimeter $q$ and girth $h$. 
\end{itemize}
Under this correspondence, the color of the outer (respectively central) boundary of the hypermap corresponds to the color of the trumpet $\rm_1$ (respectively the cornet $\rm_2$).
\end{proposition}

As in Section~\ref{sec:genover}, we denote by $F^{\circ\circ}_{p,q}$
(resp.\ $F^{\bullet\bullet}_{p,q}$, $F^{\circ\bullet}_{p,q}$) the
generating function of hypermaps with two monochromatic rooted
boundaries, the outer one white (resp.\ black, white) of degree $p$
and the central one white (resp.\ black, black) of degree $q$. We also
denote by $\tilde{F}^{\bullet\circ}_{p,q}$ the generating function of
hypermaps with two monochromatic rooted boundaries, the outer one
black of degree $p$ and the central one white of degree $q$, \emph{such
  that the ccw-girth is finite}. We colloquially call such objects
``two-way cylinders''.  Combining Corollary~\ref{cor:trumpenum} and
Proposition~\ref{prop:cylinders}, we get the following:

\begin{corollary}\label{cor:cylindersenum}
  For any $p,q \geq 1$, we have
  \begin{equation} \label{eq:cylindersenum}
    \begin{split}
      F^{\circ\circ}_{p,q} &= \sum_{h \geq 1} h \left([z^{h}] x(z)^p \right) \left( [z^{-h}] x(z)^{q} \right),\\
      F^{\bullet\bullet}_{p,q} &= \sum_{h \geq 1} h \left([z^{h}] y(z)^p \right) \left( [z^{-h}] y(z)^{q} \right),\\
      F^{\circ\bullet}_{p,q} &= \sum_{h \geq 1} h \left([z^{h}] x(z)^p \right) \left( [z^{-h}] y(z)^{q} \right),\\
      \tilde{F}^{\bullet\circ}_{p,q} &= \sum_{h \geq 1} h \left([z^{h}] y(z)^p \right) \left( [z^{-h}] x(z)^{q} \right).\\
    \end{split}
  \end{equation}
  In these sums, the term of index $h$ corresponds to the contribution
  of hypermaps with ccw-girth equal to $h$.
\end{corollary}

Note that the above expressions vanish whenever $p$ or $q$ is equal to
$0$, and that the series $F^{\circ\circ}_{p,q}$ and
$F^{\bullet\bullet}_{p,q}$ are symmetric in $p$ and $q$: this follows
from their combinatorial definition by exchanging the roles of the two
boundaries\footnote{Alternatively, the symmetry can be seen by the
  following walk-counting argument. Using the generic notations from
  Appendix~\ref{sec:dsf}, the symmetry in $p$ and $q$ amounts to the
  vanishing of $\sum_{h \in \Z} h P_{p,h} P_{q,-h}$. This quantity is
  the weighted sum, over all walks with $p+q$ steps starting and
  ending at $0$, of the position attained after $p$ steps, i.e.\ of
  the sum of the increments of the $p$ first steps. It must indeed
  vanish, since the sum of all $p+q$ increments of any such walk
  vanishes, and since the weight of a walk is invariant under cyclic
  shifts. Note that we do not use the DSF assumption here.}.

\begin{remark}
  \label{rem:hatFexpr}
  It is interesting and useful to note that the combination
  \begin{equation}
    \label{eq:hatFexpr}
    \hat{F}^{\circ\bullet}_{p,q} := F^{\circ\bullet}_{p,q}-\tilde{F}^{\bullet\circ}_{q,p} =
    \sum_{h \in \Z} h \left([z^h] x(z)^p \right) \left( [z^{-h}] y(z)^q \right)
  \end{equation}
  is the generating function of ``one-way cylinders'', namely
  hypermaps with an outer white boundary of degree $p$ and a central
  black boundary of degree $q$ \emph{such that there exists no
    separating cycle oriented in the clockwise direction}. Indeed,
  those hypermaps having at least one such cycle are counted by
  $\tilde{F}^{\bullet\circ}_{q,p}$, as seen by exchanging the roles of
  the outer and central boundaries.
\end{remark}

\begin{proof}[Proof of Proposition~\ref{prop:cylinders}]
Let us first prove that in a hypermap $\rm$ with two monochromatic boundaries of ccw-girth $h<\infty$, there exists a unique \emph{innermost} ccw cycle of length $h$. To see this, consider any ccw cycle in $\mathcal{M}$ and define its interior to be the submap lying to its left, which therefore contains the central boundary of $\mathcal{M}$. Now consider the intersection of all interiors of ccw cycles of length $h$ in $\mathcal{M}$. Since this intersection contains the central boundary, it is non-empty, and its boundary forms a ccw cycle of length $h$.

\begin{figure}[h!]
\begin{center}
\includegraphics[width=0.3\linewidth]{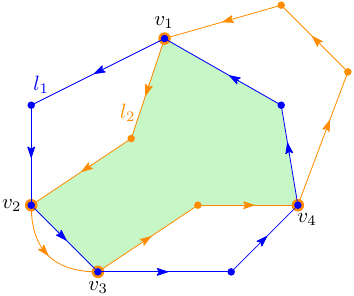}
\caption{Innermost ccw cycle of minimal length.\label{fig:innermost}}
\end{center}
\end{figure}
To prove the latter fact, consider two ccw cycles $l_1$ and $l_2$ of minimal length that intersect each other, and list their common vertices as $v_1,\ldots,v_d$, see Figure~\ref{fig:innermost}. By minimality, the restrictions of $l_1$ and $l_2$ between $v_i$ and $v_{i+1}$ must then have the same length. Hence, the boundary of the intersection of their interiors is also a ccw cycle of minimal length. The general case follows by induction.

The conclusion of the proposition then follows by cutting $\rm$ along the innermost ccw cycle of minimal length. 
\end{proof}

\subsection{Monochromatic disks}\label{sub:mono}

Our purpose in this section is to derive a general expression for the
disk generating functions $F^\circ_p$ and $F^\bullet_p$, defined
in~\eqref{eq:defFp}, in terms of the slice generating functions
$(a_k)_{k \geq -1}$ and $(b_k)_{k \geq 0}$, or equivalently in terms
of the Laurent series $x(z)$ and $y(z)$ defined
in~\eqref{eq:defxy}. We have already found an expression for their
derivatives with respect to the vertex weight~$t$
(Corollary~\ref{cor:pointed}), so in principle an expression for
$F^\circ_p$ and $F^\bullet_p$ follows by integration (the constant
term in $t$ being zero since a map contains at least one vertex). We
would however like to find an integration-free expression, which would
be manifestly algebraic in the case of bounded face degrees.  A first
approach via ``unpointing'' is described in
Appendix~\ref{sec:altmono}: it builds on the ideas
of~\cite[Section~3.3]{hankel} and extends them to the setting of
hypermaps. Here, we present a new combinatorial approach that, instead
of pointed disks, uses rather the various sorts of cylinders which we
have enumerated in the previous subsection. The idea is illustrated on
Figure~\ref{fig:DiskCylTrick} and yields the following:

\begin{proposition}
  \label{prop:DiskCylTrickNew}
 Recall that $\hat{F}_{p,q}$ was defined in Remark~\ref{rem:hatFexpr}. For any $p \geq 0$, we have
  \begin{equation}
    \label{eq:DiskCylTrickNew}
    \begin{split}
      F^\circ_p &= \frac{\hat{F}^{\circ\bullet}_{p+1,1}}{p+1} = \frac{F^{\circ\bullet}_{p+1,1}-\tilde{F}^{\bullet\circ}_{1,p+1}}{p+1} = \frac{F^{\circ\bullet}_{p+1,1} - \sum_{d \geq 2} t_d^\circ F^{\circ\circ}_{p+1,d-1}}{p+1}, \\
      F^\bullet_p &= \frac{\hat{F}^{\circ\bullet}_{1,p+1}}{p+1} = \frac{F^{\circ\bullet}_{1,p+1}-\tilde{F}^{\bullet\circ}_{p+1,1}}{p+1} = \frac{F^{\circ\bullet}_{1,p+1} - \sum_{d \geq 2} t_d^\bullet F^{\bullet\bullet}_{p+1,d-1}}{p+1}.
    \end{split}
  \end{equation}
\end{proposition}

\begin{figure}
  \centering \includegraphics[width=\textwidth]{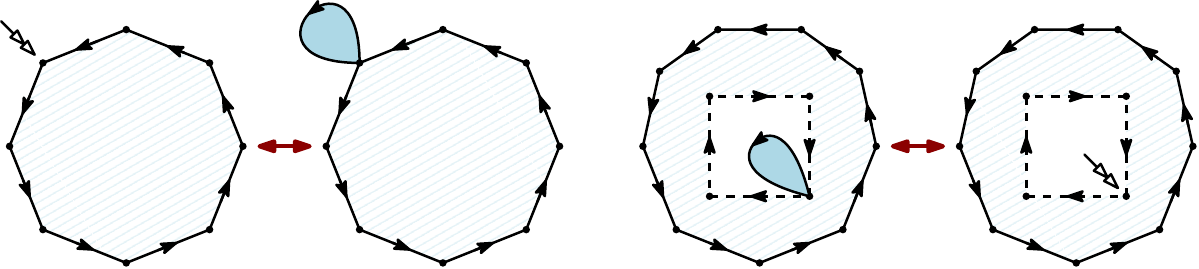}
  \caption{Idea of the proof of
    Proposition~\ref{prop:DiskCylTrickNew}.\\ Left: a disk with a white
    (say) rooted boundary of degree $p$ (here $p=8$, the hatched
    region may contain other vertices, edges and faces) can be
    transformed, by ``inflating'' the marked corner into a black face
    of degree $1$, into a cylinder with a white boundary of degree
    $p+1$ and a black boundary of degree $1$, such that the two
    boundaries are adjacent. Note that this entails that there is no
    clockwise cycle separating the two boundaries.\\ Right: a cylinder
    with a white boundary of degree $p+1$ and a black boundary of
    degree $1$, such that the black boundary is adjacent to an inner
    face of degree $d$ (here $d=5$), has at least one clockwise
    separating cycle (e.g. the one shown with dashed edges). By
    ``collapsing'' the black boundary into a marked corner, we get a
    cylinder with a white boundary of degree $p+1$ and a white rooted
    boundary of degree $d-1$.}
  \label{fig:DiskCylTrick}
\end{figure}

\begin{proof}
  The first equalities on both lines follow from the correspondence illustrated on
  the left of Figure~\ref{fig:DiskCylTrick}: given a hypermap with a
  white rooted boundary of degree $p$ (chosen as the outer face), we
  ``inflate'' the marked corner into a black boundary of
  degree $1$. This also increases the degree of the white boundary by
  $1$. Both boundaries are a priori unrooted: the black boundary may
  be rooted in a unique way, and the white boundary in $p+1$ different
  ways. Note that there is no clockwise separating cycle by
  construction, and we get precisely all ``one-way cylinders''
  contributing to $\hat{F}^{\circ\bullet}_{p+1,1}$. Indeed, as
  illustrated on the right of Figure~\ref{fig:DiskCylTrick}, in a
  cylinder with a white outer boundary and a black central boundary of
  degree $1$, there is a least one clockwise separating cycle as soon
  as the boundaries are not adjacent: take for instance the
  contour of the inner white face adjacent to the black boundary, with
  their common incident edge removed. This establishes the relation
  $(p+1)F^\circ_p = \hat{F}^{\circ\bullet}_{p+1,1}$, and the other
  relation $(p+1)F^\bullet_p = \hat{F}^{\circ\bullet}_{1,p+1}$ is
  obtained similarly by exchanging the colors and orientations.

  The second equalities on both lines of~\eqref{eq:DiskCylTrickNew}
  follow immediately from the first equality
  in~\eqref{eq:hatFexpr}. To conclude the proof of the proposition, it
  suffices to show that
  \begin{equation}
    \tilde{F}^{\bullet\circ}_{1,p+1}=\sum_{d \geq 2} t_d^\circ F^{\circ\circ}_{p+1,d-1}, \qquad
    \tilde{F}^{\bullet\circ}_{p+1,1}= \sum_{d \geq 2} t_d^\bullet F^{\bullet\bullet}_{p+1,d-1}.
  \end{equation}
  This may be seen by the following bijective argument, again
  illustrated on the right of Figure~\ref{fig:DiskCylTrick}. Consider
  a hypermap contributing to $\tilde{F}^{\bullet\circ}_{1,p+1}$, that
  is a hypermap with a black boundary of degree $1$ and a white
  boundary of degree $p+1$ such that the two boundaries are not
  adjacent. Denote by $d$ the degree of the white inner face adjacent
  to the black boundary. We turn this white face into a boundary of
  degree $d-1$ by ``collapsing'' the black boundary into a marked
  corner. Summing over $d$ the first equality follows (the factor
  $t_d^\circ$ is the weight of the white inner face that should be
  accounted in $\tilde{F}^{\bullet\circ}_{1,p+1}$). The second
  equality is obtained similarly.
\end{proof}

Combining~\eqref{eq:DiskCylTrickNew} with the formulas for cylinders
found previously, we may rewrite $F_p^\circ$ and $F_p^\bullet$ in
various forms. For instance, by applying the second equality
in~\eqref{eq:hatFexpr}, we get the compact expressions
\begin{equation}
  \label{eq:Fpcompact}
  F^\circ_p = \frac{1}{p+1}\sum_{h \in \Z} h b_{-h} [z^h]x(z)^{p+1}, \qquad
  F^\bullet_p = \frac{1}{p+1}\sum_{h \in \Z} h a_{-h} [z^{-h}]y(z)^{p+1}
\end{equation}
in terms of the slice generating functions $(a_k),(b_k)$ and the
associated Laurent series $x(z)$ and $y(z)$. Let us make a few remarks
on these expressions:
\begin{enumerate}
\item Only the term $h=1$ gives a positive contribution, and all the
  terms $h<0$ correspond to ``unwanted'' contributions which should be
  subtracted,
\item Taking $p=0$ in both expressions and recalling that
  $F^\circ_0=F^\bullet_0=t$, we recover the identity~\eqref{eq:kakbk}
  from Corollary~\ref{cor:ahbh},
\item In the case of bounded face degrees, it follows
  from~\eqref{eq:Fpcompact} and Lemma~\ref{lem:boundxypol} that
  $F^\circ_p$ and $F^\bullet_p$ are polynomials in the $a_k$'s and
  $b_k$'s.
\end{enumerate}
Alternate expressions are obtained by expanding the right-hand sides
of~\eqref{eq:DiskCylTrickNew} using Corollary~\ref{cor:cylindersenum},
to yield
\begin{equation} \label{eq:Fplesscompact}
  \begin{split}
    F_p^\circ &= \frac{1}{p+1} \left([z]x(z)^{p+1} - \sum_{d \geq 2} t_d^\circ \sum_{h \geq 1} h \left([z^h]x(z)^{p+1}\right) \left([z^{-h}] x(z)^{d-1}\right)\right),\\
    F_p^\bullet &= \frac{1}{p+1} \left(a_{-1} [z^{-1}]y(z)^{p+1} - \sum_{d \geq 2} t_d^\bullet \sum_{h \geq 1} h \left([z^{-h}]y(z)^{p+1}\right) \left([z^{h}] y(z)^{d-1}\right)\right).\\
  \end{split}
\end{equation}
We will use these less compact expressions in
Section~\ref{sec:apmono}.  Note that, by symmetry, there are two ways
to expand $F^{\circ\circ}_{p+1,d-1}$ and
$F^{\bullet\bullet}_{p+1,d-1}$, and we choose one specifically
here. Note that we may pass directly from~\eqref{eq:Fpcompact}
to~\eqref{eq:Fplesscompact} using Proposition~\ref{prop:decslices}.

To conclude, let us mention that this approach gives a new derivation of the disk generating function of non-bicolored maps. Indeed, recall from Remark~\ref{rem:nonbicolored} that, in that case, we have $x(z)=z^{-1}+S+Rz$. By plugging this expression in the first line of~\eqref{eq:Fplesscompact}, we get an expression which corresponds to~\cite[(3.16)]{hankel} in the limit $d\rightarrow \infty$.

\subsection{Disks with a Dobrushin boundary condition}\label{sub:Dobru}

We now proceed to establishing the following:

\begin{proposition}
  \label{prop:Wdobr}
  For $p,q \geq 0$, let $F^{\yy}_{p,q}$ be the generating function of
  hypermaps with a Dobrushin boundary of type $(p,q)$, as defined in
  Section~\ref{sub:DefIntro}. Then, we have
  \begin{equation}
    \label{eq:Wdobr}
    \begin{split}
      \sum_{p,q \geq 0}  \frac{F^{\yy}_{p,q}}{x^{p+1} y^{q+1}} 
      &=\exp \left(\sum_{h \in \Z} h \left([z^h] \ln \left(1-\frac{x(z)}x\right) \right) \left( [z^{-h}] \ln \left(1-\frac{y(z)}y\right) \right)\right)-1 \\
      &=  \exp\left(  \sum_{h \in \Z} \sum_{p,q \geq 1}
      h \left([z^h] \frac{x(z)^p}{p x^p} \right) \left( [z^{-h}] \frac{y(z)^q}{q y^q} \right) 
      \right) - 1.
    \end{split}
  \end{equation}
\end{proposition}

Note that, by extracting the coefficient of $\frac1{x^{p+1}y}$ or
$\frac1{x y^{p+1}}$ in these expressions, we recover the
formulas~\eqref{eq:Fpcompact} for $F^\circ_{p}=F^{\yy}_{p,0}$ and
$F^\bullet_{p}=F^{\yy}_{0,p}$.  Our starting point for establishing
the proposition is Remark~\ref{rem:hatFexpr} about ``one-way
cylinders'', which we already used in the previous subsection to
establish Proposition~\ref{prop:DiskCylTrickNew}.
One of the observations made in its proof may be generalized as
follows:

\begin{lemma}[Characterization of one-way cylinders]
  \label{lem:DobruCyl}
  A hypermap with an outer white boundary and a central black boundary
  contains no separating cycle oriented in the clockwise direction if
  and only if the two boundaries are adjacent, i.e.\ they have a
  common incident edge.
\end{lemma}

\begin{figure}[t]
  \centering
  \includegraphics[scale=.9]{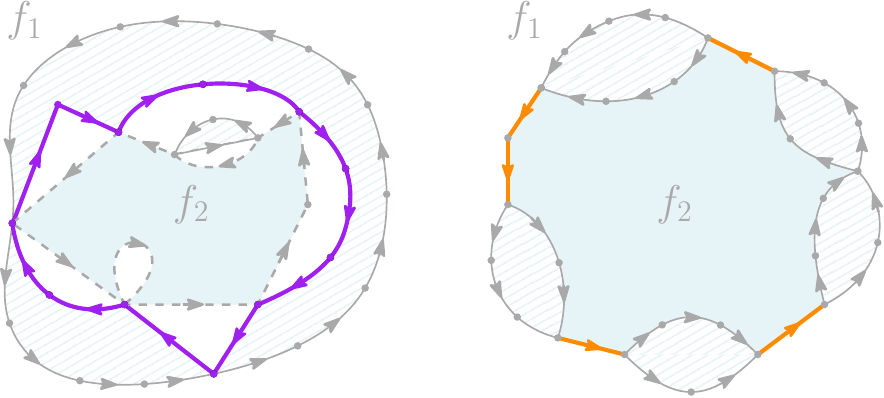}
  \caption{A criterion for the existence of a clockwise separating
    cycle in a hypermap with a white outer boundary $f_1$ and a black
    central boundary $f_2$ (the hatched regions may contain other
    vertices, edges and faces).\\ %
    Left: when $f_1$ and $f_2$ do not have a common incident edge, we
    may construct a clockwise separating cycle, here shown with purple
    thicker lines, by considering the outer contour of the region
    formed by merging $f_2$ with its adjacent white faces (thereby
    removing the dashed edges).\\  Right: when $f_1$ and $f_2$ have
    common incident edges, here shown as orange thicker lines, these
    prevent the existence of a clockwise separating cycle.  Between
    these edges, the hypermap consists of a sequence of ``blobs''
    (hatched regions).
    By removing one orange edge, thereby merging
    $f_1$ and $f_2$, we obtain a hypermap with a Dobrushin boundary
    condition.
  }
  \label{fig:DobruTrick}
\end{figure}

\begin{proof}
  Let $\rm$ be a hypermap, drawn in the plane, with an outer white
  boundary denoted by $f_1$, and a central black boundary denoted
  by $f_2$. If $f_1$ and $f_2$ have a common incident edge $e$, then
  every separating cycle is oriented in the counterclockwise
  direction, as it must necessarily pass through $e$, which is
  canonically oriented with $f_1$ on the right and $f_2$ on the
  left. See the right of Figure~\ref{fig:DobruTrick} for an
  illustration.

  Conversely, let us assume that $f_1$ and $f_2$ have no common
  incident edge. This situation is illustrated on the left of
  Figure~\ref{fig:DobruTrick}.  Let us consider the open region
  $\mathrm{R}$ of the plane obtained by merging $f_2$ with its
  adjacent white faces (thereby removing all the edges incident to
  $f_2$). Note that $\mathrm{R}$ may not be simply connected, if a white face is adjacent to $f_2$ along several edges. By our
  assumption $\mathrm{R}$ is bounded, and we let $\mathrm{R}'$ be the
  unbounded component of the complement of $\mathrm{R}$. The boundary
  of $\mathrm{R}'$, which we may call the ``outer contour'' of
  $\mathrm{R}$, is then a clockwise separating cycle of $\rm$, since
  by construction it consists of edges of $\rm$ incident to only white faces on the inside.
\end{proof}

\begin{figure}
  \centering
  \includegraphics[scale=.8]{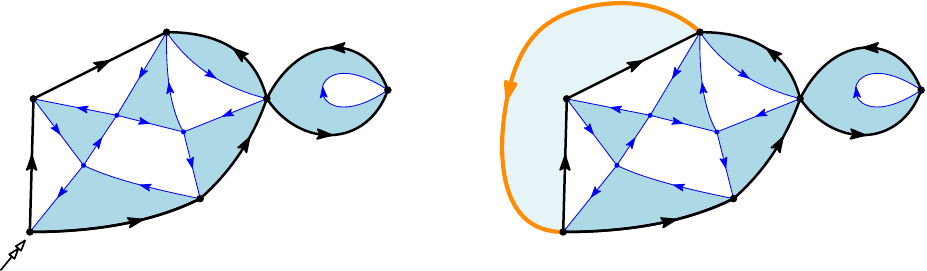}
  \caption{The hypermap with a Dobrushin boundary of
    Figure~\ref{subfig:hypermapDobrushin} (here reproduced on the
    left) can be transformed into a ``one-way cylinder'' (right) by
    adding an edge connecting the two corners at which the orientation
    of the boundary reverses, thereby splitting the non-monochromatic
    boundary into two monochromatic boundaries, one white (outer face)
    and one black (lighter blue).}
  \label{fig:DobruCyl}
\end{figure}

We may now make the connection with hypermaps with Dobrushin
boundaries.

\begin{proof}[Proof of Proposition~\ref{prop:Wdobr}]
Recall that the contour of a Dobrushin boundary consists of
two parts with opposite orientation. As shown on
Figure~\ref{fig:DobruCyl}, a hypermap with such a boundary can
be naturally transformed into a ``one-way
cylinder'', by adding a directed edge between the two vertices at which the orientation of the boundary changes. Lemma~\ref{lem:DobruCyl} shows that this transformation is
essentially a bijection, since starting conversely from a one-way
cylinder, by removing an edge incident to both boundaries we obtain an
hypermap with a Dobrushin boundary, see again the right of
Figure~\ref{fig:DobruTrick}.

To obtain a genuine bijection we have to pay attention to the
details. For any $p,q \geq 0$, the above transformation gives a
weight-preserving bijection between the set of hypermaps with a
Dobrushin boundary of type $(p,q)$, and that of hypermaps with two
monochromatic unrooted boundaries, one white of degree $p+1$ and one
black of degree $q+1$, \emph{with a marked edge incident to both
  boundaries}. Indeed, this marked edge contains all the information
we need to recover the hypermap with a Dobrushin boundary (recall
that a hypermap with a Dobrushin boundary is conventionally rooted at
the origin of the two boundary parts, thus the position of the
root corner does not contain any extra information, except in the
cases $p=0$ or $q=0$ where it allows to locate the boundary part
of zero length).

\begin{figure}
  \centering
  \includegraphics{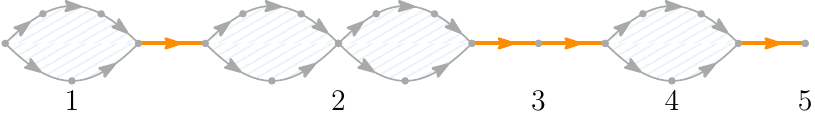}
  \caption{Schematic representation of the blob decomposition of a
    hypermap with a Dobrushin boundary (the hatched regions may
    contain extra vertices, edges and faces): by removing the four
    bridges (shown in orange), we obtain five ``blobs'', possibly with
    non simple boundaries or reduced to single vertices.}
  \label{fig:DobruBlob}
\end{figure}

Hence, the generating function $F^{\yy}_{p,q}$ of hypermaps with a
Dobrushin boundary of type $(p,q)$ is not quite equal to
$\hat{F}^{\circ\bullet}_{p+1,q+1}$, since they differ by the markings
(in the latter, both boundaries are rooted). To obtain an exact
relation, we have to go through an extra step which we call the \emph{blob
decomposition}. Observe that, unlike hypermaps with monochromatic
boundaries, hypermaps with Dobrushin boundaries may contain bridges,
i.e. edges whose removal disconnects the map. The bridges are
necessarily incident to the non-monochromatic boundary on both
sides. Let us then define a \emph{blob} as a hypermap with a Dobrushin
boundary containing no bridge. As illustrated on
Figure~\ref{fig:DobruBlob}, in general a hypermap with a Dobrushin
boundary can be decomposed into a non empty sequence of blobs,
connected by bridges. Note that this amounts to a decomposition into
strongly connected components, in the graph-theoretical
sense. Denoting by $B_{p,q}$ the generating function of blobs of type
$(p,q)$, we form the grand generating function
\begin{equation}
  \label{eq:Bxydef}
  B(x,y) := \sum_{p,q \geq 0} \frac{B_{p,q}}{x^{p+1} y^{q+1}}.
\end{equation}
Note that we have $B_{0,0}=t$, which corresponds to the blob reduced to a single vertex. 

The blob decomposition then yields the relation
\begin{equation}
  \label{eq:WBxyrel}
  \sum_{p,q \geq 0}  \frac{F^{\yy}_{p,q}}{x^{p+1} y^{q+1}} =
  \sum_{k \geq 1} B(x,y)^k = \frac{1}{1-B(x,y)} - 1.
\end{equation}
Here, $B(x,y)^k$ is the generating function of the hypermaps
consisting of $k$ blobs, and we sum over $k \geq 1$ to obtain all
possible configurations. The exponents $p+1$ and $q+1$ appearing in
the above series ensure that the contribution of the bridges to the
lengths of the two directed paths forming the Dobrushin boundary are
properly taken into account.

A similar blob decomposition can be performed for the ``one-way
cylinders'' contributing to $\hat{F}^{\circ\bullet}_{p+1,q+1}$. These
hypermaps have two marked corners, one incident to each boundary. The
blob decomposition consists in removing \emph{all} the edges incident
to both boundaries: as visible on the right of
Figure~\ref{fig:DobruTrick} this also yields a sequence of blobs,
which we conventionally start at the blob carrying the marked corner
incident to the white boundary.

This yields the relation
\begin{equation}
  \label{eq:hatFBrel}
  \sum_{p,q \geq 1} \frac{\hat{F}^{\circ\bullet}_{p,q}}{x^{p} y^{q}} =
  y \frac{\partial}{\partial y} \left(
    \sum_{k \geq 1} x \frac{\partial B(x,y)}{\partial x} B(x,y)^{k-1} \right) =
  x y \frac{\partial^2}{\partial x \partial y} \ln \left( \frac{1}{1 - B(x,y)} \right)
\end{equation}
where $- x \frac{\partial B(x,y)}{\partial x}$ is the generating
function of blobs with a corner marked on the ``white'' side, and
where acting with $-y \frac{\partial}{\partial y}$ corresponds to
marking another corner on the ``black'' side, anywhere in the sequence
of blobs. By comparing~\eqref{eq:WBxyrel} and~\eqref{eq:hatFBrel} we
find that
\begin{equation}
  \sum_{p,q \geq 0}  \frac{F^{\yy}_{p,q}}{x^{p+1} y^{q+1}} = \exp\left( \sum_{p,q \geq 1}
    \frac{\hat{F}^{\circ\bullet}_{p,q}}{p x^{p} q y^{q}} \right) - 1
\end{equation}
and using the expression~\eqref{eq:hatFexpr} for
$\hat{F}^{\circ\bullet}_{p,q}$ we obtain the expression given on the
second line of~\eqref{eq:Wdobr}.
\end{proof}

\section{All-perimeter generating functions}
\label{sec:allperim}

In this section, we consider the grand generating functions
$\Wb(x):= \sum_{p\geq 0}\frac{F_p^\circ}{x^{p+1}}$ and
$\Wn(y) := \sum_{p\geq 0}\frac{F_p^\bullet}{y^{p+1}}$ and its
derivatives introduced in Section~\ref{sec:genover}. Our purpose is to
establish the formulas~\eqref{eq:Wt}, \eqref{eq:Wcyl} and
\eqref{eq:Wsubs}. We will also discuss the relation between the
expression given in Proposition~\ref{prop:Wdobr} for the generating
function of hypermaps with a Dobrushin boundary, and those given
in~\cite{EynardBook}.

\subsection{Preliminaries, pointed disks and cylinders}
\label{sec:apprelim}

In Sections~\ref{sec:wrapping} and~\ref{sec:cyldisk}, we have found
expressions for the series $F_p^\circ,F_p^\bullet$ and its derivatives
$\frac{\partial}{\partial t}F_p^\circ,\frac{\partial}{\partial
  t}F_p^\bullet,\allowbreak F_{p,q}^{\circ\circ}=q\frac{\partial}{\partial
  t_q^\circ}F_p^\circ,\ldots$ in terms of the Laurent series
$x(z)=\sum_{k \geq -1} a_kz^{-k}$ and
$y(z)=z^{-1} + \sum_{k\geq 0} b_kz^{k}$ (whose coefficients are the
generating functions of elementary slices of type $\cA_k$ and $\cB_k$
respectively). Precisely, our expressions involved quantities of the
form $[z^h] x(z)^p$ and $[z^h] y(z)^p$, for some $h \in \Z$, and we
would therefore like to compute series of the form
$[z^h] \sum_{p \geq 0} \frac{x(z)^p}{x^{p+1}}$ and
$[z^h] \sum_{p \geq 0} \frac{y(z)^p}{y^{p+1}}$.  In view of
Remark~\ref{rem:DSFxy}, this corresponds to computing series of
downward skip-free walks weighted by the $a_k$ and $b_k$
respectively---recall that these series belong to the ring
$\Ring:=\mathbb{Q}[[t,
t^\circ_1,t^\circ_2,\ldots,t^\bullet_1,t_2^\bullet,\ldots]]$.  We
discuss the enumeration of DSF walks in detail in
Appendix~\ref{sec:dsf}, let us only quote the useful results here.

\label{page:excursionz}
Let us define an~\emph{excursion} as a DSF walk of any length (see
again Definition~\ref{def:DSF}) that starts at $0$, ends at $-1$, and
always remains nonnegative in between (these objects are also called
\L{}ukasiewicz walks). We then define $z^\bullet(y)$ as the generating
functions of excursions where, to each step with increment $i$, we
attach a weight $b_i y^{-1}$, for all $i \geq -1$ (recall that
$b_{-1}=1$). The series $z^\bullet(y)$ belongs to $\Ring[[y^{-1}]]$,
and by a standard decomposition of excursions (see
Proposition~\ref{prop:Ueq}) we have
\begin{equation}\label{eq:defGSExcursionsBullet}
  z^\bullet(y) = \sum_{k\geq -1} \frac{b_k}{y} z^\bullet(y)^{k+1} \qquad \text{i.e.} \qquad
  y = y(z^\bullet(y)).
\end{equation}

\begin{remark} \label{rem:compinv}
  One may think of $z^\bullet(y)$ as a ``compositional inverse'' of
  $y(z)$. To relate this idea to the standard notion of compositional
  inverse of formal power series---see
  e.g.~\cite[Section~5.4]{Stanley1999}---we observe that, while
  $y(z)=z^{-1} + \sum_{k\geq 0} b_kz^{k}$ is a formal Laurent series
  in $z$, its multiplicative inverse $y(z)^{-1}=z+\cdots$ is a formal
  power (not Laurent) series without constant coefficient and with a
  nonzero $z$ coefficient. As such, it indeed admits a compositional
  inverse in the usual sense, and this is why $z^\bullet(y)$ is
  naturally a series in $y^{-1}$.
\end{remark}

We similarly define $\tilde{z}^\circ(x)$ as the generating function
of excursions where, to each step with increment $i$, we attach a
weight $a_i x^{-1}$, for all $i \geq -1$. The series
$\tilde{z}^\circ(x)$ belongs to $\Ring[[x^{-1}]]$ and satisfies
\begin{equation}
  \label{eq:tzcircdef}
  \tilde{z}^\circ(x)  = \sum_{k\geq -1} \frac{a_k}{x} \tilde{z}^\circ(x)^{k+1} \qquad \text{i.e.} \qquad
  x = x\left(\tilde{z}^\circ(x)^{-1}\right).
\end{equation}
We then set $z^\circ(x):=\tilde{z}^\circ(x)^{-1}$, which may be thought
as the compositional inverse of $x(z)$. It is a formal Laurent series
in $x^{-1}$, note however that its coefficients do not belong to
$\Ring$ but to its field of fractions, as
$\tilde{z}^\circ(x) = a_{-1} x^{-1} + \cdots$ and,
by~\eqref{eq:decAslices}, we see that $a_{-1}$ has no constant
coefficient so is not invertible in $\Ring$. Even though we will write
expressions in terms of $z^\circ(x)$ for symmetry reasons, in fact the
reader should notice that they can be recast in terms of
$\tilde{z}^\circ(x)$ and that, at the end, we only obtain series
having their coefficients in $\Ring$.

The ``master'' series $z^\bullet(y)$ and $z^\circ(x)$ allow to express
several generating functions of weighted DSF walks. This is discussed
in detail in Appendix~\ref{sec:dsf}, where we establish the following
proposition, stating the results that we need for our purposes.

\begin{proposition}
  \label{prop:DSFenumxy}
  For DSF walks starting and ending at $0$ (``bridges''), we have
  \begin{equation} \label{eq:xybridges}
    \begin{split}
      \sum_{p \geq 1} \frac{P^\bullet_{p,0}}{p y^p} &= [z^0] \sum_{p \geq 1} \frac{y(z)^p}{p y^p} = \ln \left( y \, z^\bullet(y) \right), \\
      \sum_{p \geq 1} \frac{P^\circ_{p,0}}{p x^p} &= [z^0] \sum_{p \geq 1} \frac{x(z)^p}{p x^p}
                                                    = \ln \left( \frac{x}{a_{-1} z^\circ(x)} \right).
    \end{split}
  \end{equation}
  For DSF walks starting at $0$ and ending at a fixed negative
  position $-h<0$, we have
  \begin{equation} \label{eq:xycycle}
    \begin{split}
      \sum_{p \geq 1} \frac{P^\bullet_{p,-h}}{p y^p} &= [z^{-h}] \sum_{p \geq 1} \frac{y(z)^p}{p y^p} = \frac{z^\bullet(y)^h}h, \\
      \sum_{p \geq 1} \frac{P^\circ_{p,-h}}{p x^p} &= [z^h] \sum_{p \geq 1} \frac{x(z)^p}{p x^p} = \frac{z^\circ(x)^{-h}}{h}.
    \end{split}
  \end{equation}
  For DSF walks starting at $0$ and ending at a positive position, we
  have the bivariate series
  \begin{equation} \label{eq:xybivar}
    \begin{split}
      \sum_{p \geq 1} \sum_{h \geq 1} \frac{P^\bullet_{p,h} u^h}{p y^p} &= \sum_{h \geq 1} u^h [z^{h}] \sum_{p \geq 1} \frac{y(z)^p}{p y^p} 
      = \ln \left( \frac{z^\bullet(y)^{-1}-u^{-1}}{y - y(u)}\right), \\
      \sum_{p \geq 1} \sum_{h \geq 1} \frac{P^\circ_{p,h} u^h}{p x^p} &= \sum_{h \geq 1} u^h [z^{-h}] \sum_{p \geq 1} \frac{x(z)^p}{p x^p} = \ln \left( a_{-1} \frac{z^\circ(x)-u^{-1}}{x-x(u^{-1})}\right).
    \end{split}
  \end{equation}
\end{proposition}

Note that
the right-hand sides of~\eqref{eq:xybridges} belong to
$\Ring[[y^{-1}]]$ and $\Ring[[x^{-1}]]$ respectively, since
$y z^\bullet(y)$ and $\frac{x}{a_{-1} z^\circ(x)}$ are formal power
series in $y^{-1}$ and $x^{-1}$ respectively with constant coefficient
equal to $1$ (they count excursions with their last step removed, so
$1$ is the contribution from the walk with $0$ steps).  We now apply
the above formulas to treat the case of pointed disks and cylinders:

\begin{proposition} \label{prop:ggfpointcyl}
  The grand generating functions of pointed disks with a monochromatic
  boundary read
  \begin{equation}
    \label{eq:Wtbis}
    \frac{\partial W^\circ(x)}{\partial t} = \frac{d}{dx} \ln z^\circ(x), \qquad
    \frac{\partial W^\bullet(y)}{\partial t}  = - \frac{d}{dy} \ln z^\bullet(y)
  \end{equation}
  while those for cylinders with monochromatic boundaries read
  \begin{equation}
  \label{eq:Wcylbis}
  \begin{split}
    W^{\circ\circ}(x_1,x_2) &= \sum_{p,q \geq 1} \frac{F^{\circ\circ}_{p,q}}{x_1^{p+1} x_2^{q+1}}
    = \frac{\partial^2}{\partial x_1 \partial x_2} \ln \left(\frac{z^\circ(x_1)-z^\circ(x_2)}{x_1-x_2}\right),\\
    W^{\bullet\bullet}(y_1,y_2) &= \sum_{p,q \geq 1} \frac{F^{\bullet\bullet}_{p,q}}{y_1^{p+1} y_2^{q+1}} = \frac{\partial^2}{\partial y_1 \partial y_2} \ln \left(\frac{z^\bullet(y_1)-z^\bullet(y_2)}{y_1-y_2}\right),\\
    W^{\circ\bullet}(x,y) &= \sum_{p,q \geq 1} \frac{F^{\circ\bullet}_{p,q}}{x^{p+1} y^{q+1}} = -
    \frac{\partial^2}{\partial x \partial y} \ln \left( 1 - \frac{z^\bullet(y)}{z^\circ(x)}\right).
\end{split}
\end{equation}
\end{proposition}

\begin{proof}
  We start with the case of pointed disks: we have
  $\frac{\partial W^\circ(x)}{\partial t}=\sum_{p \geq 0}
  \frac{\partial F_p^\circ}{\partial t} x^{-p-1} = \sum_{p \geq 0}
  \frac{P^\circ_{p,0}}{x^{p+1}}$ by Corollary~\ref{cor:pointed}. The
  latter series is obtained by acting on the second line
  of~\eqref{eq:xybridges} with $-\frac{\partial}{\partial x}$, and
  by adding the contribution of the $p=0$ term
  $\frac{P^\circ_{0,0}}x=\frac1x$. The case of
  $\frac{\partial W^\bullet(x)}{\partial t}$ is similar.
  
  We now turn to the case of cylinders, for which we shall sum the
  expressions given in Corollary~\ref{cor:cylindersenum}. For
  $W^{\circ\circ}(x_1,x_2)$ we may write
  \begin{multline}
    W^{\circ\circ}(x_1,x_2) = \frac{\partial^2}{\partial x_1 \partial x_2} \sum_{p,q,h \geq 1} h \cdot \frac{[z^h]x(z)^p}{p x_1^p} \cdot \frac{[z^{-h}]x(z)^q}{q x_2^q} \\
    \overset{\eqref{eq:xycycle}}{=} \frac{\partial^2}{\partial x_1 \partial x_2} \sum_{q,h\geq 1} z^\circ(x_1)^{-h} \cdot \frac{[z^{-h}]x(z)^q}{q x_2^q} \overset{\eqref{eq:xybivar}}{=}
    \frac{\partial^2}{\partial x_1 \partial x_2} \ln \left( a_{-1} \frac{z^\circ(x_2)-z^\circ(x_1)}{x_2-x(z^\circ(x_1))}\right)
  \end{multline}
  which yields the first line of~\eqref{eq:Wcylbis} using
  $x(z^\circ(x_1))=x_1$ (the $a_{-1}$ factor can be removed due to the
  differentiation).  The case of $W^{\bullet\bullet}(y_1,y_2)$ is
  similar. Finally, for $W^{\circ\bullet}(x,y)$ we may write
  \begin{equation}
    W^{\circ\bullet}(x,y) = \frac{\partial^2}{\partial x \partial y} \sum_{p,q,h \geq 1} h \cdot \frac{[z^h]x(z)^p}{p x^p} \cdot \frac{[z^{-h}]y(z)^q}{q y^q} \overset{\eqref{eq:xycycle}}{=} \frac{\partial^2}{\partial x \partial y} \sum_{h \geq 1} \frac{1}h \left(\frac{z^\bullet(y)}{z^\circ(x)}\right)^h
  \end{equation}
  which yields the last line of~\eqref{eq:Wcylbis}.
\end{proof}

Note that it does not seem possible to compute by the above method the
grand generating function associated with
$\tilde{F}^{\bullet\circ}_{p,q}$. We will consider another method,
under the assumption of bounded face degrees, in
Section~\ref{sec:apdobru}.

\subsection{Monochromatic disks and the spectral curve}
\label{sec:apmono}

We now treat the case of monochromatic disks in the following:

\begin{proposition}
  \label{prop:Wsubsbis}
  The grand generating functions
  $\Wb(x)$ and
  $\Wn(y)$ of disks with
  monochromatic boundaries are given by
  \begin{equation}
    \label{eq:Wsubsbis}
    W^\circ(x) = y(z^\circ(x)) - \sum_{d \geq 1} t^\circ_d x^{d-1}, \qquad
    W^\bullet(y) = x(z^\bullet(y)) - \sum_{d \geq 1} t^\bullet_d y^{d-1}.
  \end{equation}
  Under the assumption of bounded face degrees
  ($t^\circ_d=t^\bullet_d=0$ for $d$ large enough), these identities
  hold within $\Ring((x^{-1}))$ and $\Ring((y^{-1}))$
  respectively. Without this assumption, they hold within the
  respective larger rings
  $\mathbb{Q}((x^{-1}))[[t,
  t^\circ_1,t^\circ_2,\ldots,t^\bullet_1,t_2^\bullet,\ldots]]$ and
  $\mathbb{Q}((y^{-1}))[[t,\allowbreak
  t^\circ_1,t^\circ_2,\ldots,t^\bullet_1,t_2^\bullet,\ldots]]$, in
  which the substitutions $y(z^\circ(x))$ and $x(z^\bullet(y))$ are
  respectively well-defined.
\end{proposition}

\begin{proof}
  Let us prove the first identity, the argument for the second being
  entirely similar. For simplicity, let us first present the idea of
  the computation, without paying attention to the ring in which we
  are working.  On the one hand, we may write
  \begin{equation}
    \label{eq:yzcircexppr}
    y(z^\circ(x)) = z^\circ(x)^{-1} + \sum_{k \geq 0} b_k z^\circ(x)^{k}
    \overset{\eqref{eq:decBslices}}{=} z^\circ(x)^{-1} + 
    \sum_{d \geq 1} t_d^\circ \sum_{k \geq 0}  z^\circ(x)^{k} [z^{k}] x(z)^{d-1}.
  \end{equation}
  On the other hand, by~\eqref{eq:Fplesscompact}, we have
  \begin{equation}
    W^\circ(x) = \sum_{p\geq 0}\frac{F_p^\circ}{x^{p+1}} =
    \sum_{p\geq 0} \frac{[z]x(z)^{p+1}}{(p+1)x^{p+1}} -
    \sum_{p\geq 0} \sum_{d \geq 2} t_d^\circ \sum_{h \geq 1} h \frac{[z^h]x(z)^{p+1}}{(p+1)x^{p+1}} [z^{-h}] x(z)^{d-1}.
  \end{equation}
  By~\eqref{eq:xycycle}, we deduce
  \begin{equation}
    \label{eq:wcircexppr}
    W^\circ(x) = z^\circ(x)^{-1} - \sum_{d \geq 2} t_d^\circ \sum_{h \geq 1} z^\circ(x)^{-h} [z^{-h}] x(z)^{d-1}.
  \end{equation}
  By doing a change of variable $k=-h$ in the right-hand side, the sum
  can be combined with that in the right-hand side
  of~\eqref{eq:yzcircexppr} to yield
  \begin{equation}
    y(z^\circ(x)) - W^\circ(x) =  \sum_{d \geq 1} t_d^\circ \sum_{k \in \Z} z^\circ(x)^{k} [z^{k}] x(z)^{d-1} = \sum_{d \geq 1} t_d^\circ \,x \left(z^\circ(x)\right)^{d-1} =
    \sum_{d \geq 1} t_d^\circ x^{d-1},
  \end{equation}
  which gives the wanted identity.

  Let us now discuss in which ring the above computations make sense.
  Introducing the shorthand notations
  \begin{equation}
    \begin{split}
      \Ring_1 &:= \Ring((x^{-1})) = \mathbb{Q}[[t,
                t^\circ_1,t^\circ_2,\ldots,t^\bullet_1,t_2^\bullet,\ldots]]((x^{-1})) \\
      \Ring_2 &:= \mathbb{Q}((x^{-1}))[[t,
    t^\circ_1,t^\circ_2,\ldots,t^\bullet_1,t_2^\bullet,\ldots]]
    \end{split}
  \end{equation}
  observe that we have the strict inclusion
  \begin{equation}
    \Ring_1 \subsetneq \Ring_2.
  \end{equation}
  The series $\Wb(x)$ belongs to $\Ring_1$ hence to $\Ring_2$, however
  we have
  $\sum_{d \geq 1} t^\circ_d x^{d-1} \in \Ring_2 \setminus \Ring_1$
  unless we make the assumption of bounded degrees.  Let us check
  that, in full generality, the substitution $y(z^\circ(x))$ makes
  sense within $\Ring_2$: for this we return to~\eqref{eq:yzcircexppr}
  and note that, in the right-hand side, the quantity
  $[z^{k}] x(z)^{d-1}$ vanishes if $k \geq d$, and is of the form
  $(a_{-1})^k C_{d,k}$ for some $C_{d,k} \in \Ring$ if
  $0 \leq k \leq d-1$ (as $x(z)=a_{-1} z + \cdots$).  Thus, we have
  \begin{equation}
    \label{eq:yzbis}
    y(z^\circ(x)) = z^\circ(x)^{-1} +
    \sum_{d \geq 1} t_d^\circ \sum_{k=0}^{d-1} C_{d,k} \left(a_{-1} z^\circ(x)\right)^k .
  \end{equation}
  Now, we observe that $a_{-1} z^\circ(x)$ belongs to $\Ring_1$, since
  it is the multiplicative inverse of
  $\tilde{z}^\circ(x)/a_{-1}=x^{-1}+ \cdots \in
  \Ring[[x^{-1}]]$, so that $a_{-1}z^\circ(x)=x(1+S(x))$, where $S(x)\in \Ring[[x^{-1}]]$. Hence,
  $\sum_{k=0}^{d-1} C_{d,k} \left(a_{-1} z^\circ(x)\right)^k $ also
  belongs to $\Ring_1$ for any $d \geq 1$. Under the assumption of
  bounded face degrees, this shows that $y(z^\circ(x))$ is a
  well-defined element of $\Ring_1$ as wanted. Without this
  assumption, consider an arbitrary (finite) monomial $T$ in
  $t, t^\circ_1,t^\circ_2,\ldots,t^\bullet_1,t_2^\bullet,\ldots$, then
  for $d$ large enough $t_d^\circ$ does not appear in $T$, and hence
  only finitely many terms in the right-hand side of~\eqref{eq:yzbis}
  contribute to the coefficient of $T$. This means that
  $y(z^\circ(x))$ is a well-defined element of $\Ring_2$, as wanted.
\end{proof}

\begin{remark}
  For bookkeeping purposes, let us write down the black analogue of~\eqref{eq:wcircexppr}:
  \begin{equation}
    \label{eq:wbulletexppr}
    W^\bullet(x) = a_{-1} z^\bullet(x) - \sum_{d \geq 2} t_d^\bullet \sum_{h \geq 1} z^\bullet(x)^{h} [z^{h}] y(z)^{d-1}.
  \end{equation}
\end{remark}

\begin{remark}
  Our discussion of the rings in which the
  identities~\eqref{eq:Wsubsbis} hold is somewhat inspired
  from~\cite[Section~2]{Gessel1980}. This reference considers the ring
  $\C((T))[[Y]]$ of formal power series in a variable $Y$ whose
  coefficients are formal Laurent series in another variable $T$ (we
  capitalize the letters to avoid confusions with the notations of the
  present paper). Here, we are replacing $T$ by $x^{-1}$ or $y^{-1}$,
  and $Y$ by a collection of variables $t,t_1^\circ,...$, which does
  not change fundamentally the discussion.
\end{remark}

For simplicity, let us now make the assumption of bounded face
degrees, in which case Lemma~\ref{lem:boundxypol} asserts that $x(z)$
and $y(z)$ are Laurent polynomials. They define the so-called
\emph{spectral curve}, see~\cite[Definition~8.3.1]{EynardBook}. Then,
interestingly, Proposition~\ref{prop:Wsubsbis} may be reformulated in
the following form, which corresponds essentially
to~\cite[Theorem~8.3.1]{EynardBook}\footnote{In this reference, the
  result is given under the inessential assumption that there are no
  faces of degrees $1$ and $2$. Also, slightly different notations are
  used, see again Remark~\ref{rem:Eynard_connection}.}.

\begin{proposition}[Rational parametrization of disk generating
  functions]
  \label{prop:YXsubs}
  Consider the series
  \begin{equation}
    Y(x) := W^\circ(x) + \sum_{d=1}^{\dmaxb} t^\circ_d x^{d-1}, \qquad
    X(y) := W^\bullet(y) + \sum_{d=1}^{\dmaxn} t^\bullet_d y^{d-1}
  \end{equation}
  where $\dmaxb,\dmaxn$ are the maximal degrees for white and black faces, respectively.
  Then, we have the rational parametrization
  \begin{equation}
    \label{eq:YXsubs}
    Y(x(z))=y(z), \qquad X(y(z))=x(z).
  \end{equation}
\end{proposition}

\begin{proof}
  Proposition~\ref{prop:Wsubsbis} asserts that $Y(x)=y(z^\circ(x))$
  and $X(y)=x(z^\bullet(y))$. As discussed in
  Section~\ref{sec:apprelim}, $z^\circ(x)$ and $z^\bullet(y)$ can be
  seen as compositional inverses of $x(z)$ and $y(z)$
  respectively. Therefore, we get the wanted
  relations~\eqref{eq:YXsubs} by substituting $x=x(z)$ into $Y(x)$ and
  $y=y(z)$ into $X(y)$. Note that these operations are well-defined
  since $Y(x)$ is here a formal Laurent series in $x^{-1}$, which we
  substitute by $x(z)^{-1}$ which is a formal power series in $z^{-1}$
  without constant coefficient, and similarly for the other relation.
\end{proof}

\subsection{Disks with a Dobrushin boundary condition and the resultant}
\label{sec:apdobru}

In this section we make the assumption of bounded face degrees
($t^\circ_d=t^\bullet_d=0$ for $d$ large enough) so that, by
Lemma~\ref{lem:boundxypol}, $x(z)$ and $y(z)$ are Laurent polynomials.
Our purpose is to show that the expression obtained in
Proposition~\ref{prop:Wdobr} for the generating function of hypermaps
with a Dobrushin boundary is consistent with the expressions given
in~\cite{EynardBook}. We first state the following complement to
Proposition~\ref{prop:DSFenumxy} on the enumeration of weighted DSF walks.
It is also established in Appendix~\ref{sec:dsf}.

\begin{proposition}
  \label{prop:DSFenumxybd}
  Let $\dmaxb$ be a fixed positive integer and set $t_k^\circ=0$ for
  $k>\dmaxb$. Then, over the field of Puiseux series in $y^{-1}$, the
  algebraic equation $y(z)=y$ admits $\dmaxb$ solutions. One of them
  is the generating function of excursions $z^\bullet(y)$ defined above. Denoting the other ones by
  $z_1^\bullet(y),\ldots,z_{\dmaxb-1}^\bullet(y)$, we have for any
  $h>0$
  \begin{equation}
    \sum_{p \geq 1} \frac{P^\bullet_{p,h}}{p y^p} = [z^{h}] \sum_{p \geq 1} \frac{y(z)^p}{p y^p} = \sum_{i=1}^{\dmaxb-1} \frac{z_i^\bullet(y)^{-h}}h.
  \end{equation}
  Similarly, let $\dmaxn$ be a fixed positive integer and set
  $t_k^\bullet=0$ for $k>\dmaxn$. Then, over the field of Puiseux series
  in $x^{-1}$, the algebraic equation $x(z)=x$ admits $\dmaxn$
  solutions. One of them is the inverse $z^\circ(x)$ of the generating
  function of excursions defined above. Denoting the other ones by
  $z_1^\circ(y),\ldots,z_{\dmaxn-1}^\circ(y)$, we have for any $h>0$
  \begin{equation}
    \sum_{p \geq 1} \frac{P^\circ_{p,h}}{p x^p} = [z^{-h}] \sum_{p \geq 1} \frac{x(z)^p}{p x^p} = \sum_{i=1}^{\dmaxn-1} \frac{z_i^\circ(y)^{h}}h.
  \end{equation}
\end{proposition}

Note that these expressions correspond to the coefficients of $u^h$ in
the two series~\eqref{eq:xybivar}. From them we get an interesting
expression for the grand generating function of disks with a Dobrushin
boundary condition.

\begin{proposition}
  \label{prop:Wdobrbd}
  Under the assumption of bounded face degrees, we have
  \begin{equation}
    1 + \sum_{p,q \geq 0}  \frac{F^{\yy}_{p,q}}{x^{p+1} y^{q+1}} = \left( 1 - \frac{z^\bullet(y)}{z^\circ(x)}\right)^{-1} \prod_{i=1}^{\dmaxn-1} \prod_{j=1}^{\dmaxb-1} \left( 1 - \frac{z_i^\circ(x)}{z_j^\bullet(y)}\right).
  \end{equation}
\end{proposition}

\begin{proof}
  Rewrite the left-hand side using the second line
  of~\eqref{eq:Wdobr}.  Evaluate the sums over $p,q$ using
  \eqref{eq:xycycle} for $h>0$, and Proposition~\ref{prop:DSFenumxybd}
  for $h<0$. The sum over $h>0$ is then equal to
  $-\ln\left( 1 - \frac{z^\bullet(y)}{z^\circ(x)}\right)$, and the sum
  over $h<0$ to
  $\sum_{i=1}^{\dmaxn-1} \sum_{j=1}^{\dmaxb-1} \ln \left( 1 -
    \frac{z_i^\circ(x)}{z_j^\bullet(y)}\right)$. Taking the
  exponential, the wanted expression follows.
\end{proof}

In the remainder of this section, we will argue that this expression
is equivalent to that given in~\cite{EynardBook}. Let us consider the
resultant $r(x,y)$ of the two following polynomials in $z$:
\begin{equation}
  z^{\dmaxn-1}(x(z)-x) = \sum_{k=-1}^{\dmaxn-1} (a_k - x \delta_{k,0}) z^{\dmaxn-1-k}, \qquad
  z (y(z)-y) = \sum_{k=-1}^{\dmaxb-1} (b_k - y \delta_{k,0}) z^{k+1}.
\end{equation}

On the one hand, the resultant is by definition the determinant
\begin{equation}
  r(x,y) := \det
  \begin{pmatrix}
    a_{-1} & a_0 -x & a_1 & \cdots & a_{\dmaxn-1} & & &\\
    & a_{-1} & a_0 -x & a_1 & \cdots & a_{\dmaxn-1} & &\\
    & & \ddots & \ddots & \ddots & \ddots & \ddots  &\\
    & & & a_{-1} & a_0 -x & a_1 & \cdots & a_{\dmaxn-1} \\
    b_{\dmaxb -1 } & \cdots & b_1 & b_0 - y & 1 & & &\\
    & b_{\dmaxb -1 } & \cdots & b_1 & b_0 - y & 1 & & \\
    & & \ddots & \ddots & \ddots & \ddots & \ddots  &\\
    & & & b_{\dmaxb -1 } & \cdots & b_1 & b_0 - y & 1 \\
  \end{pmatrix}
\end{equation}
where the undisplayed entries are equal to $0$.  This quantity is
proportional to $E(x,y)$ as given in~\cite[Theorem~8.3.2,
p.~378]{EynardBook}, up to the change of notation given
in~\eqref{eq:Eynard_connection}.

On the other hand, the resultant can be expressed in terms of the
roots of the polynomials. Namely, using the factorizations
\begin{equation}
  \label{eq:xyfactorpol}
  z^{\dmaxn-1}(x(z)-x)= a_{-1} \prod_{i=0}^{\dmaxn-1} (z-z_i^\circ(x)), \qquad
  z (y(z)-z) = b_{\dmaxb-1} \prod_{j=0}^{\dmaxb-1} (z-z_j^\bullet(y)),
\end{equation}
where we denote by $z_0^\circ(x)=z^\circ(x)$ and
$z_0^\bullet(y)=z^\bullet(y)$ the ``zeroth'' roots, we find
\begin{equation}
  \begin{split}
    r(x,y) &= (a_{-1})^{\dmaxb} (b_{\dmaxb-1})^{\dmaxn} \prod_{i=0}^{\dmaxn-1} \prod_{j=0}^{\dmaxb-1} \left( z_i^\circ(x) - z_j^\bullet(y) \right) \\
    &= (a_{-1})^{\dmaxb} \prod_{i=0}^{\dmaxn-1} \prod_{j=0}^{\dmaxb-1} \left( 1 - \frac{z_i^\circ(x)}{z_j^\bullet(y)}\right)
  \end{split}
\end{equation}
where we use the fact that
$b_{\dmaxb-1} \prod_{j=0}^{\dmaxb-1} (-z_j^\bullet(y))=b_{-1}=1$ since
it corresponds to the constant coefficient of the polynomial
$z(y(z)-z)$. The above double product is very close to that in
Proposition~\ref{prop:Wdobrbd}, but it involves the extra zeroth
roots. To make the two expressions match, we shall divide the above
display by the two relations~\eqref{eq:xyfactorpol} taken at
$z=z_0^\bullet(y)$ and $z=z_0^\circ(x)$ respectively. This yields
\begin{equation}
  1 + \sum_{p,q \geq 0}  \frac{F^{\yy}_{p,q}}{x^{p+1} y^{q+1}} = \frac{r(x,y)}{(a_{-1})^{\dmaxb-1}(x-X(y))(y-Y(x))}
\end{equation}
where $X(y)=x(z_0^\bullet(y))$ and $Y(x)=y(z_0^\circ(y))$ are as in
Proposition~\ref{prop:YXsubs}. By substituting $x=x(z)$ and $y=y(z')$
we recover, up to notations, the expression given at~\cite[Equation
(8.4.6), p.~397]{EynardBook} for $c=-1$.

\section{Conclusion}
\label{sec:conc}
In this paper, we have shown how to enumerate bijectively a number of families of planar hypermaps using the slice decomposition. We thus recovered several of the intriguing formulas given in~\cite[Chapter 8]{EynardBook}. 
In particular, we have seen that the fact that generating functions for hypermaps admit rational parametrizations is closely related with the enumeration of downward skip-free walks: the parameter in which the series are rational can be interpreted as the generating function of excursions. 

Another novelty of our approach is the introduction of trumpets and cornets, that is hypermaps with two boundaries, one of them tight. Such families of hypermaps have not been considered previously, but turn out to be important building blocks to which other families of hypermaps can be reduced. Trumpets and cornets also appear as natural objects in the blossoming tree approach of the forthcoming paper~\cite{AMT2025}.

In contrast with the case of ordinary, non-bicolored maps, the slice decomposition for hypermaps relies heavily on the canonical orientation of the edges, and leads to considerations specific to the world of directed graphs. In particular, our approach to enumerate planar hypermaps with Dobrushin boundaries hinges on a decomposition into strongly connected components. This idea can be generalized to handle arbitrary boundary conditions; this is the subject of two papers in preparation~\cite{BEL,Lejeune}.

\appendix

\section{Generating functions of downward skip-free walks}\label{appendix:DSF}
\label{sec:dsf}

The purpose of this appendix is to review some enumeration results
about downward skip-free walks, which we need in this
paper. Precisely, we shall establish Propositions~\ref{prop:DSFenumxy}
and~\ref{prop:DSFenumxybd}. Since the discussion is more about walks
than hypermaps, we state the results using independent notation, and
it is only at the end that we make the connection with the conventions
and notation of the main text. Let us mention that the enumeration of
one-dimensional walks, sometimes also called directed lattice paths,
is a very classical topic, see for instance \cite{BaFl2002,BaLaWa2020}
and references therein\footnote{An important reference is the
  classical textbook by Feller \cite{Feller1971} which contains many
  deep results about one-dimensional walks, notably their relation
  with Wiener-Hopf factorization, of relevance here. It however takes
  a bit of work to perform the translation from the language of
  probability theory to that of enumerative combinatorics.}. Note that
our terminology differs slightly from these references. An originality
of our discussion is that it specifically focuses on downward
skip-free walks.

Recall from Definition~\ref{def:DSF} that a \emph{downward skip-free
  walk}, or \emph{DSF walk} for short, is a finite sequence of
integers $(\pi_0,\pi_1,\ldots,\pi_\ell)$ of arbitrary length, such
that at each \emph{step} $i=1,\ldots,\ell$, the \emph{increment}
$\pi_i-\pi_{i-1}$ is greater than or equal to $-1$. The integers
$\pi_0,\pi_1,\ldots,\pi_\ell$ are called the \emph{positions} of the
walk, with $\pi_0$ the starting position and $\pi_\ell$ the final
position. Given formal variables $p_{-1},p_0,p_1,p_2,\ldots$, we
attach a weight $p_j$ to each step with increment $j$, for each
$j=-1,0,1,2,\ldots$. We then define the weight of a DSF walk as the
product of its step weights.

A first simple observation is that, introducing the formal Laurent
series
\begin{equation}
  \label{eq:Pdef}
  P(u) := p_{-1} u^{-1} + p_0 + p_1 u + p_2 u^2 + \cdots,
\end{equation}
the quantity $P(u)^\ell$ may be expanded as a sum over all DSF walks
with $\ell$ steps starting at position $\pi_0=0$, the exponent
of $u$ recording the final position $\pi_\ell$. Thus, the coefficient
\begin{equation}
  P_{\ell,h} := [u^h] P(u)^\ell
\end{equation}
is the sum of weights of all DSF walks with $\ell$ steps, starting at
$0$ and ending at $h$. Clearly, $P_{\ell,h}$ is an element of the ring
$\mathcal{P}:=\mathbb{Q}[p_{-1},p_0,p_1,p_2,\ldots]$ of polynomials in
the $p_j$ with rational coefficients.

Next, we consider \emph{excursions}, namely DSF walks which start at
$0$, end at $-1$, and always remain nonnegative in between (these
objects are also called \L{}ukasiewicz walks). We denote by $U(s)$
the generating function of excursions with an arbitrary number of steps, where we incorporate an extra auxiliary
weight $s$ per step. It is an element of the ring $\mathcal{P}[[s]]$
of formal power series in $s$ with coefficients in $\mathcal{P}$. We
then have:
\begin{proposition}
  \label{prop:Ueq}
  The generating function $U(s)$ of excursions is the unique series
  in $\mathcal{P}[[s]]$ satisfying
  \begin{equation}
    \label{eq:Ueq}
    1 = s P(U(s))
  \end{equation}
  or equivalently
  \begin{equation}
    \label{eq:Ueq2}
    U(s) = s \left( p_{-1} + p_0 U(s) + p_1 U(s)^2 + p_2 U(s)^3 + \cdots \right).
  \end{equation}
\end{proposition}
This follows from a classical ``last-passage'' decomposition: given an
excursion, denote by $j$ the increment of its first step. Then,
consider the last passages of the walk by the positions
$j-1,j-2,\ldots,1,0,-1$ (such passages necessarily exist since the
walk is downward skip-free): this yields a decomposition of the
excursion into its first step followed by the concatenation of $j+1$
(suitably shifted) excursions. Thus, the generating function of
excursions starting with a step of increment $j$ is equal to
$s p_j U(s)^{j+1}$, and we obtain~\eqref{eq:Ueq2} by summing over $j$.

A similar last-passage decomposition implies that, for any $k \geq 1$,
the series $U(s)^k$ is equal to the generating function of DSF walks
which start at $0$, end at $-k$, and remain strictly above $-k$
before. The classical \emph{cycle lemma}---see
e.g.~\cite[Chapter~5]{Stanley1999} or \cite[$\triangleright$
I.47]{Flajolet2009} and references therein---asserts that
\begin{equation}
  [s^\ell] U(s)^k = \frac{k}{\ell} P_{\ell,-k}, \qquad k,\ell \geq 1.
\end{equation}
In other words, out of all DSF walks with $\ell$ steps starting at $0$
and ending at $-k$, a fraction $k/\ell$ remain strictly above $-k$
until the last step. The cycle lemma may be rewritten as
\begin{equation}
  \label{eq:Pexc}
  \sum_{\ell \geq 0} P_{\ell,-k} s^\ell = s U(s)^{k-1} U'(s), \qquad k \geq 1.
\end{equation}
This expression remains in fact valid for $k=0$. Indeed, observe that a
general DSF walk from $0$ to $-1$ can be uniquely decomposed as
the concatenation of an excursion and of a general walk from $-1$ to
$-1$, which by shifting the latter gives the relation
$s U'(s) = U(s) \sum_{\ell \geq 0}
P_{\ell,0} s^\ell$.

We would like to find a counterpart of~\eqref{eq:Pexc} for DSF walks
starting at $0$ and ending at a positive position. To this end, let us
define, for all $h \geq 0$, an \emph{arch} of \emph{tilt} $h$ as a DSF
walk which starts at $0$, ends at $h$, and remains at positions
$\geq h$ in between. An arch of tilt $h$ is said \emph{strict} if it
has at least one step and remains at positions $>h$ in between its
starting and ending positions. Let us denote by $A_h^\geq(s)$ and
$A_h^>(s)$ the generating functions of arches and strict arches of
tilt $h$, respectively. Let us make the observation that
\begin{equation}
  \label{eq:arch0}
  A_0^\geq(s) = \frac{1}{1-A_0^>(s)}
\end{equation}
since an arch of tilt $0$ is the concatenation of an arbitrary number
of strict arches of tilt $0$. Furthermore, we have
\begin{equation}
  \label{eq:archstrict}
  A_h^\geq(s)=A_h^>(s) A_0^\geq(s) =\frac{A_h^>(s)}{1-A_0^>(s)} \qquad \text{for } h>0
\end{equation}
since an arch of tilt $h$ is the concatenation of a strict arch of the
same tilt, and of a (suitably shifted) arch of tilt $0$. Finally, we
have
\begin{equation}
  \label{eq:archstrdec}
  A_h^>(s) = s \sum_{j \geq h} p_j U(s)^{j-h} \qquad \text{for } h\geq 0
\end{equation}
as seen by decomposing a strict arch according to the increment $j$ of
its first step. All these relations allow to express the $A_h^\geq(s)$
and $A_h^>(s)$ in terms of the ``master'' series $U(s)$. But notice
that we also have
\begin{equation}
  \label{eq:archexcur}
  U(s) = s p_{-1} A_0^\geq(s) = \frac{s p_{-1}}{1-A_0^>(s)}
\end{equation}
since an excursion is obtained by appending a final downward step to
an arch of tilt $0$. Combining with~\eqref{eq:archstrdec} at $h=0$, we
recover the relation~\eqref{eq:Ueq2} determining $U(s)$. All the above
combinatorial relations may be conveniently encoded into the
following:

\begin{proposition}[formal Wiener-Hopf factorization]
  \label{prop:WHfact}
  Within the ring $\mathcal{P}((u))[[s]]$ of formal power series in
  $s$ whose coefficients are themselves formal Laurent series in $u$
  with coefficients in $\mathcal{P}$, we have
  \begin{equation}
    \label{eq:WHfact}
    \begin{split}
      1 - s P(u) &= \left( 1 - \frac{U(s)}u \right)
                   \left( 1 - \sum_{h \geq 0} A_h^>(s) u^h \right) \\
                 &= \frac{sp_{-1}}{U(s)} \left( 1 - \frac{U(s)}u \right) \left( 1 - \sum_{h > 0} A_h^\geq(s) u^h \right) .
    \end{split}
  \end{equation}
  or equivalently
  \begin{equation}
    \label{eq:WHfactbis}
    \begin{split}
      \frac1{1 - s P(u)} &= \frac1{1 - \frac{U(s)}u} \cdot
                           \frac1{1 - \sum_{h \geq 0} A_h^>(s) u^h} \\
      &= \frac{U(s)}{sp_{-1}} \cdot \frac1{1 - \frac{U(s)}u} \cdot
      \frac1{1 - \sum_{h > 0} A_h^\geq(s) u^h}.
    \end{split}
  \end{equation}
\end{proposition}

\begin{proof}
  We establish the first equality in~\eqref{eq:WHfact} by checking
  that the coefficients of $u^h$ on both sides are equal for any
  $h \in \Z$:
  \begin{itemize}
  \item for $h<-1$ they simply vanish,
  \item for $h=-1$ we find the identity
    $-s p_{-1} = - U(s) (1-A^>_0(s))$ equivalent
    to~\eqref{eq:archexcur},
  \item for $h\geq 0$ we find the identity
    $-s p_h=-A_h^>(s)+U(s) A_{h+1}^>(s)$ resulting
    from~\eqref{eq:archstrdec}.
  \end{itemize}
  We then pass to the second line of~\eqref{eq:WHfact}
  using~\eqref{eq:archstrict} and~\eqref{eq:archexcur} again. We then
  obtain~\eqref{eq:WHfactbis} by inverting within
  $\mathcal{P}((u))[[s]]$.
\end{proof}

\begin{remark}
  There is also a \emph{bijective} proof of \eqref{eq:WHfactbis},
  which consists in decomposing a general DSF walk into a sequence of
  excursions followed by a sequence of arches (which may be taken
  strict or not, hence the two versions of the identity). Let us not
  enter into details here, but refer instead to the closely related
  discussion in~\cite{BaLaWa2020}.
\end{remark}

For our purposes it is convenient to take the \emph{logarithm}
of~\eqref{eq:WHfactbis}, which gives the relation
\begin{equation}
  \sum_{\ell \geq 1} \frac{s^\ell P(u)^\ell}\ell = \ln \frac{U(s)}{s p_{-1}} + \sum_{k \geq 1} \frac{U(s)^k}{k u^k}
  + \ln \frac1{1 - \sum_{h > 0} A_h^\geq(s) u^h}.
  % \sum_{j \geq 1} \left( \sum_{h \geq 0} A_h(s) u^h \right)^j \frac{u^j}j.
\end{equation}
Notice that, in the right-hand side, the first term does not contain
the variable $u$, while the second and third terms contain only
negative and positive powers of $u$ respectively. Extracting the
nonpositive powers of $u$ in the relation, we get the identities
\begin{equation}
  \label{eq:logseries}
  \sum_{\ell \geq 1} \frac{P_{\ell,0} s^\ell}\ell = \ln \frac{U(s)}{s p_{-1}}, \qquad \qquad
  \sum_{\ell \geq 1} \frac{P_{\ell,-h} s^\ell}\ell = \frac{U(s)^h}h \qquad \text{for } h>0
\end{equation}
which amount again to the cycle lemma. For positive powers of $u$,
there is a priori no such simple identities, but we may write
\begin{equation}
  \label{eq:archgen}
  \sum_{h \geq 1} \sum_{\ell \geq 1}  \frac{P_{\ell,h} s^\ell u^h}{\ell} =
  \ln \frac1{1 - \sum_{h > 0} A_h^\geq(s) u^h} = \ln \left( p_{-1} \frac{U(s)^{-1}-u^{-1}}{1-sP(u)} \right).
\end{equation}
It is possibly to get a bit further by imposing a global bound on
increments, that is by fixing a positive integer $d$ and setting
$p_j=0$ for $j>d$. In that case, $P(u)=p_{-1} u^{-1}+\cdots+p_d u^d$
is a Laurent \emph{polynomial} in $u$, and the equation $1=sP(u)$
admits $d+1$ roots in a suitable completion of $\mathcal{P}[[s]]$. One
of these roots is the generating function of excursions $U(s)$. The
other roots may be constructed as follows: consider the formal power
series in $V$
\begin{equation}
  \sigma(V) := V \left( 1+ \frac{p_{d-1}}{p_d} V + \cdots + \frac{p_{-1}}{p_d} V^{d+1} \right)^{-1/d}
\end{equation}
where we expand $(1+\cdots)^{-1/d}$ in the standard way.  The constant
coefficient vanishes and the coefficient of $V$ is $1$, hence
$\sigma(V)$ admits a compositional inverse which we denote by
$V(\sigma)$. Then, denoting by $\omega_1,\ldots,\omega_d$ the
$d$-th roots of unity, for each $i=1,\ldots,d$, we set 
$U_i(s):=V(\omega_i (p_d s)^{1/d})^{-1}$ to be the Puiseux series in $s$
obtained by substituting $\sigma=\omega_i (p_d s)^{1/d}$ in
$V(\sigma)$, and taking the multiplicative inverse. It is
straightforward to check that it satisfies $1=sP(U_i(s))$. Note that
the leading term in $U_i(s)$ is $\omega_i^{-1} (p_d s)^{-1/d}$,
involving a negative exponent of $s$. For this reason,
$U_1(s),\ldots,U_d(s)$ are called the \emph{large roots} of the
equation $1=sP(u)$, while $U_0(s):=U(s)$ is called the \emph{small
  root} as it only contains positive powers of $s$. With all these
roots at hand, we may split the Laurent polynomial $1-sP(u)$ as
\begin{equation}
  \label{eq:Pfactnew}
  1-sP(u) = - \frac{p_d s}{u} \prod_{i=0}^d (u-U_i(s)) =
  \frac{p_{-1} s}{U(s)} \left( 1 - \frac{U(s)}{u} \right)
  \prod_{i=1}^d \left( 1 - \frac{u}{U_i(s)} \right).
\end{equation}
Here, the second equality is obtained by noting that the product of
all roots (large and small) is equal to $(-1)^{d+1}p_{-1}/p_d$.
Comparing with~\eqref{eq:WHfact}, we deduce that
\begin{equation}
  1 - \sum_{h > 0} A_h^\geq(s) u^h = \prod_{i=1}^d \left( 1 - \frac{u}{U_i(s)} \right)
\end{equation}
and, by~\eqref{eq:archgen}, we get
\begin{equation}
  \label{eq:logseriesbis}
  \sum_{\ell \geq 1} \frac{P_{\ell,h} s^\ell}\ell = \sum_{i=1}^d \frac{U_i(s)^{-h}}h \qquad \text{for }h>0.
\end{equation}
This is the last identity that we need for our purposes:
\begin{proof}[Proof of Propositions~\ref{prop:DSFenumxy} and~\ref{prop:DSFenumxybd}]
  The identities involving $P^\bullet_{p,\ldots}$ are obtained by
  substituting $p_j=b_j$, $u=y^{-1}$, $U(s)=z^\bullet(y)$,
  $d=\dmaxb-1$ and $U_i(s)=z_i^\bullet(y)$ in~\eqref{eq:logseries},
  \eqref{eq:archgen} and \eqref{eq:logseriesbis} (recall that
  $b_{-1}=1$). Similarly, for the identities involving
  $P^\circ_{p,\ldots}$ we substitute $p_j=a_j$, $u=x^{-1}$,
  $U(s)=\tilde{z}^\circ(y)=z^\circ(x)^{-1}$, $d=\dmaxn-1$,
  $U_i(s)=z_i^\circ(x)^{-1}$.
\end{proof}

Let us conclude this appendix by mentioning that,
in~\cite{BaFl2002,BaLaWa2020}, the equation $1-sP(u)$ is called the
\emph{kernel equation}, and that its roots play a fundamental
role. These references consider one-dimensional walks which are not
necessarily downward-skip free, but still have a global bound on their
increments, in the sense that in addition to the upper bound $d$ on
positive increments there exists another positive integer $c$ such
that $p_j=0$ for $j<-c$. Downward skip-free walks correspond to the
particular case $c=1$. In general, the kernel equation admits $c$
small roots and $d$ large roots. The take home message 
from~\cite{BaFl2002,BaLaWa2020} is that the generating functions of
several families of constrained walks (the so-called bridges,
excursions, meanders...) can be expressed as symmetric polynomials of
the small roots, or of the inverses of the large roots, as we may
observe here on a specific instance in~\eqref{eq:logseriesbis}.

\section{An alternative approach to the enumeration of monochromatic
  disks}
\label{sec:altmono}

In this appendix we derive an expression for the disk generating
functions $F_p^\circ$ and $F_p^\bullet$ by building on some ideas
introduced in~\cite[Section~3.3]{hankel} and extending them to the
setting of hypermaps. Let us mention that the intermediate case of
so-called constellations was treated in~\cite{NousFPSAC}. This
provides an alternative proof to the approach of Section~\ref{sub:mono}.

Following the above references, we make a detour via a subfamily of
pointed disks. Namely, we consider pointed disks with a monochromatic
boundary condition, and the extra constraint that, among all vertices
incident to the boundary face, the directed distance to the pointed
vertex attains a minimum at the vertex incident to the root
corner. By Definition~\ref{def:boundaryLabelCondition}, this
says that the label boundary condition has a minimum at the root
vertex and, following the terminology of Appendix~\ref{sec:dsf}, that
it forms an arch of tilt $0$.

Let us denote by $F_p^{\blar}$ and
$F_p^{\nlar}$ the generating function of such hypermaps with
respectively a white and a black boundary of degree
$p$. For any $h \geq 0$, let us furthermore denote by $P_{p,-h}^{\blar}$ (resp.\
  $P_{p,-h}^{\nlar}$) the generating function of DSF walks with
  $p$ steps starting at $0$ and ending at $-h$, that not visit positions
  smaller than $-h$, and carry a weight $a_k$ (resp.\ $b_k$) per
  step of increment $k$. We first take $h=0$ in the following lemma, which is
  an immediate consequence of Proposition~\ref{prop:decPointedRooted}.

\begin{lemma}\label{lem:bijExcursion}
 For any $p \geq 0$, we have
  \begin{equation}
    F_p^{\blar} = t P_{p,0}^{\blar} \qquad \text{and} \qquad F_p^{\blar} = t P_{p,0}^{\blar},
  \end{equation}
  where the factor $t$ accounts for the pointed vertex.
\end{lemma}

Now, let us observe that unpointed disks can be been as pointed disks in
which the pointed vertex $v^\star$ coincides with the root vertex
$\v(\rho)$. It is clear from
Definition~\ref{def:boundaryLabelCondition} that, for such pointed
disks, the label boundary condition is a DSF walk never visiting
negative positions. Thus, all hypermaps
contributing to $F_p^\circ$ (resp.\ $F_p^\bullet$) also contribute to
$F_p^{\blar}$ (resp.\ $F_p^{\nlar}$). The remaining contribution comes
from the pointed disks such that the pointed vertex $v^\star$
\emph{differs} from the root vertex $\v(\rho)$. We denote by
$F_p^{\bstr}$ and $F_p^{\nstr}$ the generating function of such
hypermaps with respectively a white or a black boundary of degree
$p$. By the above discussion we have
\begin{equation}
  \label{eq:Fplarstr}
  F_p^\circ = F_p^{\blar} - F_p^{\bstr} \qquad \text{and} \qquad
  F_p^\bullet = F_p^{\nlar} - F_p^{\nstr}.
\end{equation}
The interest of these seemingly tautological equalities lies in the following:

\begin{lemma} \label{lem:Fpstr}
  For any $p \geq 0$ we have
  \begin{equation} \label{eq:Fpstr}
    F_p^{\bstr}= a_{-1}\sum_{d \geq 2} t_d^\circ \sum_{h\geq 1}P^\circ_{d-1,h+1}P^{\blar}_{p,-h}
    \quad \text{and} \quad
    F_p^{\nstr}= \sum_{d\geq 2} t_d^\bullet \sum_{h\geq 1}P^\bullet_{d-1,h+1}P^{\nlar}_{p,-h}
  \end{equation}
  where $P^\circ_{\cdot,\cdot}$ and $P^\bullet_{\cdot,\cdot}$ are as in~\eqref{eq:pathIncBis}.
\end{lemma}

\begin{proof}
  \begin{figure}[t!]
    \centering
    \includegraphics[width=0.85\linewidth]{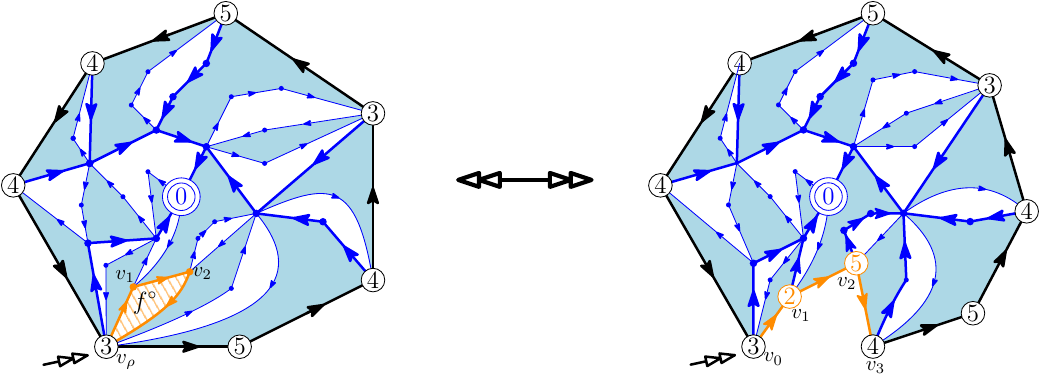}
    \caption{\label{fig:Rooted}Illustration of the proof of Lemma~\ref{lem:Fpstr}:
      starting from a pointed hypermap with a white monochromatic boundary of degree $p=7$ contributing to $F_p^{\bstr}$ (left), we cut at the root vertex $v_\rho$ in the way explained in the text, thereby merging the white boundary face with the inner white face $f^\circ$ of degree $d=3$, giving another pointed hypermap with a white monochromatic boundary of degree $p+d=10$ (right). 
    }
  \end{figure}
  Note first that, for $p=0$, both relations hold since everything vanishes.
  Let us prove the first relation for $p \geq 1$, the other is just obtained by
  exchanging the colors (recall that $b_{-1}=1$).  Let $(\rm,v^*)$ be a pointed disk
  contributing to $F_p^{\bstr}$, and write $v_\rho:=\v(\rho)$ for its root vertex. Define $e_1:=(v_\rho,v_1)$ as the first edge of the \emph{rightmost} geodesic from $\rho$ to $v^*$, and let $f^\circ$ be the (white) face on the right of $(v_\rho,v_1)$, see Figure~\ref{fig:Rooted}. Since $\pd(v_1,v^*)=\pd(v_\rho,v^*)-1$, so that $v_1$ cannot be incident to the boundary face of $\rm$, given its boundary condition. Hence, $f^\circ$ is an inner face and must have a degree $d$ at least $2$. We label its vertices clockwise as $v_\rho=v_0,v_1,\ldots,v_{d-1}$. 

We define a new hypermap $\tilde \rm$ from $\rm$ by cutting $v_\rho$ into two vertices $v_0$ and $v_d$. Edges incident to $v_\rho$ in $\rm$ become incident to either $v_0$ or $v_d$ according to the following rule: 
\begin{itemize}
 \item edges of $\rm$ situated between its root corner and $e_1$ (included) when turning clockwise around $v_\rho$ are incident to $v_0$ in $\tilde \rm$, 
 \item other edges are incident to $v_d$ in $\tilde \rm$. 
 \end{itemize}
 The corner of $\tilde \rm$ is defined as the image of $\rho$ that is incident to $v_0$.
 
 The above operation has the effect of merging $f^\circ$ with the boundary face of $\rm$, so that $\tilde \rm$ has a white boundary of degree $d+p$. We list its incident vertices as $v_0,v_1,\ldots,v_d,v_{d+1},\ldots,\allowbreak v_{d+p-1}$, starting from $v_0$ and following the orientation of the edges. 
 
We now characterize the label boundary condition of $\tilde \rm$. First observe that, by construction, the first edge of any geodesic between $v_\rho$ and $v^*$ in $\rm$ is incident to $v_0$ in $\tilde \rm$. Hence, we have: 
\begin{equation}
\pd_{\rm}(v_\rho,v^*)=\pd_{\tilde \rm}(v_0,v^\star)<\pd_{\tilde \rm}(v_d,v^\star).
\end{equation}
Endow $\tilde \rm$ with its canonical labeling $\ell$. The latter observation implies that $\ell(v_d)>0$ and that $\ell(v_i)\geq 0$, for any $i\in\{d+1,\ldots,d+p-1\}$. Moreover, since $(v_0,v_1)$ is the first edge of a geodesic from $v_0$ to $v^\star$, we necessarily have $\ell(v_1)=-1$.

To summarize, $\tilde \rm$ is a pointed hypermap with a monochromatic white boundary of degree $d+p$, whose weight satisfies $t^\circ_d \cdot  w(\tilde \rm)= t \cdot w(\rm)$, where the $t^\circ_d$ factor accounts for the weight of $f^\circ$ and the $t$ factor for the duplication of $v_\rho$. Moreover, its label boundary condition forms a DSF walk of $d+p$ steps which:
\begin{enumerate}[(a)]
\item starts with a down-step, \label{item:down}
\item arrives at a positive position after $d$ steps, \label{item:pos}
\item then stays at nonnegative positions after. \label{item:nn}
\end{enumerate}
\medskip

Reciprocally,
let $\bar \rm$ be a pointed rooted hypermap with a white boundary of length $p+d$, for some $p \geq 1$ and $d \geq 2$, whose label boundary condition satisfies the above constraints. 
Denote by $\bar u^*$ the pointed vertex and list the vertices incident to the boundary of $\bar \rm$ (again starting from the root corner and following the canonical orientation of its edges) as $\bar u_0,\ldots, \bar u_d,\bar u_{d+1},\ldots, \bar u_{p+d}$. Identify $u_0$ and $u_d$, and write $u_\rho$ for the resulting vertex and $\phi(\bar \rm)$ for the hypermap obtained\footnote{It might happen that the boundary of $\bar \rm$ is incident to several corners incident to $u_0$ or to $u_d$, i.e. there can exist $i\neq 0$ or $j\neq d$ such that $u_i=u_0$ or $u_j=u_d$ (or both). In that case, we identify $u_0$ and $u_d$ in such a way that the first corner and the $(d+1)$-th corner around the boundary of $\bar \rm$ are identified to produce $\phi(\bar \rm)$.}. It is easy to see that the resulting map is a planar hypermap with a white boundary. Indeed, the identification produces one new face, whose boundary is clearly directed. To prove that its label boundary condition is an arch of tilt $0$, it only remains to prove that:
\begin{equation}\label{eq:ToProve}
\pd_{\phi(\bar \rm)}(\bar u_i,\bar u^*)\geq \pd_{\phi(\bar \rm)}(\bar u_\rho,\bar u^*)\quad \text{ for any }i\in \{d+1,\ldots,d+p\}.
\end{equation}
Since $\bar \rm$ satisfies Property~\ref{item:pos}, we know that $\pd_{\phi(\bar \rm)}(\bar u_\rho,\bar u^*)=\pd_{\bar \rm}(\bar u_0,\bar u^*)$. Moreover, for $i \in \{d+1,\ldots,d+p\}$, consider a geodesic path $\gamma$ from $\bar u_i$ to $\bar u^*$ in $\phi(\bar \rm)$. If it passes through $u_0$, then~\eqref{eq:ToProve} holds since $\gamma$ is geodesic. Otherwise, $\gamma$ is actually a path in $\bar \rm$ which, by Property~\ref{item:nn}, has length at least $\pd_{\bar \rm}(\bar u_0,\bar u^*)=\pd_{\phi(\bar \rm)}(\bar u_\rho,\bar u^*)$, so that~\eqref{eq:ToProve} also holds.

This shows that we have a bijection between the pointed hypermaps
contributing to $F_p^{\bstr}$ and those whose labeled boundary
condition are DSF walks satisfying the constraints~\ref{item:down},
\ref{item:pos} and \ref{item:nn}, for some $d\geq 2$. The wanted expression~\eqref{eq:Fpstr} follows by
Propositions~\ref{prop:decPointedRooted} and~\ref{prop:decCompositePaths}, noting that a constrained DSF
walk can be decomposed into the concatenation of a down-step (contributing a weight
by $a_{-1}$), a walk with $d-1$ steps starting from position $-1$ and ending at
some position $h\geq 1$ (such walks are counted by $P^\circ_{d-1,h+1}$ upon shifting), and a walk with
$p$ steps starting from $h$, ending at $0$ and remaining nonnegative
(as counted by $P^{\blar}_{p,-h}$ upon shifting), and finally summing
over $d \geq 2$ and $h \geq 1$.
\end{proof}

Combining~\eqref{eq:Fplarstr} with Lemmas~\ref{lem:bijExcursion}
and~\ref{lem:Fpstr}, we find the expressions
\begin{equation} \label{eq:Fplar}
  \begin{split}
    F_p^\circ &= t P_{p,0}^{\blar} - a_{-1}\sum_{d \geq 2} t_d^\circ \sum_{h\geq 1}P^\circ_{d-1,h+1}P^{\blar}_{p,-h}, \\
  F_p^\bullet &= t P_{p,0}^{\nlar} - \sum_{d\geq 2} t_d^\bullet \sum_{h\geq 1}P^\bullet_{d-1,h+1}P^{\nlar}_{p,-h}.
  \end{split}
\end{equation}
Let us now explain why this expressions are equivalent to those
obtained in Section~\ref{sub:mono}. By the cycle lemma, already
mentioned in Appendix~\ref{sec:dsf}, we have for all $h \geq 0$
\begin{equation}
  a_{-1} P_{p,-h}^{\blar} = \frac{h+1}{p+1} P^\circ_{p+1,-(h+1)}, \qquad
  P_{p,-h}^{\nlar} = \frac{h+1}{p+1} P^\bullet_{p+1,-(h+1)},
\end{equation}
since appending a down-step to a DSF walk counted by
$P_{p,-h}^{\blar}$ or $P_{p,-h}^{\nlar}$ amounts to the concatenation
of $h+1$ excursions.  Thus, the above relations may be rewritten as
\begin{equation} \label{eq:Fpappendix}
  \begin{split}
    (p+1) F_p^\circ &= \frac{t}{a_{-1}} P^{\circ}_{p+1,-1}- \sum_{d \geq 2} t_d^\circ \sum_{h\geq 2} h P^\circ_{d-1,h}P^\circ_{p+1,-h}, \\
  (p+1) F_p^\bullet &= t P_{p+1,-1}^\bullet - \sum_{d\geq 2} t_d^\bullet \sum_{h\geq 2} h P^\bullet_{d-1,h}P^\bullet_{p+1,-h}.
  \end{split}
\end{equation}
where we have made a change of variable $h+1 \to h$ in the rightmost
sums. Now, using respectively Propositions~\ref{prop:alternativeDec}
and~\ref{prop:decslices}, we have
\begin{equation}
  \frac{t}{a_{-1}} = 1 - \sum_{d \geq 2} t_d^\circ P^\circ_{d-1,1}, \qquad
  t = a_{-1} - \sum_{d \geq 2} t_d^\bullet P^\bullet_{d-1,1}
\end{equation}
and substituting these relations in~\eqref{eq:Fpappendix} we get
\begin{equation} \label{eq:Fpappendixbis}
  \begin{split}
    (p+1) F_p^\circ &= P^{\circ}_{p+1,-1}- \sum_{d \geq 2} t_d^\circ \sum_{h\geq 1} h P^\circ_{d-1,h}P^\circ_{p+1,-h}, \\
  (p+1) F_p^\bullet &= a_{-1} P_{p+1,-1}^\bullet - \sum_{d\geq 2} t_d^\bullet \sum_{h\geq 1} h P^\bullet_{d-1,h}P^\bullet_{p+1,-h},
  \end{split}
\end{equation}
where the reader should notice that the sums over $h$ now start at
$1$.  Via~\eqref{eq:pathIncBis}, we get the same expressions as
in~\eqref{eq:Fplesscompact}. Let us mention that the relation
between~\eqref{eq:Fpappendix} and~\eqref{eq:Fpappendixbis} is very
similar to that between~\cite[Equations~(3.15) and (3.16)]{hankel}.

Let us conclude this appendix by remarking
that the expression given in~\cite[Theorem~3]{BoSc02} is actually a particular case of the
expression~\eqref{eq:Fplar} above. Indeed, taking $p=2$, and writing
$P_{2,0}^{\nlar}=b_0^2+b_1$, $P_{2,-1}^{\nlar}=2b_0$,
$P_{2,-2}^{\nlar}=1$, $P_{2,-h}^{\nlar}=0$ for $h \geq 3$, we get
\begin{equation}
  F_2^\bullet = t (b_0^2+b_1) - \sum_{d \geq 2} t_d^\bullet \left( 2 b_0 P^\bullet_{d-1,2} + P^\bullet_{d-1,3} \right).
\end{equation}
At $t=1$, we recover precisely the result of Bousquet-Mélou and
Schaeffer, using the correspondence of notation given at
Remark~\ref{rem:BoSc_connection}, and noting that the $B_i$'s and
$M(\boldsymbol{x},\boldsymbol{y})$ in their paper are respectively
equal to $\sum_{d \geq 2} t_d^\bullet P^\bullet_{d-1,i}$ and
$t_2^\bullet F_2^\bullet$ here.

\printbibliography

\end{document}